\documentclass[11pt]{article}
\usepackage{amssymb,amsmath,amsthm,amsfonts}
\usepackage{mathrsfs}
\usepackage{enumerate}

\usepackage[
bookmarks=true,         
bookmarksnumbered=true, 
colorlinks=true, pdfstartview=FitV, linkcolor=blue, citecolor=blue,
urlcolor=blue]{hyperref}

\usepackage{color}

\DeclareMathOperator*{\esssup}{ess\,sup}
\DeclareMathOperator*{\essinf}{ess\,inf}
\DeclareMathOperator*{\esslimsup}{ess\,lim\,sup}
\DeclareMathOperator*{\essliminf}{ess\,lim\,inf}
\newcommand{\dd}{d}

\newtheorem{lemma}{Lemma}[section]
\newtheorem{remark}{Remark}[section]
\newtheorem{example}{Example}[section]
\newtheorem{theorem}{Theorem}[section]

\newtheorem{definition}{Definition}[section]
\newtheorem{proposition}{Proposition}[section]

\newtheorem{assumption}{Assumption}[section]

\makeatletter
   
   \@addtoreset{equation}{section}
\makeatother

\begin{document}
\title{Stochastic differential games with random coefficients and state constraints: associated stochastic Hamilton--Jacobi--Bellman--Isaacs equations with Neumann boundary conditions}

\author{Rainer Buckdahn$^{1,2}$,\,\, Juan Li$^{3,2}$\thanks{Corresponding authors.},\,\, Yanwei Li$^4$\footnotemark[1]\\
	{\small $^1$Laboratoire de Math\'{e}matiques de Bretagne Atlantique, Univ Brest, }\\
	{\small UMR CNRS 6205, 6 avenue Le Gorgeu, 29200 Brest, France.}\\
	{\small $^2$Research Center for Mathematics and Interdisciplinary Sciences,}\\
	{\small Shandong University, Qingdao 266237, P.~R.~China.}\\
	{\small $^3$School of Mathematics and Statistics, Shandong University, Weihai 264209, P.~R.~China.}\\
	{\small $^4$School of Mathematical Sciences, Peking University, Beijing 100871, P.~R.~China.}\\
	{\small \textit{E-mails: rainer.buckdahn@univ-brest.fr; juanli@sdu.edu.cn; liyanwei@math.pku.edu.cn}}\\
}
\date{\today}
\maketitle
\vspace{-2em}

\textbf{Abstract}.
We study a stochastic differential game in which the state process is governed by a reflected stochastic differential equation (RSDE) and the recursive payoff is defined by a generalized backward stochastic differential equation (GBSDE). Unlike the Markovian models in \cite{Li08sicon,Fan}, all coefficients may depend directly on $\omega$, so the value functions are random fields. We show that the lower and upper value functions are viscosity solutions of stochastic Hamilton--Jacobi--Bellman--Isaacs (HJBI) equations with nonlinear Neumann boundary conditions, thereby extending the framework of \cite{Qiu22AAP} to reflected state constraints. Under weak boundary regularity, we derive estimates for RSDEs and GBSDEs, prove regularity of the value functions, and establish the dynamic programming principle (DPP) as well as the existence of viscosity solutions. Finally, under additional assumptions in a one-dimensional linear setting, we prove a comparison theorem and the uniqueness of bounded $(L_y,0)$-viscosity solutions, where $L_y$ is the constant coefficient of $y$ in the running generator. To prove the comparison theorem, we homogenize the boundary condition and construct smooth adapted comparison fields by solving deterministic uniformly parabolic Neumann problems in finite dimensions.


\section{Introduction}

In this paper, we study stochastic differential games with state constraints in a non-Markovian framework and derive the associated stochastic HJBI equations with nonlinear Neumann boundary conditions. Our main probabilistic tools are RSDEs and GBSDEs. RSDEs were introduced by Lions and Sznitman in \cite{Lions1984} and are closely related to the Skorokhod problem. They keep the state process in a prescribed domain $D$. Whenever the process reaches the boundary, a reflection term pushes it back into the domain. Further developments include \cite{Aida,Saisho}. The boundary local time makes the estimates delicate, and their derivation often requires strong smoothness assumptions on the boundary. One aim of this paper is to weaken these assumptions. GBSDEs were introduced by Pardoux and Zhang in \cite{Pardoux and Zhang} to give a probabilistic representation of partial differential equations (PDEs) with nonlinear Neumann boundary conditions. Unlike classical backward stochastic differential equations (BSDEs) of \cite{Pardoux and Peng}, a GBSDE contains an additional integral with respect to an increasing process, which gives rise to the nonlinear boundary term. Pardoux and Zhang proved the existence and the uniqueness of solutions to GBSDEs, established a comparison theorem, and obtained the corresponding nonlinear Feynman--Kac formula. Related work includes reflected GBSDEs and obstacle problems \cite{Ren06}, mean-field equations \cite{Feng}, and stochastic PDEs with Neumann boundary conditions \cite{Ren19}.

Buckdahn and Li \cite{Li08sicon} studied zero-sum recursive stochastic differential games with deterministic coefficients and proved that the value function is the unique viscosity solution of the associated Isaacs equation with a Cauchy boundary condition. They established the DPP, used invariance on the Cameron--Martin space to show that the value function is deterministic, and proved the existence and the uniqueness of the viscosity solution for the associated HJBI PDEs by BSDE methods. Subsequent works in the Markovian framework include, e.g., \cite{Fan}. For random coefficients, the study of stochastic optimal control goes back to Peng \cite{Peng1992}, who introduced a completely nonlinear stochastic Hamilton--Jacobi--Bellman (HJB) equation and studied classical solutions under strong regularity assumptions. Qiu \cite{Qiu18sicon} developed a viscosity solution theory for stochastic HJB equations, based on the random field decomposition introduced in \cite{33}, and considered nonrecursive control problems. The DPP leads to a backward stochastic partial differential equation (BSPDE). Applying the random field decomposition to this BSPDE yields a pathwise differential form suitable for viscosity methods. Later the work \cite{Qiu22AAP} treated recursive stochastic differential games with random coefficients and established the existence and the uniqueness of viscosity solutions for the associated stochastic HJBI equation, thereby extending \cite{Li08sicon} beyond the Markovian setting. For the classical foundations of viscosity solution theory, see \cite{Lions1992}; further viscosity theories for stochastic HJB equations and stochastic PDEs can be found in \cite{Pham and Zhang,Qiu19,Tang}.

In many economic and financial applications, the state process shall remain in a prescribed region, so a boundary condition is essential; see Achdou et al. \cite{ach}. Such constrained stochastic control and game problems have been studied extensively. In the Markovian setting, Li and Tang \cite{Li and Tang} studied recursive stochastic control, and Xiao et al. \cite{Fan} studied stochastic differential games. In the non-Markovian setting, Bayraktar and Qiu \cite{Qiu19AAP} considered weak and strong solutions of the stochastic HJB equation with a Neumann boundary condition in one dimension. The present work extends the non-Markovian analysis to stochastic differential games with reflected state constraints. Although a classical Neumann condition involves only first derivatives, earlier probabilistic studies often assumed a $C^2$ boundary or convexity to obtain regularity of the reflected process. We get estimates for the reflected state process on certain nonconvex and nonsmooth domains, while a stronger continuity estimate for the reflection process requires additional first-order boundary regularity.

Thus, the Neumann condition is not imposed independently of the game. It is generated jointly by the reflection term in the state equation and the boundary integral in the GBSDE. The same boundary geometry therefore enters the estimates for the reflected process, the DPP, and the viscosity inequalities. This coupling is one of the main differences from the unconstrained case and requires probabilistic and PDE arguments to be developed under a common set of geometric assumptions.

Let us now be more precise, we study a stochastic differential game whose state process is governed by the following controlled RSDE:
\begin{equation}\label{eq RSDE}
\left\{\begin{aligned}
\dd X_s^{t,x;\theta,\gamma}
={}&b\left(s,X_s^{t,x;\theta,\gamma},\theta_s,\gamma_s\right)\dd s+\sigma\left(s,X_s^{t,x;\theta,\gamma},\theta_s,\gamma_s\right)\dd W_s
+\vec n\left(X_s^{t,x;\theta,\gamma}\right)\dd K_s^{t,x;\theta,\gamma}, s\in[t,T];\\
\dd K_s^{t,x;\theta,\gamma}
={}&\mathbf{1}_{\partial D}\left(X_s^{t,x;\theta,\gamma}\right)\dd K_s^{t,x;\theta,\gamma},
\quad \text{$K^{t,x;\theta,\gamma}$ is nondecreasing},\quad t\leq s\leq T;\\
X_t^{t,x;\theta,\gamma}={}&x\in\overline D.
\end{aligned}\right.
\end{equation}
Here $\vec n$ denotes an inward normal, as specified below. Let $\Theta_0$ and $\Gamma_0$ be two nonempty compact metric spaces. Denote by $\Theta$ and $\Gamma$ the sets of all $\Theta_0$-valued and $\Gamma_0$-valued processes, respectively, adapted to the filtration generated by the driving Brownian motion. The controls are $\theta \in \Theta$ and $\gamma \in \Gamma$, and the superscripts in $(X^{t,x;\theta,\gamma},K^{t,x;\theta,\gamma})$ indicate the initial condition and the controls. The payoff to player I (and cost for player II) is defined by
\begin{equation*}
J(t,x;\theta,\gamma):=Y_t^{t,x;\theta,\gamma},
\end{equation*}
where $\left(Y^{t,x;\theta,\gamma}, Z^{t,x;\theta,\gamma}\right)$ is the solution of the GBSDE
\begin{equation}\label{1.3}
\left\{\begin{aligned}
-\dd Y_s^{t,x;\theta,\gamma}
={}&f\left(s,X_s^{t,x;\theta,\gamma},Y_s^{t,x;\theta,\gamma},Z_s^{t,x;\theta,\gamma},\theta_s,\gamma_s\right)\dd s\\
&\quad+g\left(s,X_s^{t,x;\theta,\gamma},Y_s^{t,x;\theta,\gamma},\theta_s,\gamma_s\right)\dd K_s^{t,x;\theta,\gamma}
-Z_s^{t,x;\theta,\gamma}\dd W_s,\quad s\in[t,T];\\
Y_T^{t,x;\theta,\gamma}
={}&\Phi\left(X_T^{t,x;\theta,\gamma}\right).
\end{aligned}\right.
\end{equation}
As in \cite{Li08sicon,Qiu22AAP}, we consider nonanticipative strategies (see Definition \ref{strategy}). We denote by $\mathcal{A}$ (resp. $\mathcal{M}$) the set of all nonanticipative strategies for player I (resp. player II), and we define the lower value function of our stochastic differential game by
\begin{equation}\label{lower function}
V(t, x)=\underset{\mu \in \mathcal{M}}{\essinf}\ \underset{\theta \in \Theta}{\esssup}\ J(t,x;\theta,\mu(\theta)), \quad t \in[0, T],
\end{equation}
and the upper value function is given by
\begin{equation}\label{upper function}
U(t, x)=\underset{\alpha \in \mathcal{A}}{\esssup}\ \underset{\gamma \in \Gamma}{\essinf}\ J(t,x;\alpha(\gamma),\gamma), \quad t \in[0, T].
\end{equation}

All coefficients in the problem are random and may depend directly on $\omega \in \Omega$. The DPP (see \cite{Peng1992}) shows that the value functions $V$ and $U$ are random fields and satisfy the stochastic HJBI equations
\begin{equation}\label{HJBI1}
\left\{\begin{array}{l}
-\dd V(t, x)=\mathbb{H}_{-}\left(t, x, D^2 V, D \psi, D V, V, \psi\right) \dd t-\psi(t, x) \dd W_t,\quad(t, x) \in Q:=[0, T) \times D; \\
\mathbb{G}_-\left(t, x,\partial_tV,\frac{\partial V}{\partial \vec n}, D^2 V, D \psi, D V, V, \psi\right) = 0,\quad x \in \partial D;\\
\mathbb{G}_+\left(t, x,\partial_tV,\frac{\partial V}{\partial \vec n}, D^2 V, D \psi, D V, V, \psi\right) = 0,\quad x \in \partial D;\\
	V(T, x)=\Phi(x), \quad x \in \overline{D},
\end{array}\right.
\end{equation}
and
\begin{equation}\label{HJBI2}
\left\{\begin{array}{l}
-\dd U(t, x)=\mathbb{H}_{+}\left(t, x, D^2 U, D \zeta, D U, U, \zeta\right) \dd t-\zeta(t, x) \dd W_t, \quad(t, x) \in Q; \\
	\mathbb{K}_-\left(t, x,\partial_tU,\frac{\partial U}{\partial \vec n}, D^2 U, D \zeta, D U, U, \zeta\right) = 0,\quad x \in \partial D;\\
	\mathbb{K}_+\left(t, x,\partial_tU,\frac{\partial U}{\partial \vec n}, D^2 U, D \zeta, D U, U, \zeta\right) = 0,\quad x \in \partial D;\\
	U(T, x)=\Phi(x), \quad x \in \overline{D},
\end{array}\right.
\end{equation}
respectively, with
\begin{align}
H(t, x, A, B, p, y, z,\theta,\gamma)&=\operatorname{tr}\left(\frac{1}{2} \sigma \sigma^{*}(t, x, \theta, \gamma) A+\sigma(t, x, \theta, \gamma) B^{*}\right)+b^{*}(t, x, \theta, \gamma) p\notag\\
&+f\left(t, x, y, z+p^{*} \sigma(t, x, \theta, \gamma), \theta, \gamma\right),\notag\\
\mathbb{H}_{-}(t, x, A, B, p, y, z)&= \underset{\theta \in \Theta_0}{\esssup}\ \underset{\gamma \in \Gamma_0}{\essinf}\ H(t, x, A, B, p, y, z,\theta,\gamma),\notag\\
\mathbb{H}_{+}(t, x, A, B, p, y, z)&= \underset{\gamma \in \Gamma_0}{\essinf}\ \underset{\theta \in \Theta_0}{\esssup}\ H(t, x, A, B, p, y, z,\theta,\gamma),\notag\displaybreak[3]\\
G_-(t,x,\alpha,\beta,A, B, p, y, z,\theta,\gamma)&=\left(\alpha+H(t, x, A, B, p, y, z,\theta,\gamma)\right)\vee\left(\beta+g(t,x,y,\theta,\gamma)\right),\notag\\
G_+(t,x,\alpha,\beta,A, B, p, y, z,\theta,\gamma)&=\left(\alpha+H(t, x, A, B, p, y, z,\theta,\gamma)\right)\wedge\left(\beta+g(t,x,y,\theta,\gamma)\right),\notag\\
\mathbb{G}_{-}(t,x,\alpha,\beta,A, B, p, y, z)&= \underset{\theta \in \Theta_0}{\esssup}\ \underset{\gamma \in \Gamma_0}{\essinf}\ G_-(t,x,\alpha,\beta,A, B, p, y, z,\theta,\gamma),\notag\\
\mathbb{G}_{+}(t,x,\alpha,\beta,A, B, p, y, z)&= \underset{\theta \in \Theta_0}{\esssup}\ \underset{\gamma \in \Gamma_0}{\essinf}\ G_+(t,x,\alpha,\beta,A, B, p, y, z,\theta,\gamma),\notag\\
\mathbb{K}_{-}(t,x,\alpha,\beta,A, B, p, y, z)&=\underset{\gamma \in \Gamma_0}{\essinf}\ \underset{\theta \in \Theta_0}{\esssup}\ G_-(t,x,\alpha,\beta,A, B, p, y, z,\theta,\gamma),\notag\\
\mathbb{K}_{+}(t,x,\alpha,\beta,A, B, p, y, z)&=\underset{\gamma \in \Gamma_0}{\essinf}\ \underset{\theta \in \Theta_0}{\esssup}\ G_+(t,x,\alpha,\beta,A, B, p, y, z,\theta,\gamma),\notag\\
\frac{\partial}{\partial \vec n}&:=\left\langle\vec n,\frac{\partial}{\partial x}\right\rangle.
\label{h-}
\end{align}
Here $(t,x,A,B,p,y,z)\in[0,T]\times\overline D\times\mathbb R^{d\times d}
\times\mathbb R^{d\times m}\times\mathbb R^d\times\mathbb R\times\mathbb R^m$,
and the pairs of random fields $(V,\psi)$ and $(U,\zeta)$ are unknown.

Following \cite{Qiu22AAP}, we use the following decomposition of $V(t,x)$:
\begin{equation*}
V(t, x)=V(T, x)-\int_t^T \mathfrak{d}_s V(s, x) \dd s-\int_t^T \mathfrak{d}_\omega V(s, x) \dd W_s, \quad(t, x) \in[0, T] \times D.
\end{equation*}
This decomposition defines the linear operators $\mathfrak d_t$ and $\mathfrak d_\omega$, whose uniqueness follows from the Doob--Meyer decomposition theorem. These operators coincide with those discussed in \cite{33}. The stochastic HJBI equations (\ref{HJBI1}) and (\ref{HJBI2}) may then be written equivalently as
\begin{equation}\label{HJBI3}
\left\{\begin{array}{l}
-\mathfrak{d}_t V(t, x)-\mathbb{H}_{-}\left(t, x, D^2 V(t, x), D \mathfrak{d}_\omega V(t, x), D V(t, x), V(t, x), \mathfrak{d}_\omega V(t, x)\right)=0,\quad(t, x) \in Q; \\
\mathbb{G}_-\left(t, x,\mathfrak{d}_t V(t, x),\frac{\partial V}{\partial \vec n}(t, x), D^2 V(t, x), D \mathfrak{d}_\omega V(t, x), D V(t, x), V(t, x), \mathfrak{d}_\omega V(t, x)\right) = 0,\quad x \in \partial D;\\
\mathbb{G}_+\left(t, x,\mathfrak{d}_t V(t, x),\frac{\partial V}{\partial \vec n}(t, x), D^2 V(t, x), D \mathfrak{d}_\omega V(t, x), D V(t, x), V(t, x), \mathfrak{d}_\omega V(t, x)\right) = 0,\quad x \in \partial D;\\
	V(T, x)=\Phi(x), \quad x \in \overline{D},
\end{array}\right.
\end{equation}
and
\begin{equation}\label{HJBI4}
\left\{\begin{array}{l}
-\mathfrak{d}_t U(t, x)-\mathbb{H}_{+}\left(t, x, D^2 U(t, x), D \mathfrak{d}_\omega U(t, x), D U(t, x), U(t, x), \mathfrak{d}_\omega U(t, x)\right)=0, \quad(t, x) \in Q; \\
\mathbb{K}_-\left(t, x,\mathfrak{d}_t U(t, x),\frac{\partial U}{\partial \vec n}(t, x), D^2 U(t, x), D \mathfrak{d}_\omega U(t, x), D U(t, x), U(t, x), \mathfrak{d}_\omega U(t, x)\right) = 0,\quad x \in \partial D;\\
\mathbb{K}_+\left(t, x,\mathfrak{d}_t U(t, x),\frac{\partial U}{\partial \vec n}(t, x), D^2 U(t, x), D \mathfrak{d}_\omega U(t, x), D U(t, x), U(t, x), \mathfrak{d}_\omega U(t, x)\right) = 0,\quad x \in \partial D;\\
	U(T, x)=\Phi(x), \quad x \in \overline{D},
\end{array}\right.
\end{equation}
respectively. Solving (\ref{HJBI1}) and (\ref{HJBI2}) for pairs $(V, \psi)$ and $(U, \zeta)$ is equivalent to finding $V$ and $U$ satisfying (\ref{HJBI3}) and (\ref{HJBI4}), respectively. While viscosity solutions in \cite{Qiu22AAP} are defined under Cauchy boundary conditions, we treat nonlinear Neumann boundary conditions by extending the construction of the nonlinear operators introduced in \cite{Qiu22AAP}.

In addition to the existence result, we establish a comparison theorem in a one-dimensional linear setting. The state equation is driven by a common noise with respect to which the coefficients are adapted and by an independent uniformly nondegenerate noise. After homogenizing the Neumann boundary condition, we first prove a weak comparison result between viscosity semisolutions and smooth adapted random fields. We then construct smooth upper and lower comparison fields by solving deterministic uniformly parabolic Neumann problems in finite dimensions. This gives the comparison theorem and the uniqueness of bounded $(L_y,0)$-viscosity solutions, where $L_y$ is the constant coefficient of $y$ in the running generator. These additional assumptions are used only in Section~\ref{sec:uniqueness}; the general existence theorem remains valid under Assumption \ref{HSDG}.

The paper is organized as follows. Section 2 introduces the function spaces and preliminary results. Section 3 studies RSDEs and the associated GBSDEs on domains with weak boundary regularity. Section 4 proves regularity of the value functions, establishes the DPP, and defines viscosity solutions of stochastic HJBI equations with nonlinear Neumann boundary conditions. Section 5 proves the existence result by combining BSDE techniques with the DPP. Section 6 proves the comparison theorem and the uniqueness of bounded $(L_y,0)$-viscosity solutions in a one-dimensional linear setting. Finally, Appendix \ref{sec:domain-discussion} discusses the domain assumptions, and Appendix \ref{app:DPP} proves the DPP.
\section{Preliminaries}

Throughout this paper, $T>0$ is a fixed finite time horizon, $(\Omega,\mathcal{F},P)$ is a complete probability space, and $W=\{W_t:t \in [0,T]\}$ is an $m$-dimensional standard Brownian motion on it. The filtration $\left\{{\mathcal{F}}_t\right\}_{t\in[0,T]}$ is the natural one generated by $W$ and augmented by the $P$-null sets in $\mathcal{F}$. The Euclidean norm of a vector $x\in\mathbb{R}^n$ and a matrix $z\in\mathbb{R}^{n\times m}$ is denoted by $|x|$ and $|z|=\sqrt{\operatorname{tr}(zz^*)}$, respectively, where $z^*$ is the transpose of $z$. We write $\left\langle x,y\right\rangle$ for the inner product of $x,y\in\mathbb{R}^n$ and $\langle A,B\rangle:=\operatorname{tr}(AB^*)$ for the Frobenius inner product of matrices of the same size.

The following spaces will be used throughout the paper:\\
- $\mathcal{T}_{\tau_1, \tau_2}$ represents the set of all $\left\{\mathcal{F}_t\right\}_{t\in[0,T]}$-stopping times with values in the stochastic interval $\left[\tau_1, \tau_2\right]$. We write $\mathcal{T}^{0}:=\mathcal{T}_{0,T}$ and, instead of $\mathcal{T}_{\tau,T}$, we write $\mathcal{T}^{\tau}$. We also let $\mathcal{T}_{+}^{\tau}$ denote the set of all $\left\{\mathcal{F}_t\right\}_{t\in[0,T]}$-stopping times with values in the stochastic interval $\left(\tau, T\right]$;\\
- $L^p\left(\mathcal{F}_\tau;\mathbb{R}^n\right)$ (where $\tau \in \mathcal{T}_{0, T}$) is the set of $\mathcal{F}_{\tau}$-measurable, $\mathbb{R}^n$-valued random variables $\xi$ such that $E\left[|\xi|^p\right]<\infty$;\\
- $\mathcal{S}^p\left(0,T;\mathbb{R}^n\right)$ (also denoted by $\mathcal{S}^p\left(\mathbb{R}^n\right)$ or $\mathcal{S}^p$) denotes the set of $\mathbb{R}^n$-valued, $\left\{\mathcal{F}_t\right\}_{t\in[0,T]}$-adapted continuous processes $\left(y_t\right)_{t\in[0,T]}$ such that $E\left[\sup\limits_{t\in[0,T]}\left|y_t\right|^p\right]<\infty$;\\
- Given a Banach space $\mathbb B$, $\mathcal{S}^{\infty}\left(0,T;\mathbb B\right)$ denotes the space of
$\mathbb B$-valued, adapted continuous processes $Y$ such that
$\|Y\|_{\mathcal S^\infty}:=
\|\sup\limits_{0\leq t\leq T}\|Y_t\|_{\mathbb B}\|_{L^\infty}<\infty$;\\
- For $p,q\geq1$, $L^{p,q}\left(0, T;\mathbb{R}^n\right)$ denotes the set of $\mathbb{R}^n$-valued and $\left\{\mathcal{F}_t\right\}_{t\in[0,T]}$-progressively measurable processes $\left(z_t\right)_{t \in[0, T]}$ such that $E\left[\left(\int_0^T\left|z_s\right|^q\dd s\right)^\frac{p}{q}\right]<\infty$; when $q=p$, we write $L^p\left(0, T;\mathbb{R}^n\right)$;\\
- The local spaces are defined by requiring the corresponding integrability conditions on every deterministic compact time interval \([0,T_0]\), with \(0<T_0<T\):
\begin{equation*}
\begin{aligned}
L^p_{\mathrm{loc}}(0,T;\mathbb R^n)
&:=\bigcap_{0<T_0<T}L^p(0,T_0;\mathbb R^n),\\
L^{p,q}_{\mathrm{loc}}(0,T;\mathbb R^n)
&:=\bigcap_{0<T_0<T}L^{p,q}(0,T_0;\mathbb R^n),\qquad
\mathcal S^p_{\mathrm{loc}}(0,T;\mathbb R^n)
:=\bigcap_{0<T_0<T}\mathcal S^p(0,T_0;\mathbb R^n);
\end{aligned}
\end{equation*}
- For $0\leq t\leq T$, $\mathcal{A}(t,T;\mathbb{R})$ represents the set of all real-valued, continuous, increasing, $\{\mathcal{F}_s\}_{s\in[t,T]}$-progressively measurable processes $(A_s)_{s\in[t,T]}$ such that $A_t=0$ and $E[e^{\lambda A_T}]<\infty$ for every $\lambda>0$;\\
- $C^k([0,T];\mathbb{R}^n)$ is the set of $k$-times continuously differentiable functions from $[0,T]$ to $\mathbb{R}^n$;\\
- $C^k_b([0,T];\mathbb{R}^n)$ is the set of functions in $C^k([0,T];\mathbb{R}^n)$ with bounded derivatives up to order $k$;\\
The analogous definitions apply when $\mathbb{R}^n$ (or $\mathbb{R}$) is replaced by a Banach space $\mathbb{B}$.

We next introduce the space of test functions used in \cite{Qiu22AAP}.
\begin{definition}\label{testfunction}
For $u \in \mathcal{S}^4\left(0, T;C_b^1\left(\mathbb{R}^d\right)\right)$ with $D u, D^2 u \in \mathcal{S}_{\mathrm{loc}}^4\left(0, T;C_b^1\left(\mathbb{R}^d\right)\right)$, we say $u \in \mathcal{C}_{\mathcal{F}}^3([0,T]\times\mathbb{R}^d)$ if there exists $\left(\mathfrak{d}_t u, \mathfrak{d}_\omega u\right) \in L_{\mathrm{loc}}^4\left(0, T;C_b^1\left(\mathbb{R}^d\right)\right) \times L_{\mathrm{loc}}^{4,2}\left(0, T;C_b^2\left(\mathbb{R}^d\right)\right)$ with $D \mathfrak{d}_t u$, $D \mathfrak{d}_\omega u \in \mathcal{S}_{\mathrm{loc}}^4(0, T;C_b^1\left(\mathbb{R}^d\right))$ such that $P$-a.s.,
\begin{equation*}
\begin{aligned}
u(r, x)={}&u\left(T_0, x\right)-\int_r^{T_0} \mathfrak{d}_s u(s, x) \dd s \\
&-\int_r^{T_0} \mathfrak{d}_\omega u(s, x) \dd W_s,\quad 0 \leq r \leq T_0<T,\quad x \in \mathbb{R}^d.
\end{aligned}
\end{equation*}
\end{definition}
\begingroup\emergencystretch=1em
For our random coefficients, we need to use the following It\^o--Wentzell formula; see
\cite[pp.~118--119]{Kunita} or \cite[Lemma~4.1]{Qiu18sicon}.\par
\endgroup
\begin{lemma}
Let $\phi\in\mathcal{C}_{\mathcal F}^3([0,T]\times\mathbb R^d)$ and fix $0<T_0<T$. Then
\begin{equation*}
\begin{aligned}
\phi(t,x)={}&\phi(T_0,x)-\int_t^{T_0}\mathfrak d_s\phi(s,x)\dd s-\int_t^{T_0}\mathfrak d_\omega\phi(s,x)\dd W_s,\quad t\in[0,T_0].
\end{aligned}
\end{equation*}
Put
\begin{equation*}
L\phi(t,x)=\mathfrak d_t\phi(t,x)+\operatorname{tr}\left\{\frac{1}{2}\sigma_t\sigma_t^*D^2\phi(t,x)+\sigma_tD\mathfrak d_\omega\phi(t,x)\right\}+D\phi(t,x)b_t,\quad t\in[0,T_0].
\end{equation*}
Then, for the It\^o process
\begin{equation*}
\left\{\begin{aligned}
\dd X_t={}&b_t\dd t+\sigma_t\dd W_t+\vec n_t\dd K_t,\\
X_0={}&x,
\end{aligned}\right.
\end{equation*}	
with progressively measurable processes $(b_t,\sigma_t,\vec n_t,K_t)_{t\in[0,T]}$ satisfying $(b_t)_{t\in[0,T]}\in L^2([0,T];\mathbb R^d)$, $(\sigma_t)_{t\in[0,T]}\in L_{\mathrm{loc}}^2([0,T];\mathbb R^{d\times m})$, $(K_t)_{t\in[0,T]}\in\mathcal A(0,T;\mathbb R)$, and $\vec n_t\in L^\infty([0,T];\mathbb R^d)$, we have
\begin{equation*}
\begin{aligned}
\phi(t,X_t)-\phi(0,x)
={}&\int_0^tL\phi(s,X_s)\dd s
+\int_0^t\left(\mathfrak d_\omega\phi(s,X_s)+D\phi(s,X_s)\sigma_s\right)\dd W_s\\
&+\int_0^tD\phi(s,X_s)\vec n_s\dd K_s,\quad t\in[0,T_0].
\end{aligned}
\end{equation*}
\end{lemma}
The following Kuratowski--Ryll-Nardzewski measurable selection theorem will be used in the proof of the existence result; see \cite{Qiu22AAP}.
\begin{theorem}\label{mst}
Let $(\Lambda,\mathcal M)$ be a measurable space equipped with a nonnegative measure $\mu$, and let $(\mathcal O,\mathcal B(\mathcal O))$ be a Polish space. Suppose that $F$ is a set-valued map from $\Lambda$ to $\mathcal B(\mathcal O)$ which is weakly measurable in the following sense: {\rm(i)} for $\mu$-a.e. $\lambda\in\Lambda$, $F(\lambda)$ is a nonempty closed subset of $\mathcal O$; {\rm(ii)} for every open set $O\subset\mathcal O$, $\{\lambda:F(\lambda)\cap O\neq\varnothing\}\in\mathcal M$. Then there exists a measurable map $f\colon(\Lambda,\mathcal M)\to(\mathcal O,\mathcal B(\mathcal O))$ such that $f(\lambda)\in F(\lambda)$ for $\mu$-a.e. $\lambda\in\Lambda$.
\end{theorem}
\section{Estimates for RSDEs and GBSDEs}
This section derives estimates for RSDEs and GBSDEs with random coefficients under weak boundary regularity. Let $D\subset\mathbb{R}^d$ be a bounded, connected, open set. For $x\in \partial D$ and $r>0$, put
$N_{x,r}:=\{\vec n\in\mathbb{R}^d:|\vec n|=1\ \text{and}\ B(x-r\vec n,r)\cap D=\varnothing\}$ and $N_x:=\bigcup_{r>0}N_{x,r}$, where $B(x,r)$ is the open ball with center $x$ and radius $r$. At a nonsmooth boundary point, $N_x$ does not need to be a singleton. Accordingly, the notation $\vec n(X_s)\dd K_s$ means that a progressively measurable selection $\vec n_s\in N_{X_s}$ is used $\dd K_s$-a.e.; when $N_x$ is a singleton, $\vec n(x)$ denotes its unique element.
\begin{assumption}\label{HD}
\textnormal{(i)} (Outer ball condition) There exists a constant $r_0>0$ such that $N_x=N_{x,r_0}\neq \varnothing$ for all $x\in\partial D$;\\
\textnormal{(ii)} There exist $\delta>0$ and $\beta\geq 1$ such that, for all $x\in \partial D$, there is a unit vector $l_x$ satisfying
\begin{equation*}
\left\langle l_x,\vec n \right\rangle\geq\frac{1}{\beta},
\qquad \text{for every }y\in B(x,\delta)\cap\partial D
\text{ and every }\vec n\in N_y;
\end{equation*}
\textnormal{(iii)} There exist a nonnegative function $f_D\in C^2_b(\mathbb{R}^d)$ and a constant $\gamma_D >0$ such that, with the same $r_0$ as in (i), for all $x\in\partial D$ and all $\vec n\in N_x$,
\begin{equation*}
\frac{\gamma_D}{r_0}\leq\left\langle D f_D(x),\vec n\right\rangle\leq \frac{2\gamma_D}{r_0}.
\end{equation*}
\end{assumption}
\begin{assumption}\label{HD'}
There exists a function $\phi_D\in C_{\mathrm{Lip}}^1(\mathbb{R}^d)$ such that $D=\{x:\phi_D(x)>0\}$, $\partial D=\{x:\phi_D(x)=0\}$ and $|D\phi_D(x)|=1$ on $\partial D$.
\end{assumption}

Appendix \ref{sec:domain-discussion} shows that Assumption \ref{HD'} implies Assumption \ref{HD}. It also gives two examples showing that Assumption \ref{HD} allows nonsmooth or nonconvex boundaries.

Consider the following RSDE:
\begin{equation}\label{RSDE}
\left\{\begin{array}{l}
\dd X_s^{t,\xi}=b\left(s, X_s^{t,\xi}\right) \dd s+\sigma\left(s, X_s^{t,\xi}\right) \dd W_s+\vec n(X_s^{t,\xi})\dd K_s^{t,\xi}, \quad t \leq s \leq T;\\
\dd K_s^{t,\xi}=\mathbf{1}_{\partial D}(X_s^{t,\xi})\dd K_s^{t,\xi}, \quad \text{$K^{t,\xi}$ is nondecreasing},\quad t \leq s \leq T;\\
X_t^{t,\xi}=\xi \in \overline{D},
\end{array}\right.
\end{equation}
Here $t\in\mathcal T^0$, $\xi\in L^{2p}(\mathcal{F}_t;\overline D)$, and $p\geq2$ are given. The coefficients $b\colon\Omega\times[0,T]\times\overline D\to\mathbb R^d$ and $\sigma\colon\Omega\times[0,T]\times\overline D\to\mathbb R^{d\times m}$ are measurable. For every $x\in\overline D$, the processes $b(\cdot,\cdot,x)$ and $\sigma(\cdot,\cdot,x)$ are progressively measurable. We impose the following assumption.
\begin{assumption}\label{HRSDE}
\begin{enumerate}[(i)]
\item The coefficients $b$ and $\sigma$ are uniformly bounded and, for every
$t\in[0,T]$, $b(t,\cdot)$ is continuous on $\overline D$, $P$-a.s.
\item There exists $K\geq0$ such that, for $\dd t\otimes\dd P$-almost every
$(t,\omega)$ and all $x_1,x_2\in\overline D$,
\begin{equation*}
\left\langle x_1-x_2,b(t,x_1)-b(t,x_2)\right\rangle\leq K|x_1-x_2|^2,
\end{equation*}
\begin{equation*}
|\sigma(t,x_1)-\sigma(t,x_2)|\leq K|x_1-x_2|.
\end{equation*}
\end{enumerate}
\end{assumption}
\begingroup\emergencystretch=3em
Under the Assumptions \ref{HD} and \ref{HRSDE}, RSDE (\ref{RSDE}) has a unique
solution
$(X^{t,\xi},K^{t,\xi})\in
\mathcal{S}^{2p}(t,T;\overline{D})\times\mathcal{A}(t,T;\mathbb R)$.
See \cite{Lions1984,Saisho} for the proof. The following estimates also hold.\par
\endgroup
\begin{proposition}\label{estimate RSDE}
Suppose the Assumptions \ref{HD} and \ref{HRSDE} hold. Then, for all $p\geq2$, $t,s\in\mathcal T^0$ with $t\leq s$, $\xi\in L^{2p}(\mathcal{F}_t;\overline{D})$, and $\xi'\in L^{2p}(\mathcal{F}_s;\overline{D})$, there exists $C_p\geq 0$, depending on $K,T,b,\sigma,D,p$, such that $P$-a.s.,
\begin{equation*}
E\left[\sup\limits_{r\in[t,T]} |X_r^{t,\xi}-X_r^{s,\xi'}|^{2p}\Big|\mathcal{F}_t\right]\leq C_p\left(E\left[|\xi-\xi'|^{2p}\big|\mathcal{F}_t\right]+E[|s-t|^p\mid\mathcal F_t]\right),
\end{equation*}
\begin{equation*}
E\left[|K_s^{t,\xi}|^{2p}\Big|\mathcal{F}_t\right]\leq C_pE[|s-t|^p\mid\mathcal F_t],
\end{equation*}
\begin{equation}\label{*3}
E\left[\sup\limits_{r\in[t,T]} |K_r^{t,\xi}|^{2p}\Big|\mathcal{F}_t\right]\leq C_p,
\end{equation}
where we put $X_r^{s,\xi'}=\xi'$ and $K_r^{s,\xi'}=0$, for $r\leq s$.
Moreover, if also Assumption \ref{HD'} holds, then $P$-a.s.,
\begin{equation*}
E\left[\sup\limits_{r\in[t,T]} |K_r^{t,\xi}-K_r^{s,\xi'}|^{2p}\Big|\mathcal{F}_t\right]\leq C_p\left(E\left[|\xi-\xi'|^{2p}\big|\mathcal{F}_t\right]+E[|s-t|^p\mid\mathcal F_t]\right),
\end{equation*}
and, for every $\lambda>0$, there exists a constant $C_{\lambda}>0$ such that $P$-a.s.,
\begin{equation}\label{*4}
E\left[e^{\lambda K_T^{t,\xi}}\Big|\mathcal{F}_t\right]\leq C_{\lambda}.
\end{equation}
\end{proposition}
\begin{proof}
First, by (\ref{RSDE}), if $r\geq s\geq t$, we have
\begin{equation*}
\begin{split}
X_r^{t,\xi}-X_r^{s,\xi'}={}&X_s^{t,\xi}-\xi'+\int_s^r\left[ b\left(l,X_l^{t,\xi}\right)-b\left(l,X_l^{s,\xi'}\right)\right]\dd l+\int_s^r \left[\sigma\left(l,X_l^{t,\xi}\right)-\sigma\left(l,X_l^{s,\xi'}\right)\right]\dd W_l\\
&+\int_s^r \vec n(X_l^{t,\xi})\dd K_l^{t,\xi}-\int_s^r \vec n(X_l^{s,\xi'})\dd K_l^{s,\xi'}.
\end{split}
\end{equation*}
For simplicity, we set
\begin{equation*}
\begin{gathered}
 X_l:=X_l^{t,\xi},\qquad X'_l:=X_l^{s,\xi'},\\
 b_l:=b(l,X_l),\quad b'_l:=b(l,X'_l),\quad
 \sigma_l:=\sigma(l,X_l),\quad \sigma'_l:=\sigma(l,X'_l),\\
 R_l:=\exp\left\{-\frac{f_D(X_l)+f_D(X'_l)}{\gamma_D}\right\},\qquad
 q_l:=D f_D(X_l)\sigma_l+D f_D(X'_l)\sigma'_l,
\end{gathered}
\end{equation*}
Using $f_D$ and $\gamma_D$ from Assumption \ref{HD}, we apply It\^o's formula to
$R_l|X_l-X'_l|^2$ over the interval $[s,r]$. We also put
\begin{equation*}
\begin{split}
 a_l={}&\langle D f_D(X_l),b_l\rangle+\langle D f_D(X'_l),b'_l\rangle
 +\frac{1}{2}\operatorname{tr}\left(D^2 f_D(X_l)\sigma_l\sigma_l^*
 +D^2 f_D(X'_l)\sigma'_l{\sigma'_l}^*\right).
\end{split}
\end{equation*}
Then
\begin{align}
 R_r|X_r-X'_r|^2
={}&R_s|X_s-X'_s|^2\notag\\
&+\int_s^rR_l\bigg[
 2\langle X_l-X'_l,b_l-b'_l\rangle+|\sigma_l-\sigma'_l|^2
 -\frac{|X_l-X'_l|^2}{\gamma_D}a_l\notag\\
 &\qquad\qquad+\frac{|X_l-X'_l|^2}{2\gamma_D^2}|q_l|^2
 -\frac{2}{\gamma_D}\left\langle q_l,
 (X_l-X'_l)^*(\sigma_l-\sigma'_l)\right\rangle\bigg]\dd l\notag\\
&+\int_s^rR_l\left[
 2(X_l-X'_l)^*(\sigma_l-\sigma'_l)
 -\frac{|X_l-X'_l|^2}{\gamma_D}q_l\right]\dd W_l\notag\displaybreak[2]\\
&+\int_s^rR_l\left[
 2\langle X_l-X'_l,\vec n(X_l)\rangle
 -\frac{|X_l-X'_l|^2}{\gamma_D}
  \langle D f_D(X_l),\vec n(X_l)\rangle\right]\dd K_l^{t,\xi}\notag\\
&+\int_s^rR_l\left[
 2\langle X'_l-X_l,\vec n(X'_l)\rangle
 -\frac{|X_l-X'_l|^2}{\gamma_D}
  \langle D f_D(X'_l),\vec n(X'_l)\rangle\right]\dd K_l^{s,\xi'}.
  \label{ito}
\end{align}
Remark \ref{HDD} in Appendix \ref{sec:domain-discussion} shows that the latter two integrals are nonpositive. Moreover, using the uniform boundedness assumptions, the one-sided Lipschitz condition for $b$, and the Lipschitz continuity of $\sigma$ in Assumption~\ref{HRSDE}, we obtain
\begin{equation}\label{*6}
\begin{split}
&|X_r^{t,\xi}-X_r^{s,\xi'}|^2 \\
\leq &C \bigg\{ |X_s^{t,\xi}-\xi'|^2+\int_s^r |X_l^{t,\xi}-X_l^{s,\xi'}|^2 \dd l\bigg\}\\
&-\frac{1}{\gamma_D}\int_s^r R_l|X_l-X'_l|^2q_l\,\dd W_l
+2\int_s^r R_l(X_l-X'_l)^*(\sigma_l-\sigma'_l)\,\dd W_l.
\end{split}
\end{equation}
For any $p\geq2$, the Burkholder--Davis--Gundy and H\"{o}lder inequalities yield the following estimate, $P$-a.s.:
\begin{equation*}
\begin{aligned}
&E\left[\sup\limits_{r\in[s,v]}|X_r-X'_r|^{2p}\mid\mathcal F_t\right]\\
\leq &C_p\bigg\{E\left[|X_s^{t,\xi}-\xi'|^{2p}\mid\mathcal F_t\right]
+E\left[\int_s^v|X_l-X'_l|^{2p}\dd l\mid\mathcal F_t\right]+E\left[\left(\int_s^v|X_l-X'_l|^4\dd l\right)^{p/2}
\bigg|\mathcal F_t\right]\bigg\}\\
\leq &C_p\bigg\{E\left[|X_s^{t,\xi}-\xi'|^{2p}\mid\mathcal F_t\right]
+E\left[\int_s^v|X_l-X'_l|^{2p}\dd l\mid\mathcal F_t\right]\bigg\},
\qquad s\leq v\leq T.
\end{aligned}
\end{equation*}
Writing the stopped integrals on $[0,T]$ with their indicator functions,
Gronwall's inequality then gives
\begin{equation}\label{estimate2}
E[\sup\limits_{r\in[s,T]}|X_r^{t,\xi}-X_r^{s,\xi'}|^{2p}|\mathcal{F}_t]\leq CE[|X_s^{t,\xi}-\xi'|^{2p}|\mathcal{F}_t].
\end{equation}
On the other hand,
\begin{equation}\label{estimate3}
\begin{aligned}
&E\left[\sup\limits_{r\in[t,s]}|X_r^{t,\xi}-X_r^{s,\xi'}|^{2p}\bigg|\mathcal{F}_t\right]
={}E\left[\sup\limits_{r\in[t,s]}|X_r^{t,\xi}-\xi'|^{2p}\bigg|\mathcal{F}_t\right]\\
\leq &C_pE\left[\sup\limits_{r\in[t,s]}|X_r^{t,\xi}-\xi|^{2p}\bigg|\mathcal{F}_t\right]
+C_pE\left[|\xi-\xi'|^{2p}\bigg|\mathcal{F}_t\right].
\end{aligned}
\end{equation}
We next estimate $|X_r^{t,\xi}-\xi|^{2p}$. As in (\ref{ito}) and
(\ref{*6}), we apply It\^o's formula to
\begin{equation*}
 \exp\left\{-\frac{1}{\gamma_D}f_D\left(X_r^{t,\xi}\right)\right\}
 |X_r^{t,\xi}-\xi|^2
\end{equation*}
over the interval $[t,r]$. This gives
\begin{equation*}
\begin{aligned}
&|X_r^{t,\xi}-\xi|^{2}\\
\leq &C\bigg\{\int_t^r \left(1+|X_l^{t,\xi}-\xi|+|X_l^{t,\xi}-\xi|^{2}\right)\dd l\\
&\quad-\int_t^r \frac{1}{\gamma_D}\exp\left\{-\frac{1}{\gamma_D}f_D(X_l^{t,\xi})\right\}
D f_D(X_l^{t,\xi})|X_l^{t,\xi}-\xi|^2\sigma(l,X_l^{t,\xi})\dd W_l\\
&\quad+\int_t^r 2\exp\left\{-\frac{1}{\gamma_D}f_D(X_l^{t,\xi})\right\}
(X_l^{t,\xi}-\xi)^*\sigma(l,X_l^{t,\xi})\dd W_l\bigg\}.
\end{aligned}
\end{equation*}
Since $D$ is bounded, the following estimate holds, $P$-a.s.:
\begin{equation}\label{estimate5}
E\left[\sup\limits_{r\in[t,s]}|X_r^{t,\xi}-\xi|^{2p}\bigg|\mathcal{F}_t\right]\leq C_pE[|s-t|^p\mid\mathcal F_t].
\end{equation}
Combining (\ref{estimate2}), (\ref{estimate3}) and (\ref{estimate5}), we see that
\begin{equation*}
E\left[\sup\limits_{r\in[t,T]} |X_r^{t,\xi}-X_r^{s,\xi'}|^{2p}|\mathcal{F}_t\right]\leq C_p\left(E\left[|\xi-\xi'|^{2p}\big|\mathcal{F}_t\right]+E[|s-t|^p\mid\mathcal F_t]\right).
\end{equation*}
Next, applying It\^o's formula to $f_D(X_r^{t,\xi})$ yields
\begin{equation}\label{itok}
\begin{split}
f_D(X_r^{t,\xi})={}&f_D(\xi)+\int_t^r \left(D f_D(X_l^{t,\xi})b(l,X_l^{t,\xi})+\frac{1}{2}\operatorname{tr}\{D^2 f_D(X_l^{t,\xi})\sigma\sigma^*(l,X_l^{t,\xi})\}\right)\dd l\\
&\quad+\int_t^rD f_D(X_l^{t,\xi})\sigma(l,X_l^{t,\xi})\dd W_l+\int_t^r D f_D(X_l^{t,\xi})\vec n(X_l^{t,\xi})\dd K_l^{t,\xi}.
\end{split}
\end{equation}
By Assumption \ref{HD}(iii),
\begin{equation*}
\begin{aligned}
K_r^{t,\xi}\leq &C\bigg\{f_D(X_r^{t,\xi})-f_D(\xi)
-\int_t^r \left(D f_D(X_l^{t,\xi})b(l,X_l^{t,\xi})
+\frac{1}{2}\operatorname{tr}\{D^2 f_D(X_l^{t,\xi})\sigma\sigma^*(l,X_l^{t,\xi})\}\right)\dd l\\
&\quad-\int_t^rD f_D(X_l^{t,\xi})\sigma(l,X_l^{t,\xi})\dd W_l\bigg\}.
\end{aligned}
\end{equation*}
Then the Burkholder--Davis--Gundy inequality, together with (\ref{estimate5}),
gives
\begin{equation*}
\begin{split}
E\left[|K_s^{t,\xi}|^{2p}|\mathcal{F}_t\right]\leq CE\left[\sup\limits_{r\in[t,s]}|X_r^{t,\xi}-\xi|^{2p}\bigg|\mathcal{F}_t\right]+CE[|s-t|^p\mid\mathcal F_t]\leq CE[|s-t|^p\mid\mathcal F_t].
\end{split}
\end{equation*}
If $D$ also satisfies Assumption \ref{HD'}, Appendix \ref{sec:domain-discussion} gives $\vec n(x)=D\phi_D(x)$, where $\phi_D$ is from Assumption \ref{HD'} and satisfies $|D\phi_D(x)|=1$ for all $x\in\partial D$. The function $f_D$ can then be replaced by $\phi_D^\varepsilon\in C^{2}_b(\mathbb{R}^d)$, a standard mollification of $\phi_D$. Then (\ref{itok}) becomes
\begin{equation*}
\begin{split}
\phi_D^\varepsilon(X_r^{t,\xi})={}&\phi_D^\varepsilon(\xi)+\int_t^r \left(D\phi_D^\varepsilon(X_l^{t,\xi})b(l,X_l^{t,\xi})+\frac{1}{2}\operatorname{tr}\{D^2\phi_D^\varepsilon(X_l^{t,\xi})\sigma\sigma^*(l,X_l^{t,\xi})\}\right)\dd l\\
&\quad+\int_t^rD\phi_D^\varepsilon(X_l^{t,\xi})\sigma(l,X_l^{t,\xi})\dd W_l+\int_t^r \left\langle D\phi_D^\varepsilon(X_l^{t,\xi}),D\phi_D(X_l^{t,\xi})\right\rangle \dd K_l^{t,\xi}.
\end{split}
\end{equation*}
Consequently,
\begin{equation}\label{itok3}
\begin{split}
&\int_t^r \left\langle D\phi_D^\varepsilon(X_l^{t,\xi}),D\phi_D(X_l^{t,\xi})\right\rangle \dd K_l^{t,\xi}\\
={}&\phi_D^\varepsilon(X_r^{t,\xi})-\phi_D^\varepsilon(\xi)-\int_t^r \left(D\phi_D^\varepsilon(X_l^{t,\xi})b(l,X_l^{t,\xi})+\frac{1}{2}\operatorname{tr}\{D^2\phi_D^\varepsilon(X_l^{t,\xi})\sigma\sigma^*(l,X_l^{t,\xi})\}\right)\dd l\\
&-\int_t^rD\phi_D^\varepsilon(X_l^{t,\xi})\sigma(l,X_l^{t,\xi})\dd W_l.
\end{split}
\end{equation}
Then an argument analogous to \cite[Theorem~3.2]{Feng} or
\cite[Proposition~3.2]{Pardoux and Zhang} yields
\begin{equation*}
\begin{aligned}
&E\left[\sup\limits_{r\in[t,T]} |\int_t^r \left\langle D\phi_D^\varepsilon(X_l^{t,\xi}),D\phi_D(X_l^{t,\xi})\right\rangle \dd K_l^{t,\xi}-\int_t^r \left\langle D\phi_D^\varepsilon(X_l^{s,\xi'}),D\phi_D(X_l^{s,\xi'})\right\rangle \dd K_l^{s,\xi'}|^{2p}\Big|\mathcal{F}_t\right]\\
\leq &C\left(E\left[|\xi-\xi'|^{2p}\big|\mathcal{F}_t\right]+E[|s-t|^p\mid\mathcal F_t]\right).
\end{aligned}
\end{equation*}
Since, on $\partial D$, $D\phi_D^\varepsilon$ converges uniformly to $D\phi_D$ and $|D\phi_D(x)|^2=1$, the integrands converge uniformly to 1 on the support of $\dd K$. Indeed, the corresponding error up to time $r$ is bounded by $\|D\phi_D^\varepsilon-D\phi_D\|_\infty K_r$. Therefore, the preceding moment bounds allow to take $\varepsilon\downarrow0$, giving
\begin{equation}\label{kk}
\begin{aligned}
E\left[\sup\limits_{r\in[t,T]} |K_r^{t,\xi}-K_r^{s,\xi'}|^{2p}\Big|\mathcal{F}_t\right]\leq C\left(E\left[|\xi-\xi'|^{2p}\big|\mathcal{F}_t\right]+E[|s-t|^p\mid\mathcal F_t]\right).
\end{aligned}
\end{equation}
Let $G^{\varepsilon}_r=\int_t^r \left\langle D\phi_D^\varepsilon(X_l^{t,\xi}),D\phi_D(X_l^{t,\xi})\right\rangle \dd K_l^{t,\xi}-\phi_D^\varepsilon(X_r^{t,\xi})$, $r\in[t,T]$. Then, for $\lambda>0$, applying It\^o's formula to $e^{\frac{\lambda}{2}G^{\varepsilon}_r}$ yields
\begin{equation*}
\begin{split}
e^{\frac{\lambda}{2}G^{\varepsilon}_s}={}&e^{-\frac{\lambda}{2}\phi_D^\varepsilon(\xi)}-\frac{\lambda}{2}\int_t^s e^{\frac{\lambda}{2}G^{\varepsilon}_l}\left(D\phi_D^\varepsilon(X_l^{t,\xi})b(l,X_l^{t,\xi})
+\frac{1}{2}\operatorname{tr}\{D^2\phi_D^\varepsilon(X_l^{t,\xi})\sigma\sigma^*(l,X_l^{t,\xi})\}\right)\dd l\\
&\quad-\frac{\lambda}{2}\int_t^s e^{\frac{\lambda}{2}G^{\varepsilon}_l}D\phi_D^\varepsilon(X_l^{t,\xi})\sigma(l,X_l^{t,\xi})\dd W_l\\
&\quad+\frac{\lambda^2}{8}\int_t^s e^{\frac{\lambda}{2}G^{\varepsilon}_l}|D\phi_D^\varepsilon(X_l^{t,\xi})\sigma(l,X_l^{t,\xi})|^2\dd l,\quad s\in[t,T].
\end{split}
\end{equation*}
Thus, by the Burkholder--Davis--Gundy inequality,
\begin{equation*}
E[e^{\lambda G^{\varepsilon}_r}|\mathcal{F}_t]\leq C_\lambda E\left[\int_t^r e^{\lambda G^{\varepsilon}_l}\dd l\bigg|\mathcal{F}_t\right]+C_\lambda,\quad r\in[t,T].
\end{equation*}
The uniform convergence of $\phi_D^\varepsilon$ and $D\phi_D^\varepsilon$, together with the boundedness of $D$, allows to pass to the limit in the preceding estimate, and from Gronwall's inequality it follows
\begin{equation*}
E\left[e^{\lambda K_r^{t,\xi}}\Big|\mathcal{F}_t\right]\leq C_{\lambda},\qquad r\in[t,T].
\end{equation*}
This completes the proof.
\end{proof}
\begin{remark}
When $D$ does not satisfy Assumption \ref{HD'}, we retain continuous dependence of $X^{t,\xi}$ on $(t,\xi)$ and of $K$ on $t$, but estimate (\ref{kk}) does not need to hold. Indeed, if only $C_1 \leq \left\langle D f_D(x),\vec n(x)\right\rangle\leq C_2$ is assumed for some $C_1,C_2>0$, then (\ref{itok}) yields only
\begingroup\medmuskip=3mu
\begin{equation*}
\begin{aligned}
K_r^{t,\xi}\leq &C_1^{-1}\bigg\{f_D(X_r^{t,\xi})-f_D(\xi)
-\int_t^r \left(D f_D(X_l^{t,\xi})b(l,X_l^{t,\xi})
+\frac{1}{2}\operatorname{tr}\{D^2 f_D(X_l^{t,\xi})\sigma\sigma^*(l,X_l^{t,\xi})\}\right)\dd l\\
&\quad-\int_t^rD f_D(X_l^{t,\xi})\sigma(l,X_l^{t,\xi})\dd W_l\bigg\},\\
K_r^{s,\xi'}\geq &C_2^{-1}\bigg\{f_D(X_r^{s,\xi'})-f_D(\xi')
-\int_s^r \left(D f_D(X_l^{s,\xi'})b(l,X_l^{s,\xi'})
+\frac{1}{2}\operatorname{tr}\{D^2 f_D(X_l^{s,\xi'})\sigma\sigma^*(l,X_l^{s,\xi'})\}\right)\dd l\\
&\quad-\int_s^rD f_D(X_l^{s,\xi'})\sigma(l,X_l^{s,\xi'})\dd W_l\bigg\}.
\end{aligned}
\end{equation*}
\endgroup
For simplicity, denote by $A_r^{t,\xi}$ and $A_r^{s,\xi'}$, respectively, the expressions in braces on the right-hand sides above, and take the difference between $K_r^{t,\xi}$ and $K_r^{s,\xi'}$. This gives
\begin{equation}\label{3.23}
K_r^{t,\xi}-K_r^{s,\xi'}\leq C_1^{-1} A_r^{t,\xi}-C_2^{-1}A_r^{s,\xi'}=C_1^{-1}\left(A_r^{t,\xi}-A_r^{s,\xi'}\right)+\left(C_1^{-1}-C_2^{-1}\right)A_r^{s,\xi'}.
\end{equation}
However, for $0\leq t\leq s\leq T$, the conditional expectation $E[\cdot\mid\mathcal{F}_t]$ of the right-hand side of (\ref{3.23}) cannot be controlled by $E[|\xi-\xi'|\mid\mathcal{F}_t]+|s-t|$ but only by $E[|\xi-\xi'|\mid\mathcal{F}_t]+|T-s|$. The reason is that, without Assumption \ref{HD'}, the normal map is set-valued and not necessarily continuous.
\end{remark}
We will use the following stability estimate in Section 4.
\begin{proposition}\label{SRSDE}
Let $(X_t^{i,s,x},K_t^{i,s,x})_{t\in[s,T]}$ be the solution of RSDE (\ref{RSDE}) with coefficients $(b^i,\sigma^i)$, $i=1,2$, and initial data $(s,x)$. Here $b^i,\sigma^i$ satisfy Assumption \ref{HRSDE}, and $D$ fulfills the Assumptions \ref{HD} and \ref{HD'}. Putting $\overline{b}_t=\sup\limits_{x\in\overline{D}}|b^1(t,x)-b^2(t,x)|$ and $\overline{\sigma}_t=\sup\limits_{x\in\overline{D}}|\sigma^1(t,x)-\sigma^2(t,x)|$, we have, for all $p\geq2$,
\begin{equation*}
E\left[\sup\limits_{s\leq r\leq T}|X_r^{1,s,x}-X_r^{2,s,x}|^{2p}+ \sup\limits_{s\leq r\leq T}|K_r^{1,s,x}-K_r^{2,s,x}|^{2p}\bigg|\mathcal{F}_s\right]\leq C E\left[\int_s^T (|\overline{b}_r|^{2p}+|\overline{\sigma}_r|^{2p})\dd r\bigg|\mathcal{F}_s\right].
\end{equation*}
\end{proposition}
\begin{proof}
For simplicity, we put $\overline{X}_t^{s,x}=X_t^{1,s,x}-X_t^{2,s,x}$ and $\overline{K}_t^{s,x}=K_t^{1,s,x}-K_t^{2,s,x}$. We apply It\^o's formula to $\exp\left\{-\frac{1}{\gamma_D}\left[f_D(X_t^{1,s,x})+f_D(X_t^{2,s,x})\right]\right\}|\overline{X}_t^{s,x}|^2$ on the interval $[s,T]$ and use the following inequalities:
\begin{equation*}
\begin{aligned}
&\left\langle x_1-x_2,b^1(t,x_1)-b^2(t,x_2)\right\rangle\\
={}&\left\langle x_1-x_2,b^1(t,x_1)-b^1(t,x_2)\right\rangle
+\left\langle x_1-x_2,b^1(t,x_2)-b^2(t,x_2)\right\rangle\\
\leq &C|x_1-x_2|^2+|\overline{b}_t|^2,\\
&|\sigma^1(t,x_1)-\sigma^2(t,x_2)|
\leq K|x_1-x_2|+|\overline{\sigma}_t|.
\end{aligned}
\end{equation*}
For convenience, we set
\begin{equation*}
\begin{aligned}
R_r&:=\exp\left\{-\frac{f_D(X_r^{1,s,x})+f_D(X_r^{2,s,x})}{\gamma_D}\right\},\\
\delta\sigma_r&:=\sigma^1(r,X_r^{1,s,x})-\sigma^2(r,X_r^{2,s,x}),\\
q_r&:=D f_D(X_r^{1,s,x})\sigma^1(r,X_r^{1,s,x})
 +D f_D(X_r^{2,s,x})\sigma^2(r,X_r^{2,s,x}).
\end{aligned}
\end{equation*}
Computations similar to those in the proof of Proposition \ref{estimate RSDE} give
\begin{equation*}
\begin{aligned}
|\overline X_t^{s,x}|^2\leq &C\bigg\{\int_s^t|\overline X_r^{s,x}|^2\dd r
+\int_s^t\bigl(|\overline b_r|^2+|\overline\sigma_r|^2\bigr)\dd r\bigg\}\\
&\quad+2\int_s^tR_r(\overline X_r^{s,x})^*\delta\sigma_r\dd W_r
-\frac1{\gamma_D}\int_s^tR_r|\overline X_r^{s,x}|^2q_r\dd W_r.
\end{aligned}
\end{equation*}
Then, the Burkholder--Davis--Gundy and H\"{o}lder inequalities give the following $P$-a.s. estimate:
\begin{equation*}
E\left[\sup\limits_{s\leq r\leq T}|\overline{X}_r^{s,x}|^{2p}\bigg|\mathcal{F}_s\right]\leq CE\left[\int_s^T |\overline{X}_r^{s,x}|^{2p}\dd r\bigg|\mathcal{F}_s\right]+ CE\left[\int_s^T \left(|\overline{b}_r|^{2p}+|\overline{\sigma}_r|^{2p}\right)\dd r\bigg|\mathcal{F}_s\right].
\end{equation*}
Gronwall's inequality, in turn, yields, $P$-a.s.,
\begin{equation*}
E\left[\sup\limits_{s\leq r\leq T}|\overline{X}_r^{s,x}|^{2p}\bigg|\mathcal{F}_s\right]\leq CE\left[\int_s^T \left(|\overline{b}_r|^{2p}+|\overline{\sigma}_r|^{2p}\right)\dd r\bigg|\mathcal{F}_s\right].
\end{equation*}
Moreover, applying It\^o's formula to $\phi_D^\varepsilon(X_t^{1,s,x})-\phi_D^\varepsilon(X_t^{2,s,x})$ (cf. (\ref{itok3})), we obtain
\begin{equation*}
E\left[\sup\limits_{s\leq r\leq T}|\overline{K}_r^{s,x}|^{2p}\bigg|\mathcal{F}_s\right]\leq CE\left[\int_s^T \left(|\overline{b}_r|^{2p}+|\overline{\sigma}_r|^{2p}\right)\dd r\bigg|\mathcal{F}_s\right].
\end{equation*}
The estimate for $\overline K$ follows from Assumption \ref{HD'} and the same argument as for (\ref{kk}).
\end{proof}

We next consider the GBSDE
\begin{equation}\label{GBSDE}
Y_s^{t,\xi}=\Phi(X_T^{t,\xi})+\int_s^T f(r,X_r^{t,\xi},Y_r^{t,\xi},Z_r^{t,\xi})\dd r+\int_s^T g(r,X_r^{t,\xi},Y_r^{t,\xi})\dd K_r^{t,\xi}-\int_s^T Z_r^{t,\xi}\dd W_r,\quad s\in[t,T],
\end{equation}
where $(X_s^{t,\xi},K_s^{t,\xi})_{s\in[t,T]}$ is the solution of RSDE (\ref{RSDE}). The mapping $\Phi\colon\Omega\times\overline{D}\to\mathbb{R}$ is $\mathcal{F}_T$-measurable. The functions $f\colon\Omega\times[0,T]\times\overline{D}\times\mathbb{R}\times\mathbb{R}^m\to\mathbb{R}$ and $g\colon\Omega\times[0,T]\times\overline{D}\times\mathbb{R}\to\mathbb{R}$ are progressively measurable in $(\omega,t)$ and satisfy the following assumption.
\begin{assumption}\label{HGBSDE}
\textnormal{(i)} For all $\omega\in\Omega$, $t\in[0,T]$, and $(x,y,z)\in\overline{D}\times\mathbb{R}\times\mathbb{R}^m$, the function $f(\omega,t,x,\cdot,z)$ is continuous, and $g\in \mathcal{C}_{\mathcal{F}}^3([0,T]\times\overline{D}\times\mathbb{R})$; see Definition \ref{testfunction}.\\
\textnormal{(ii)} There exist constants $\lambda_1,\lambda_2\in\mathbb{R}$ and $K\geq0$ such that, $\dd t\otimes\dd P$-a.e., the following conditions hold, for all
$(x,y,z),(x_1,y_1,z_1),(x_2,y_2,z_2)
\in\overline D\times\mathbb R\times\mathbb R^m$:
\begin{equation*}
\makebox[\linewidth][l]{\(
\begin{aligned}
\textnormal{(a)}\
&(y_1-y_2)\bigl(f(t,x,y_1,z)-f(t,x,y_2,z)\bigr)
\leq\lambda_1|y_1-y_2|^2,\\
&(y_1-y_2)\bigl(g(t,x,y_1)-g(t,x,y_2)\bigr)
\leq\lambda_2|y_1-y_2|^2;\\
\textnormal{(b)}\
&\left|\Phi(x_1)-\Phi(x_2)\right|
+\left|f(t,x_1,y,z_1)-f(t,x_2,y,z_2)\right|+\left|g(t,x_1,y)-g(t,x_2,y)\right|\\
&\leq K\left(\left|x_1-x_2\right|+\left|z_1-z_2\right|\right);\\
\textnormal{(c)}\
&|\Phi(x)|+|f(t,0,y,0)|\leq K(1+|x|+|y|),\\
&\bigl(|g|+|Dg|+|D^2g|+|\mathfrak d_tg|
 +|\mathfrak d_\omega g|+|D\mathfrak d_\omega g|\bigr)(t,x,y)
 \leq K(1+|x|+|y|).
\end{aligned}
\)}
\end{equation*}
\end{assumption}
Here $D=(D_x,D_y)$ stands for the gradient operator w.r.t.\ $(x,y)$.
Under these assumptions, \cite[Theorem~1.7]{Pardoux and Zhang} guarantees the existence and the uniqueness of a solution $\left(Y_s^{t, \xi}, Z_s^{t, \xi}\right)_{s \in[t, T]} \in \mathcal{S}^2(0,T;\mathbb{R}) \times L^2(0,T;\mathbb{R}^m)$ to GBSDE (\ref{GBSDE}). The following estimates also hold true.
\begin{proposition}\label{estimate GBSDE}
Suppose the Assumptions \ref{HD}, \ref{HD'}, \ref{HRSDE}, and \ref{HGBSDE} are fulfilled. Then there exists a constant $C\geq 0$ such that, for all $0\leq t\leq s \leq T$, $\xi\in L^2(\mathcal{F}_t;\overline{D})$, and $\xi'\in L^2(\mathcal{F}_s;\overline{D})$,
\begin{equation*}
\begin{split}
E\left[\sup\limits_{r \in[t, T]}\left|Y_r^{t, \xi}\right|^2+\int_t^T\left|Z_r^{t, \xi}\right|^2 \dd r \bigg| \mathcal{F}_t\right] \leq &C,\\
E\left[\sup\limits_{r \in\left[s, T\right]}\left|Y_r^{t, \xi}-Y_r^{s, \xi'}\right|^2+\int_s^T\left|Z_r^{t, \xi}-Z_r^{s, \xi'}\right|^2 \dd r \bigg|\mathcal{F}_{t}\right]
\leq &C\left\{\left(E\left[|\xi-\xi'|^4\big|\mathcal F_t\right]\right)^{1/4}+|s-t|^{\frac{1}{2}}\right\}.
\end{split}
\end{equation*}
\end{proposition}
\begin{proof}
The first estimate is proved in \cite{Pardoux and Zhang}, and the uniqueness of the GBSDE solution gives
\begin{equation}\label{jing}
\left|Y_s^{t, \xi}\right|=\left|Y_s^{s, X_s^{t, \xi}}\right| \leq C\left(1+\left|X_s^{t, \xi}\right|\right) \leq C,\quad P\text{-a.s.}
\end{equation}
Consequently, Assumption \ref{HGBSDE}(ii)(c) gives
\begin{equation}\label{eq:g-bound}
\begin{aligned}
&\bigl(|g|+|Dg|+|D^2g|+|\mathfrak d_tg|
 +|\mathfrak d_\omega g|+|D\mathfrak d_\omega g|\bigr)
 (s,X_s^{t,\xi},Y_s^{t,\xi})\\
&\qquad\leq K(1+|X_s^{t,\xi}|+|Y_s^{t,\xi}|)\leq C,
\quad P\text{-a.s.}
\end{aligned}
\end{equation}
We now focus on the second estimate stated in the proposition. For simplicity, set $\Psi:=\Psi^{t, \xi}$, $\Psi':=\Psi^{s, \xi'}$, and $\hat{\Psi}:=\Psi-\Psi'$, where $\Psi\in\{X,Y,Z,K\}$. We also put
\begin{equation*}
 \Delta_{t,s}(\xi,\xi'):=
 \left(E\left[|\xi-\xi'|^4\big|\mathcal F_t\right]\right)^{1/4}
 +|s-t|^{\frac{1}{2}}.
\end{equation*}
Let $R_D=\operatorname{diam}(D)$. Since $D$ is bounded,
$|\xi-\xi'|\leq R_D$, and hence
\begin{equation*}
\begin{split}
 \left(E\left[|\xi-\xi'|^4\big|\mathcal F_t\right]\right)^{\frac{1}{2}}
 \leq &R_D
 \left(E\left[|\xi-\xi'|^4\big|\mathcal F_t\right]\right)^{1/4},\\
 \left(E\left[|\xi-\xi'|^8\big|\mathcal F_t\right]\right)^{1/4}
 \leq &R_D
 \left(E\left[|\xi-\xi'|^4\big|\mathcal F_t\right]\right)^{1/4}.
\end{split}
\end{equation*}
Together with $|s-t|\leq \sqrt T\,|s-t|^{\frac{1}{2}}$, these inequalities show
that all conditional fourth- and eighth-moment terms appearing below are
bounded by a constant multiple of $\Delta_{t,s}(\xi,\xi')$.
For constants $\lambda, \mu \geq 0$ to be specified later, It\^o's formula applied to $e^{\lambda r+\mu K_r}\left|\hat{Y}_r\right|^2$ yields, for all $r \in[s, T]$,
\begin{equation}\label{itoGBSDE}
\begin{split}
&e^{\lambda r+\mu K_r}\left|\hat{Y}_r\right|^2+\int_r^T e^{\lambda l+\mu K_l}\left|\hat{Y}_l\right|^2\left(\lambda \dd l+\mu \dd K_l\right)+\int_r^T e^{\lambda l+\mu K_l}\left|\hat{Z}_l\right|^2 \dd l\\
={}&e^{\lambda T+\mu K_T}\left|\Phi\left(X_T\right)-\Phi\left(X_T^{\prime}\right)\right|^2+2 \int_r^T e^{\lambda l+\mu K_l} \hat{Y}_l\left(f\left(l, X_l,Y_l,Z_l\right)-f\left(l, X_l^{\prime},Y_l^{\prime},Z_l^{\prime}\right)\right) \dd l \\
& +2 \int_r^T e^{\lambda l+\mu K_l} \hat{Y}_l\left(g\left(l, X_l, Y_l\right) \dd K_l-g\left(l, X_l^{\prime}, Y_l^{\prime}\right) \dd K_l^{\prime}\right)-2 \int_r^T e^{\lambda l+\mu K_l}\hat{Z}_l\hat{Y}_l \dd W_l.
\end{split}
\end{equation}
It follows from Assumption \ref{HGBSDE} that
\begin{equation*}
\begin{split}
\hat{Y}_l\left(f\left(l, X_l,Y_l,Z_l\right)-f\left(l, X_l^{\prime},Y_l^{\prime},Z_l^{\prime}\right)\right)\leq &\lambda_1|\hat{Y}_l|^2+K|\hat{Y}_l|(|\hat{X}_l|+|\hat{Z}_l|)\\
\leq &|\hat{X}_l|^2+C|\hat{Y}_l|^2+\frac{1}{2}|\hat{Z}_l|^2,\\
\hat{Y}_l\left(g\left(l, X_l,Y_l\right)\dd K_l-g\left(l, X_l^{\prime},Y_l^{\prime}\right)\dd K_l^{\prime}\right)\leq &\left(\lambda_2|\hat{Y}_l|^2+K|\hat{X}_l||\hat{Y}_l|\right)\dd K_l+\hat{Y}_lg(l,X_l^{\prime},Y_l^{\prime})(\dd K_l-\dd K_l^{\prime}).
\end{split}
\end{equation*}
Substituting these inequalities in \eqref{itoGBSDE} and choosing suitable $\lambda$ and $\mu$ gives, for all $s\leq r\leq v$,
\begin{equation}\label{itoGBSDE2}
\begin{aligned}
&e^{\mu K_r}\left|\hat{Y}_r\right|^2+\frac{1}{2} \int_r^v e^{\mu K_l}\left|\hat{Z}_l\right|^2 \dd l \\
\leq &e^{\lambda v+\mu K_v}\left|\hat{Y}_v\right|^2+\int_r^v e^{\lambda l+\mu K_l}\left|\hat{X}_l\right|^2 \dd l+\int_r^v e^{\lambda l+\mu K_l}\left|\hat{X}_l\right|^2 \dd K_l\\
&\quad+2 \int_r^v e^{\lambda l+\mu K_l} \hat{Y}_l g\left(l, X_l^{\prime}, Y_l^{\prime}\right) \dd \hat{K}_l-2 \int_r^v e^{\lambda l+\mu K_l}\hat{Z}_l\hat{Y}_l \dd W_l.
\end{aligned}
\end{equation}
Next, we put $h_l=g(l,X_l',Y_l')$ and estimate the second-to-last integral.
For this end we use further down the product formula in order to transfer
the signed measure $\dd\hat K_l$ to terms
controlled by the moment bounds for $\hat K$, $Y'$, $Z'$, and the derivatives
of $g$. This also isolates the term involving the difference of the two
reflection measures. Let us be more precise. We begin with applying the It\^o--Wentzell formula to
$e^{\lambda l+\mu K_l}\hat{Y}_lh_l\hat{K}_l$. For this we set
\begin{equation*}
\begin{split}
 Q_l={}&\mathfrak d_\omega g(l,X_l',Y_l')
 +D_xg(l,X_l',Y_l')\sigma(l,X_l')
 +D_yg(l,X_l',Y_l')Z_l',\\
 B_l={}&\mathfrak d_tg(l,X_l',Y_l')
 +\frac{1}{2}\operatorname{tr}\left(
 \sigma\sigma^*(l,X_l')D^2_{xx}g(l,X_l',Y_l')\right)\\
 &+\left\langle D^2_{xy}g(l,X_l',Y_l'),
 \sigma(l,X_l')(Z_l')^*\right\rangle
 +\frac{1}{2}|Z_l'|^2D^2_{yy}g(l,X_l',Y_l')\\
 &+\sum_{i=1}^d\sum_{k=1}^m\sigma_{ik}(l,X_l')
 \partial_{x_i}(\mathfrak d_\omega g)_k(l,X_l',Y_l')
 +\left\langle D_y\mathfrak d_\omega g(l,X_l',Y_l'),Z_l'\right\rangle\\
 &+\left\langle D_xg(l,X_l',Y_l'),b(l,X_l')\right\rangle
 -D_yg(l,X_l',Y_l')f(l,X_l',Y_l',Z_l'),
\end{split}
\end{equation*}
and we obtain
\begin{equation}\label{3.29}
\begin{split}
&\int_r^v e^{\lambda l+\mu K_l}\hat Y_lh_l\,\dd \hat K_l\\
={}&e^{\lambda v+\mu K_v}\hat Y_vh_v\hat K_v
-e^{\lambda r+\mu K_r}\hat Y_rh_r\hat K_r
+\int_r^ve^{\lambda l+\mu K_l}\hat K_l
 \left(g_l^1\dd l+g_l^2\dd K_l'+g_l^3\dd K_l\right)\\
&-\int_r^ve^{\lambda l+\mu K_l}\hat K_l
 \left(\hat Y_lQ_l+h_l\hat Z_l\right)\dd W_l,
\end{split}
\end{equation}
where
\begin{equation*}
\begin{split}
 g_l^1={}&h_l\left(f(l,X_l,Y_l,Z_l)-f(l,X_l',Y_l',Z_l')\right)
 -\langle\hat Z_l,Q_l\rangle-\hat Y_l(B_l+\lambda h_l),\\
 g_l^2={}&-|h_l|^2-\hat Y_l\left[
 \left\langle D_xg(l,X_l',Y_l'),\vec n(X_l')\right\rangle
 -D_yg(l,X_l',Y_l')h_l\right],\\
 g_l^3={}&g(l,X_l,Y_l)h_l-\mu\hat Y_lh_l.
\end{split}
\end{equation*}
Combining (\ref{*3}), (\ref{*4}), (\ref{eq:g-bound}), (\ref{itoGBSDE2}), and (\ref{3.29}) yields, for $0\leq t\leq s\leq r\leq v$,
\begin{equation}\label{eq:product-bound}
\begin{split}
&E\bigg[e^{\mu K_r}|\hat{Y}_r|^2+\int_r^v e^{\mu K_l}|\hat{Z}_l|^2\dd l\bigg|\mathcal{F}_t\bigg]\\
\leq &CE\left[e^{\mu K_v}|\hat{Y}_v|^2|\mathcal{F}_t\right]+C\left(E\left[\sup\limits_{r\in[s,v]}|\hat{X}_r|^4\bigg|\mathcal{F}_t\right]\right)^{\frac{1}{2}}+C\left(E\left[\sup\limits_{r\in[s,v]}|\hat{K}_r|^4\bigg|\mathcal{F}_t\right]\right)^{\frac{1}{2}}\\
&\quad+CE\left[\int_r^v e^{\mu K_l}|\hat{K}_l|(|Z_l|^2+|Z_l^{\prime}|^2)\dd l\bigg|\mathcal{F}_t\right].
\end{split}
\end{equation}
Assumption \ref{HGBSDE}, GBSDE (\ref{GBSDE}), and (\ref{jing}) give
\begin{equation*}
\left|\int_r^u Z_l\dd W_l\right|\leq C+C\int_r^v |Z_l|\dd l+C\int_r^v \dd K_l,
\qquad u\in[r,v].
\end{equation*}
By successive application of the Burkholder--Davis--Gundy, Doob and H\"{o}lder inequalities, together with Proposition \ref{estimate RSDE}, we obtain
\begin{equation*}
\begin{aligned}
&E\left[\left(\int_r^v |Z_l|^2\dd l\right)^2\bigg|\mathcal{F}_t\right]
\leq CE\bigg[\sup\limits_{u\in[r,v]}\left|\int_r^u Z_l\dd W_l\right|^4\bigg|\mathcal{F}_t\bigg]
\leq C\left(\frac{4}{3}\right)^4E\bigg[\left|\int_r^v Z_l\dd W_l\right|^4\bigg|\mathcal{F}_t\bigg]\\
\leq &C+C|v-r|^2E\left[\left(\int_r^v |Z_l|^2\dd l\right)^2\bigg|\mathcal{F}_t\right],
\qquad 0\leq t\leq s\leq r\leq v.
\end{aligned}
\end{equation*}
We choose a partition $s=s_0<s_1<\cdots<s_n=T$ of $[s,T]$ such that
$\max\limits_{1\leq i\leq n}|s_{i-1}-s_i|
\leq \sqrt{1/(2C)}$. Then,
\begin{equation*}
E\left[\left(\int_s^T |Z_l|^2\dd l\right)^2\bigg|\mathcal{F}_t\right]\leq C,
\end{equation*}
and
Proposition \ref{estimate RSDE} gives
\begin{align}
&E\left[\int_r^T e^{\mu K_l}|\hat{K}_l||Z_l|^2\dd l\bigg|\mathcal{F}_t\right]\notag\displaybreak[1]
\leq CE\left[e^{\mu K_T}\sup\limits_{l\in[s,T]}|\hat{K}_l|
\int_r^T |Z_l|^2\dd l\bigg|\mathcal{F}_t\right]\notag\\
\leq &CE\left[e^{4\mu K_T}\bigg|\mathcal{F}_t\right]^{\frac{1}{4}}
E\left[\sup\limits_{l\in[s,T]}|\hat{K}_l|^4\bigg|\mathcal{F}_t\right]^{\frac{1}{4}}
E\left[\left(\int_r^T |Z_l|^2\dd l\right)^2\bigg|\mathcal{F}_t\right]^{\frac{1}{2}}\notag\\
\leq &C\Delta_{t,s}(\xi,\xi').
\label{estimate z1}
\end{align}
The same argument shows that $E\left[\displaystyle\int_r^T e^{\mu K_l}|\hat{K}_l||Z_l^{\prime}|^2\dd l\bigg|\mathcal{F}_t\right]\leq C\Delta_{t,s}(\xi,\xi')$. Consequently, combining (\ref{eq:product-bound}) with (\ref{estimate z1}) and both preceding estimates yields the following $P$-a.s. estimate:
\begin{equation}\label{3.33}
E\left[e^{\mu K_r}\left|\hat{Y}_r\right|^2+\int_r^Te^{\mu K_l}\left|\hat{Z}_l\right|^2 \dd l \bigg|\mathcal{F}_{t}\right] \leq C\Delta_{t,s}(\xi,\xi').
\end{equation}
Moreover, by the Burkholder--Davis--Gundy inequality and (\ref{eq:g-bound}), we obtain
\begin{align}
&E\left[\sup\limits_{r\in[s,T]}\left|\int_r^T e^{\lambda l+\mu K_l}\hat{K}_l\hat{Y}_l
\left[\mathfrak{d}_{\omega} g(l,X_l^{\prime},Y_l^{\prime})
+D_x g(l,X_l^{\prime},Y_l^{\prime})\sigma(l,X_l^{\prime})
+D_y g(l,X_l^{\prime},Y_l^{\prime})Z_l^{\prime}\right]\dd W_l\right|
\bigg|\mathcal{F}_t\right]\notag\\
\leq &CE\left[\left(\int_s^T e^{2\mu K_l}|\hat{K}_l\hat{Y}_l|^2\dd l\right)^{\frac{1}{2}}
\bigg|\mathcal{F}_t\right]
+CE\left[\left(\int_s^T e^{2\mu K_l}|\hat{K}_l\hat{Y}_lZ_l^{\prime}|^2\dd l\right)^{\frac{1}{2}}
\bigg|\mathcal{F}_t\right]\notag\displaybreak[1]\\
\leq &\frac18E\left[\sup\limits_{r\in[s,T]}e^{\mu K_r}|\hat{Y}_r|^2\bigg|\mathcal{F}_t\right]
+CE\left[\sup\limits_{l\in[s,T]}e^{\mu K_l}|\hat{K}_l|^2\bigg|\mathcal{F}_t\right]\notag\\
&\quad+\frac18E\left[\sup\limits_{r\in[s,T]}e^{\mu K_r}|\hat{Y}_r|^2\bigg|\mathcal{F}_t\right]
+CE\left[e^{\mu K_T}\sup\limits_{r\in[s,T]}|\hat{K}_r|^2
\int_s^T|Z_l^{\prime}|^2\dd l\bigg|\mathcal{F}_t\right]\notag\\
\leq &\frac14E\left[\sup\limits_{r\in[s,T]}e^{\mu K_r}|\hat{Y}_r|^2\bigg|\mathcal{F}_t\right]
+CE\left[\sup\limits_{l\in[s,T]}|\hat{K}_l|^4\bigg|\mathcal{F}_t\right]^{\frac{1}{2}}
+CE\left[\sup\limits_{l\in[s,T]}|\hat{K}_l|^8\bigg|\mathcal{F}_t\right]^{\frac14},\notag\\
&\text{and}\notag\\
&E\left[\sup\limits_{r\in[s,T]}\left|\int_r^T e^{\lambda l+\mu K_l}\hat{K}_l
 g(l,X_l^{\prime},Y_l^{\prime})\hat{Z}_l\dd W_l\right|\bigg|\mathcal{F}_t\right]\notag\\
\leq &C\left(E\left[\sup\limits_{l\in[s,T]}|\hat K_l|^4\bigg|\mathcal F_t\right]\right)^{\frac{1}{2}}
+\frac18E\left[\int_s^T e^{\mu K_l}|\hat Z_l|^2\dd l\bigg|\mathcal F_t\right].\label{estimate dw}
\end{align}
The constants in the preceding estimates are independent of $r$, and the
factors $1/8$ allow the terms containing $\hat Y$ and $\hat Z$ to be absorbed
in the left-hand side of (\ref{itoGBSDE2}). The remaining terms are
controlled by Proposition \ref{estimate RSDE}; this is the final absorption
step in the proof. The numerical fractions are used only for this absorption.

Taking the supremum over $r\in[s,T]$ and then the conditional expectation
given $\mathcal F_t$ in (\ref{itoGBSDE2}), we obtain the following
$P$-a.s. estimate:
\begin{align*}
&E\left[\sup\limits_{r \in[s, T]} e^{\mu K_r}\left|\hat{Y}_r\right|^2 \bigg| \mathcal{F}_t\right]\displaybreak[1]
\leq C\left(E\left[\sup\limits_{r \in[s, T]}\left|\hat{X}_r\right|^4 \bigg| \mathcal{F}_t\right]\right)^{\frac{1}{2}}
+C\left(E\left[\sup\limits_{r \in[s, T]}\left|\hat{K}_r\right|^8 \bigg| \mathcal{F}_t\right]\right)^{\frac{1}{4}}\\
&\quad+C\left(E\left[\sup\limits_{r \in[s, T]}\left|\hat{K}_r\right|^4 \bigg| \mathcal{F}_t\right]\right)^{\frac{1}{2}}
+C E\left[\int_s^T e^{\mu K_l}\left|\hat{Z}_l\right|^2 \dd l \bigg| \mathcal{F}_t\right]
\leq C\Delta_{t,s}(\xi,\xi').
\end{align*}
Consequently, Proposition \ref{estimate RSDE} and (\ref{3.33}) allow to conclude. This completes the proof.
\end{proof}

We also need the following stability estimate for GBSDEs.

\begin{proposition}\label{stability GBSDE}
Suppose Assumption \ref{HGBSDE} holds for $(\Phi^i,f^i,g^i)$, $i=1,2$.
Let $(X^{i,t,\xi}_s,K^{i,t,\xi}_s)_{s\in[t,T]}$ denote the solution of
RSDE (\ref{RSDE}) with the coefficients $(b^i,\sigma^i)$. Put
\begin{equation*}
\begin{gathered}
 \overline{\Phi}=\sup\limits_{x\in\overline{D}}|\Phi^1(x)-\Phi^2(x)|,\qquad
 \overline{f}_t=\sup\limits_{(x,y,z)\in\overline{D}\times\mathbb{R}\times\mathbb{R}^m}|f^1(t,x,y,z)-f^2(t,x,y,z)|,\\
 \overline{g}_t=\sup\limits_{(x,y)\in\overline{D}\times\mathbb{R}}|g^1(t,x,y)-g^2(t,x,y)|,\\
 \overline{b}_t=\sup\limits_{x\in\overline{D}}|b^1(t,x)-b^2(t,x)|,\qquad
 \overline{\sigma}_t=\sup\limits_{x\in\overline{D}}|\sigma^1(t,x)-\sigma^2(t,x)|.
\end{gathered}
\end{equation*}
For $i=1,2$, let $(Y^{i,t,\xi},Z^{i,t,\xi})$ solve
\begin{equation*}
Y_s^{i,t,\xi}=\Phi^i(X_T^{i,t,\xi})+\int_s^T f^i(r,X_r^{i,t,\xi},Y_r^{i,t,\xi},Z_r^{i,t,\xi})\dd r+\int_s^T g^i(r,X_r^{i,t,\xi},Y_r^{i,t,\xi})\dd K_r^{i,t,\xi}-\int_s^T Z_r^{i,t,\xi}\dd W_r,
\end{equation*}
for $s\in[t,T]$. Then
\begin{equation*}
\begin{aligned}
&E\left[\sup\limits_{r\in[s,T]}e^{\mu K_r^{1,t,\xi}}
\left|Y_r^{1,t,\xi}-Y_r^{2,t,\xi}\right|^2
+\int_s^Te^{\mu K_l^{1,t,\xi}}
\left|Z_l^{1,t,\xi}-Z_l^{2,t,\xi}\right|^2 \dd l \bigg|\mathcal{F}_{s}\right]\\
\leq &C\left(E\left[|\overline{\Phi}|^4\bigg|\mathcal{F}_s\right]\right)^{\frac{1}{2}}
+C\left(E\left[\sup\limits_{r\in[s,T]}|\overline{f}_r|^4\bigg|\mathcal{F}_s\right]\right)^{\frac{1}{2}}
+C\left(E\left[\sup\limits_{r\in[s,T]}|\overline{g}_r|^4\bigg|\mathcal{F}_s\right]\right)^{\frac{1}{2}}\\
&\quad+\left(E\left[\int_0^T\left(\left|\overline{b}_l\right|^4+\left|\overline{\sigma}_l\right|^4\right)\dd l\bigg|\mathcal{F}_s\right]\right)^{\frac{1}{2}}
+\left(E\left[\int_0^T\left(\left|\overline{b}_l\right|^4+\left|\overline{\sigma}_l\right|^4\right)\dd l\bigg|\mathcal{F}_s\right]\right)^{\frac{1}{4}}.
\end{aligned}
\end{equation*}
\end{proposition}
\begin{proof}
The proof is similar to that of Proposition \ref{estimate GBSDE}. To simplify, we set $\Psi^i :=\Psi^{i,t, \xi}$ and $\hat{\Psi}:=\Psi^1 -\Psi^2$, where $\Psi\in\{X,Y,Z,K\}$ and $i=1,2$. For constants $\lambda, \mu \geq 0$ to be specified later, It\^o's formula applied to $e^{\lambda r+\mu K^1_r}\left|\hat{Y}_r\right|^2$ yields, for $r\in[t, T]$,
\begin{equation*}
\begin{split}
&e^{\lambda r+\mu K_r^1}\left|\hat{Y}_r\right|^2+\int_r^T e^{\lambda l+\mu K_l^1}\left|\hat{Y}_l\right|^2\left(\lambda \dd l+\mu \dd K_l^1\right)+\int_r^T e^{\lambda l+\mu K_l^1}\left|\hat{Z}_l\right|^2 \dd l\\
={}&e^{\lambda T+\mu K_T^1}\left|\Phi^1\left(X^1_T\right)-\Phi^2\left(X_T^{2}\right)\right|^2+2 \int_r^T e^{\lambda l+\mu K^1_l} \hat{Y}_l\left(f^1\left(l, X^1_l,Y^1_l,Z^1_l\right)-f^2\left(l, X_l^{2},Y_l^{2},Z_l^{2}\right)\right) \dd l \\
& +2 \int_r^T e^{\lambda l+\mu K^1_l} \hat{Y}_l\left(g^1\left(l, X^1_l, Y^1_l\right) \dd K^1_l-g^2\left(l, X_l^{2}, Y_l^{2}\right) \dd K_l^{2}\right)-2 \int_r^T e^{\lambda l+\mu K^1_l}\hat{Z}_l\hat{Y}_l \dd W_l.
\end{split}
\end{equation*}
It follows from Assumption \ref{HGBSDE} that
\begin{equation*}
\begin{split}
&2\hat{Y}_l\left(f^1\left(l, X^1_l,Y^1_l,Z^1_l\right)-f^2\left(l, X_l^{2},Y_l^{2},Z_l^{2}\right)\right)\\
\leq &|\overline{f}_l|^2+|\hat{X}_l|^2+C|\hat{Y}_l|^2+\frac{1}{2}|\hat{Z}_l|^2,\\
&\hat{Y}_l\left(g^1\left(l, X^1_l,Y^1_l\right)\dd K^1_l-g^2\left(l, X_l^{2},Y_l^{2}\right)\dd K_l^{2}\right)\\
\leq &\left(|\overline{g}_l|^2+C|\hat{Y}_l|^2+|\hat{X}_l|^2\right)\dd K^1_l+\hat{Y}_lg^2(l,X_l^{2},Y_l^{2})\dd \hat{K}_l.
\end{split}
\end{equation*}
By choosing suitable $\lambda,\mu\geq0$, we have, for every $r\in[s,T]$,
\begin{equation}\label{itoGBSDE2222}
\begin{aligned}
&e^{\mu K^1_r}\left|\hat{Y}_r\right|^2
+\frac{1}{2}\int_r^T e^{\mu K^1_l}\left|\hat{Z}_l\right|^2 \dd l\\
\leq &Ce^{\mu K^1_T}\left|\Phi^1\left(X^1_T\right)-\Phi^2\left(X_T^{2}\right)\right|^2
+C\int_r^T e^{\mu K^1_l}\left|\hat{X}_l\right|^2 \dd l
+C\int_r^T e^{\mu K^1_l}\left|\hat{X}_l\right|^2 \dd K^1_l
+C\int_r^T e^{\mu K^1_l}\left|\overline{f}_l\right|^2 \dd l\\
&\quad+C\int_r^T e^{\mu K^1_l}\left|\overline{g}_l\right|^2 \dd K^1_l
+2\int_r^T e^{\lambda l+\mu K^1_l}\hat{Y}_l g^2\left(l,X_l^{2},Y_l^{2}\right)\dd\hat{K}_l
-2\int_r^T e^{\lambda l+\mu K^1_l}\hat{Z}_l\hat{Y}_l\dd W_l.
\end{aligned}
\end{equation}
We now apply the It\^o--Wentzell formula to $e^{\lambda l+\mu K^1_l} \hat{Y}_l g^2\left(l, X_l^{2}, Y_l^{2}\right)\hat{K}_l$. The argument used for (\ref{3.29})--(\ref{estimate z1}) can be used also here. In particular, we use the estimate
\begin{equation*}
E\left[\int_r^T e^{\lambda l+\mu K^1_l}
 \left|\overline{g}_l\right|^2\dd K^1_l\bigg|\mathcal{F}_s\right]
\leq C
 E\left[\sup\limits_{l\in[s,T]}|\overline{g}_l|^4
 \bigg|\mathcal{F}_s\right]^{\frac{1}{2}}
 E\left[|K^1_T|^4\bigg|\mathcal{F}_s\right]^{1/4}.
\end{equation*}
This allows to obtain, for $s\leq r\leq T$,
\begin{equation*}
\begin{aligned}
&e^{\mu K_r^1}\left|\hat{Y}_r\right|^2
+E\left[\int_r^Te^{\mu K_l^1}\left|\hat{Z}_l\right|^2 \dd l \bigg|\mathcal{F}_{r}\right]\\
\leq &C\left(E\left[|\overline{\Phi}|^4\bigg|\mathcal{F}_r\right]\right)^{\frac{1}{2}}
+C\left(E\left[\sup\limits_{v\in[r,T]}|\overline{f}_v|^4\bigg|\mathcal{F}_r\right]\right)^{\frac{1}{2}}
+CE\left[\sup\limits_{l\in[r,T]}|\overline{g}_l|^4\bigg|\mathcal{F}_r\right]^{\frac{1}{2}}
E\left[|K^1_T|^4\bigg|\mathcal{F}_r\right]^{\frac{1}{4}}\\
&\quad+C\left(E\left[\sup\limits_{v\in[r,T]}|\hat{X}_v|^4\bigg|\mathcal{F}_r\right]\right)^{\frac{1}{2}}
+C\left(E\left[\sup\limits_{v\in[r,T]}|\hat{K}_v|^4\bigg|\mathcal{F}_r\right]\right)^{\frac{1}{2}}
+C\left(E\left[\sup\limits_{v\in[r,T]}|\hat{K}_v|^4\bigg|\mathcal{F}_r\right]\right)^{\frac{1}{4}}.
\end{aligned}
\end{equation*}
On the other hand, from (\ref{itoGBSDE2222}),
\begin{equation*}
\begin{aligned}
&\sup\limits_{r\in[s,T]}e^{\mu K^1_r}\left|\hat{Y}_r\right|^2
+\frac{1}{2}\int_s^T e^{\mu K^1_l}\left|\hat{Z}_l\right|^2 \dd l\\
\leq &Ce^{\mu K^1_T}\left|\overline{\Phi}\right|^2
+Ce^{\mu K^1_T}|\hat X_T|^2
+C\int_s^T e^{\mu K^1_l}\left|\hat{X}_l\right|^2 \dd l
+C\int_s^T e^{\mu K^1_l}\left|\hat{X}_l\right|^2 \dd K^1_l
+C\int_s^T e^{\mu K^1_l}\left|\overline{f}_l\right|^2 \dd l\\
&\quad+C\int_s^T e^{\mu K^1_l}\left|\overline{g}_l\right|^2 \dd K^1_l
+C\int_s^T e^{\mu K^1_l}\hat{Y}_l g^2\left(l,X_l^{2},Y_l^{2}\right)\dd\hat{K}_l
+C\sup\limits_{r\in[s,T]}\left|\int_r^T e^{\lambda l+\mu K^1_l}\hat{Z}_l\hat{Y}_l \dd W_l\right|.
\end{aligned}
\end{equation*}
Taking the conditional expectation $E[\cdot|\mathcal{F}_s]$, we first use
the Burkholder--Davis--Gundy inequality to control the stochastic integral.
The finite-variation term containing $\dd\hat K$ is handled by the same
product-formula argument as in (\ref{estimate dw}); the coefficient difference
$\overline g$ is kept separate from the differences of the state and solution
processes. Young's inequality absorbs the mixed terms involving $\hat Y$ and
$\hat Z$. Finally, Proposition \ref{SRSDE} converts the estimates for
$\hat X$ and $\hat K$ into the stated coefficient differences. We thus obtain
\begin{equation*}
\begin{aligned}
&E\left[\sup\limits_{r\in[s,T]}e^{\mu K_r^1}\left|\hat{Y}_r\right|^2
+\int_s^Te^{\mu K_l^1}\left|\hat{Z}_l\right|^2 \dd l \bigg|\mathcal{F}_{s}\right]\\
\leq &C\left(E\left[|\overline{\Phi}|^4\bigg|\mathcal{F}_s\right]\right)^{\frac{1}{2}}
+C\left(E\left[\sup\limits_{r\in[s,T]}|\overline{f}_r|^4\bigg|\mathcal{F}_s\right]\right)^{\frac{1}{2}}
+CE\left[\sup\limits_{l\in[s,T]}|\overline{g}_l|^4\bigg|\mathcal{F}_s\right]^{\frac{1}{2}}
E\left[|K^1_T|^4\bigg|\mathcal{F}_s\right]^{\frac{1}{4}}\\
&\quad+C\left(E\left[\sup\limits_{r\in[s,T]}|\hat{X}_r|^4\bigg|\mathcal{F}_s\right]\right)^{\frac{1}{2}}
+C\left(E\left[\sup\limits_{r\in[s,T]}|\hat{K}_r|^4\bigg|\mathcal{F}_s\right]\right)^{\frac{1}{2}}
+C\left(E\left[\sup\limits_{r\in[s,T]}|\hat{K}_r|^4\bigg|\mathcal{F}_s\right]\right)^{\frac{1}{4}}
\end{aligned}
\end{equation*}
\begin{equation*}
\begin{aligned}
\leq &C\left(E\left[|\overline{\Phi}|^4\bigg|\mathcal{F}_s\right]\right)^{\frac{1}{2}}
+C\left(E\left[\sup\limits_{r\in[s,T]}|\overline{f}_r|^4\bigg|\mathcal{F}_s\right]\right)^{\frac{1}{2}}
+CE\left[\sup\limits_{l\in[s,T]}|\overline{g}_l|^4\bigg|\mathcal{F}_s\right]^{\frac{1}{2}}\qquad\qquad\qquad\quad\\
&+E\left[\int_0^T\left(\left|\overline{b}_l\right|^4+\left|\overline{\sigma}_l\right|^4\right)\dd l\bigg|\mathcal{F}_s\right]^{\frac{1}{2}}
+E\left[\int_0^T\left(\left|\overline{b}_l\right|^4+\left|\overline{\sigma}_l\right|^4\right)\dd l\bigg|\mathcal{F}_s\right]^{\frac{1}{4}}.
\end{aligned}
\end{equation*}
This completes the proof.
\end{proof}
\section{Value functions and viscosity solutions}
In this section, we consider the stochastic differential game described in
the introduction under the following assumptions.
\begin{assumption}\label{HSDG}
\textnormal{(i)} $\Phi \in L^{\infty}\left(\Omega, \mathcal{F}_T;H^{1, \infty}\left(\mathbb{R}^d\right)\right)$.\\
\textnormal{(ii)} Each component of $b$, $\sigma$, $f$, $g$, $\mathfrak d_tg$, and $\mathfrak d_\omega g$ is predictable in $(\omega,t)$, Borel measurable in its other variables, and continuous in all variables, for fixed $\omega\in\Omega$. Moreover, for every $(\theta,\gamma)\in\Theta_0\times\Gamma_0$, $g(\cdot,\cdot,\cdot,\theta,\gamma)\in\mathcal C_{\mathcal F}^3([0,T]\times\overline{D}\times\mathbb R)\cap\mathcal S^8(C_b(\overline{D}\times\mathbb R))$.\\
\textnormal{(iii)} There exist constants $L>0$ and $L_f,L_g\in[0,L]$ such that
$\|\Phi\|_{L^{\infty}\left(\Omega,\mathcal{F}_T;H^{1,\infty}(\mathbb{R}^d)\right)}\leq L$, but also, such that, for all $(\omega,t)\in\Omega\times[0,T]$ and all
$(x,y,z,\theta,\gamma),(x_1,y_1,z_1,\theta,\gamma),(x_2,y_2,z_2,\theta,\gamma)
\in\mathbb{R}^d\times\mathbb{R}\times\mathbb{R}^m\times\Theta_0\times\Gamma_0$,
the following inequalities are fulfilled:
\begin{equation*}
\begin{aligned}
\textnormal{(a)}\quad &|b(t,x,\theta,\gamma)|+|\sigma(t,x,\theta,\gamma)|\leq L,\\
&|f(t,x,y,z,\theta,\gamma)|\leq L(1+|y|+|z|),\\
\textnormal{(b)}\quad &|\Phi(x_1)-\Phi(x_2)|
+|b(t,x_1,\theta,\gamma)-b(t,x_2,\theta,\gamma)|\\
&+|\sigma(t,x_1,\theta,\gamma)-\sigma(t,x_2,\theta,\gamma)|
+|f(t,x_1,y,z,\theta,\gamma)-f(t,x_2,y,z,\theta,\gamma)|\\
&+|g(t,x_1,y,\theta,\gamma)-g(t,x_2,y,\theta,\gamma)|
\leq L|x_1-x_2|,\\
\textnormal{(c)}\quad &|f(t,x,y_1,z_1,\theta,\gamma)-f(t,x,y_2,z_2,\theta,\gamma)|
\leq L_f(|y_1-y_2|+|z_1-z_2|),\\
&|g(t,x,y_1,\theta,\gamma)-g(t,x,y_2,\theta,\gamma)|
\leq L_g|y_1-y_2|,\\
\textnormal{(d)}\quad
&\bigl(|g|+|Dg|+|D^2g|+|\mathfrak d_tg|
 +|\mathfrak d_\omega g|+|D\mathfrak d_\omega g|\bigr)
 (t,x,y,\theta,\gamma)\leq L(1+|x|+|y|).
\end{aligned}
\end{equation*}
\textnormal{(iv)} $D$ satisfies the Assumptions \ref{HD} and \ref{HD'}.
\end{assumption}
Propositions \ref{estimate RSDE} and \ref{estimate GBSDE} yield the
following two lemmas, which state the flow and stability estimates used
in what follows.
\begin{lemma}\label{pro RSDE}
Suppose Assumption \ref{HSDG} holds. Given
$(\theta,\gamma)\in\Theta\times\Gamma$ and $p\geq2$, there exists
$C>0$, depending only on $L,T,p$, such that the following assertions hold
for the solution of RSDE (\ref{eq RSDE}), whenever
$0\leq \rho\leq t\leq s\leq T$ and
$\xi,\xi'\in L^0(\Omega,\mathcal F_\rho;\overline{D})$.

\begin{enumerate}[\upshape(i)]
\item The processes
$\left(X_r^{\rho,\xi;\theta,\gamma}\right)_{t\leq r\leq T}$ and
$\left(X_r^{t,X_t^{\rho,\xi;\theta,\gamma};\theta,\gamma}\right)_{t\leq r\leq T}$
are indistinguishable. The same is true of
$\left(K_r^{\rho,\xi;\theta,\gamma}-K_t^{\rho,\xi;\theta,\gamma}\right)_{t\leq r\leq T}$
and
$\left(K_r^{t,X_t^{\rho,\xi;\theta,\gamma};\theta,\gamma}\right)_{t\leq r\leq T}$.

\item $E\left[\sup\limits_{r\in[t,T]}|K_r^{t,\xi;\theta,\gamma}|^{2p}\mid\mathcal F_t\right]
\leq C(1+|\xi|^{2p})$, $P$-a.s.
\item $E\left[\sup\limits_{r\in[t,T]}|X_r^{t,\xi;\theta,\gamma}-X_r^{t,\xi';\theta,\gamma}|^{2p}
+\sup\limits_{r\in[t,T]}|K_r^{t,\xi;\theta,\gamma}-K_r^{t,\xi';\theta,\gamma}|^{2p}
\mid\mathcal F_t\right]\leq C|\xi-\xi'|^{2p}$, $P$-a.s.
\item $E\left[|X_s^{t,\xi;\theta,\gamma}-\xi|^{2p}
+|K_s^{t,\xi;\theta,\gamma}|^{2p}\mid\mathcal F_t\right]
\leq C|s-t|^p$, $P$-a.s.
\end{enumerate}
\end{lemma}
\begin{lemma}\label{pro GRSDE}
\textnormal{(i)} For all $(t, \theta, \gamma) \in[0, T) \times \Theta \times \Gamma$ and every $\xi \in L^2\left(\Omega, \mathcal{F}_t;\overline{D}\right)$, GBSDE (\ref{1.3}) has a unique solution $\left(Y^{t,\xi;\theta,\gamma}, Z^{t,\xi;\theta,\gamma}\right) \in \mathcal{S}^2([t, T];\mathbb{R}) \times L^2\left([t, T];\mathbb{R}^m\right)$ and the solution satisfies
\begin{equation*}
\begin{aligned}
E\left[\sup\limits_{s \in[t, T]}\left|Y_s^{t,\xi;\theta,\gamma}\right|^2+\int_t^T\left|Z_s^{t,\xi;\theta,\gamma}\right|^2 \dd s\bigg|\mathcal{F}_t\right] \leq &C, \quad P\text{-a.s.}, \\
\left|Y_t^{t,\xi;\theta,\gamma}\right| \leq &C, \quad P\text{-a.s.};
\end{aligned}
\end{equation*}
\textnormal{(ii)} Given another $\xi' \in L^2\left(\Omega, \mathcal{F}_t;\overline{D}\right)$, the following estimates hold $P$-a.s.:
\begin{equation*}
\begin{aligned}
E\left[\sup\limits_{s \in[t, T]}\left|Y_s^{t,\xi;\theta,\gamma}-Y_s^{t,\xi';\theta,\gamma}\right|^2+\int_t^T\left|Z_s^{t,\xi;\theta,\gamma}-Z_s^{t,\xi';\theta,\gamma}\right|^2 \dd s\bigg|\mathcal{F}_t\right] \leq &C|\xi-\xi'|,\\
\left|Y_t^{t,\xi;\theta,\gamma}-Y_t^{t,\xi';\theta,\gamma}\right| \leq &C|\xi-\xi'|^{\frac{1}{2}}.
\end{aligned}
\end{equation*}
\end{lemma}

The preceding lemmas give the following bounds for the value functions and the payoff.
\begin{proposition}\label{bound}
Suppose Assumption \ref{HSDG} holds. Then
\begin{equation*}
\underset{(t, x) \in[0, T] \times \overline{D}}{\esssup}\ \ \underset{(\theta, \gamma) \in \Theta \times \Gamma}{\esssup} \max \{|V(t, x)|,|U(t, x)|,|J(t,x;\theta,\gamma)|\} \leq C, \quad P\text{-a.s.}
\end{equation*}
(Recall that $J(t,x;\theta,\gamma)$, $V(t,x)$ and $U(t,x)$ are defined in (\ref{1.3}), (\ref{lower function}) and (\ref{upper function}), respectively.) Moreover, there exists a constant $L_0>0$ such that, for every $t\in[0,T]$ and $x,\bar{x}\in\overline{D}$,
\begin{equation*}
\begin{aligned}
&|U(t, x)-U(t, \bar{x})|+|V(t, x)-V(t, \bar{x})|
+\underset{(\theta, \gamma) \in \Theta \times \Gamma}{\esssup}
|J(t,x;\theta,\gamma)-J(t,\bar{x};\theta,\gamma)|\\
\leq &L_0|x-\bar{x}|^{\frac{1}{2}}, \quad P\text{-a.s.}
\end{aligned}
\end{equation*}
\end{proposition}

Following \cite{Li08sicon}, we define the admissible controls and strategies for the game.
\begin{definition}
An admissible control process $\left(\theta_r\right)_{r \in[t, s]}$ (resp. $\left(\gamma_r\right)_{r \in[t, s]}$) for player I (resp. II) on $[t,s]$, $0\leq t<s\leq T$, is an $\left\{\mathcal{F}_r\right\}_{r\in[t,s]}$-progressively measurable process taking values in $\Theta_0$ (resp. $\Gamma_0$). The set of all admissible controls for player I (resp. II) on $[t, s]$ is denoted by $\Theta_{t, s}$ (resp. $\Gamma_{t, s}$). We identify two processes $\theta$ and $\theta'$ in $\Theta_{t, s}$ and write $\theta \equiv \theta'$ on $[t, s]$, if $P\{\theta=\theta'\text{ a.e. on }[t,s]\}=1$.
\end{definition}

\begin{definition}\label{strategy}
 A nonanticipative strategy for player I on $[t,s]$, $0\leq t<s\leq T$, is a mapping $\alpha\colon\Gamma_{t,s}\to\Theta_{t,s}$ such that, for any stopping time $S\in\mathcal T_{t,s}$ and any $\gamma_1,\gamma_2\in\Gamma_{t,s}$, if $\gamma_1\equiv\gamma_2$ on $[t,S]$, then $\alpha(\gamma_1)\equiv\alpha(\gamma_2)$ on $[t,S]$, where $[t,S]=\{(r,\omega)\in[0,T]\times\Omega:t\leq r\leq S(\omega)\}$. Nonanticipative strategies for player II on $[t,s]$, namely mappings $\mu\colon\Theta_{t,s}\to\Gamma_{t,s}$, are defined symmetrically. The set of all nonanticipative strategies $\alpha\colon\Gamma_{t,s}\to\Theta_{t,s}$ for player I on $[t,s]$ is denoted by $\mathcal A_{t,s}$. The corresponding set for player II is denoted by $\mathcal M_{t,s}$.
\end{definition}

The preceding control and strategy definitions also apply to stochastic
intervals $[t,s]$, where $t,s\in\mathcal T^0$ and $t\leq s$.
We extend the backward semigroup introduced in \cite{Peng1997} for BSDEs to the present GBSDE setting. For $(t,x)\in[0,T]\times\overline D$, stopping times $\tau\in\mathcal T_{t,T}$ and $\sigma\in\mathcal T_{t,\tau}$, controls $\theta\in\Theta_{t,\tau}$ and $\gamma\in\Gamma_{t,\tau}$, and $\xi\in L^2(\mathcal F_\tau;\mathbb R)$, define $G_{\sigma,\tau}^{t,x;\theta,\gamma}[\xi]:=\overline Y_\sigma^{t,x;\theta,\gamma}$, where $(\overline Y_s^{t,x;\theta,\gamma},\overline Z_s^{t,x;\theta,\gamma})_{s\in[t,\tau]}$ solves the following GBSDE:
\begin{equation*}
\left\{\begin{aligned}
-\dd \overline{Y}_s^{t,x;\theta,\gamma}={}&f\left(s, X_s^{t,x;\theta,\gamma}, \overline{Y}_s^{t,x;\theta,\gamma}, \overline{Z}_s^{t,x;\theta,\gamma}, \theta_s, \gamma_s\right) \dd s-\overline{Z}_s^{t,x;\theta,\gamma}\dd W_s \\
&\quad+g\left(s, X_s^{t,x;\theta,\gamma}, \overline{Y}_s^{t,x;\theta,\gamma}, \theta_s, \gamma_s\right) \dd K_s^{t,x;\theta,\gamma}, \quad s \in[t, \tau],\\
\overline{Y}_\tau^{t,x;\theta,\gamma}={}&\xi.
\end{aligned}\right.
\end{equation*}
Here $(X^{t,x;\theta,\gamma},K^{t,x;\theta,\gamma})$ is the solution of (\ref{eq RSDE}). The backward semigroup has the following semigroup property: for all $\tau\in\mathcal{T}_{t,T}$,
\begin{equation*}
J(t,x;\theta,\gamma)=Y_{t}^{t,x;\theta,\gamma}=G_{t,T}^{t,x;\theta,\gamma}\left[\Phi(X_T^{t,x;\theta,\gamma})\right]=G_{t,\tau}^{t,x;\theta,\gamma}\left[Y_{\tau}^{t,x;\theta,\gamma}\right],\quad P\text{-a.s.}
\end{equation*}
We now prove the continuity of the value functions and the payoff.
\begin{theorem}\label{continuity}
Suppose Assumption \ref{HSDG} holds. Then $V(\cdot,\cdot)$, $U(\cdot,\cdot)$, and $J(\cdot,\cdot;\theta,\gamma)$ are continuous on $[0,T]\times\overline{D}$, $P$-a.s.
\end{theorem}
\begin{proof}
First, by Proposition \ref{bound} we see
\begin{equation*}
|J(t,x_1;\theta,\gamma)-J(t,x_2;\theta,\gamma)|\leq C|x_1-x_2|^{\frac{1}{2}},\quad P\text{-a.s.}
\end{equation*}
for all $x_1,x_2\in\overline{D}$, where the constant $C$ is independent of $\theta,\gamma$. This shows that $V(t,\cdot)$, $U(t,\cdot)$, and $J(t,\cdot;\theta,\gamma)$ are continuous on $\overline{D}$, $P$-a.s., for all $t\in[0,T]$ and $(\theta,\gamma)\in\Theta\times\Gamma$.

Next, fix $x\in\overline D$ and consider the continuity in time. We prove only the continuity of $t\mapsto V(t,x)$; the proofs for $U$ and $J$ are analogous. Given any $\varepsilon>0$, \cite[Lemma~5.5]{Qiu18sicon} and \cite[Lemma~2.1]{Qiu22AAP} allow to choose $n\geq4$, a partition $0=t_0<t_1<\cdots<t_{n-1}<t_n=T$, and deterministic functions $\Phi_n \in C_b^3\left(\mathbb{R}^{mn+d}\right)$,
\begin{equation*}
b_n^i, \sigma_n^{i j} \in C_b^3\left([0, T] \times \Theta_0 \times \Gamma_0;C_b^3(\mathbb{R}^{mn+d}) \right), \quad 1 \leq i \leq d, 1 \leq j \leq m,
\end{equation*}
as well as smooth functions $f_n$ and $g_n$ defined on the corresponding
finite-dimensional domains, such that their derivatives of positive order are bounded,
and their zeroth-order dependence on $(y,z)$ has at most linear growth. This choice can be made such that the following three conditions are fulfilled:\\
\textnormal{(i)} Defined by
\begin{equation*}
\begin{aligned}
&\Phi^{\varepsilon} :=\sup\limits_{x \in \mathbb{R}^d}\left|\Phi_n\left(W_{t_1}, \ldots, W_{t_n}, x\right)-\Phi(x)\right|,\\
&b_t^{\varepsilon} :=\sup\limits_{(x, \theta, \gamma) \in \mathbb{R}^d \times \Theta_0 \times \Gamma_0}\left|b_n\left(W_{t_1 \wedge t}, \ldots, W_{t_n \wedge t}, t, x, \theta, \gamma\right)-b(t, x, \theta, \gamma)\right|, \\
&\sigma_t^{\varepsilon} :=\sup\limits_{(x, \theta, \gamma) \in \mathbb{R}^d \times \Theta_0 \times \Gamma_0}\left|\sigma_n\left(W_{t_1 \wedge t}, \ldots, W_{t_n \wedge t}, t, x, \theta, \gamma\right)-\sigma(t, x, \theta, \gamma)\right|,
\end{aligned}
\end{equation*}
\begin{equation*}
\begin{aligned}
&f_t^{\varepsilon}:=\sup\limits_{(x, y, z, \theta, \gamma) \in \mathbb{R}^d \times \mathbb{R} \times \mathbb{R}^m \times \Theta_0 \times \Gamma_0}\left|f_n\left(W_{t_1 \wedge t}, \ldots, W_{t_n \wedge t}, t, x, y, z, \theta, \gamma\right)-f(t, x, y, z, \theta, \gamma)\right|,\\
&g_t^{\varepsilon}:=\sup\limits_{(x, y, \theta, \gamma)\in \mathbb{R}^d \times \mathbb{R} \times \Theta_0 \times \Gamma_0}\left|g_n\left(W_{t_1 \wedge t}, \ldots, W_{t_n \wedge t}, t, x, y, \theta, \gamma\right)-g(t, x, y, \theta, \gamma)\right|,\quad t\in[0,T],
\end{aligned}
\end{equation*}
the mapping $\Phi^\varepsilon$ is $\mathcal F_T$-measurable, the processes $b^\varepsilon$, $\sigma^\varepsilon$, $f^\varepsilon$, and $g^\varepsilon$ are $\{\mathcal F_t\}_{t\in[0,T]}$-adapted, and
\begin{equation*}
\left\|\Phi^{\varepsilon}\right\|_{L^8\left(\mathcal{F}_T;\mathbb{R}\right)}+\left\|f^{\varepsilon}\right\|_{\mathcal{S}^8(0,T;\mathbb{R})}+\left\|g^{\varepsilon}\right\|_{\mathcal{S}^8(0,T;\mathbb{R})}+\left\|b^{\varepsilon}\right\|_{L^8\left(0,T;\mathbb{R}^d\right)}+\left\|\sigma^{\varepsilon}\right\|_{L^8\left(0,T;\mathbb{R}^{d\times m}\right)}\leq\varepsilon,
\end{equation*}
where $\varepsilon=\frac{1}{n}$.\\
\begingroup
\setlength{\emergencystretch}{1em}
\textnormal{(ii)} $\left\{\Phi_n,b_n,\sigma_n\right\}_{n\geq4}$ and the zero-level
families $f_n(\cdot,\allowbreak\cdot,0,0,\cdot,\cdot)$ and
$g_n(\cdot,\allowbreak\cdot,0,\cdot,\cdot)$ are uniformly bounded by a constant
independent of $n$;\\
\textnormal{(iii)} For $n\geq 4$, the functions $\Phi_n,b_n,\sigma_n,f_n,g_n$ are uniformly Lipschitz continuous in their respective spatial variables, with a Lipschitz constant $L_c$ independent of $n$ and $\varepsilon$. Together with (ii), this gives a common linear-growth bound for $f_n$ and $g_n$.
\par
\endgroup

We denote by $G_{s,T}^{n,s,x;\theta,\gamma}[\cdot]$ the backward semigroup associated with the generators $f_n$ and $g_n$. For all $(s,x)\in[0,T)\times\overline{D}$, we set
\begin{equation*}
V^{\varepsilon}(s, x)=\underset{\mu \in \mathcal{M}}{\essinf}\ \underset{\theta \in \Theta}{\esssup}\  G_{s, T}^{n, s,x;\theta,\mu(\theta)}\left[\Phi_n\left(W_{t_1}, \ldots, W_{t_n}, X_T^{s,x;\theta,\mu(\theta),n}\right)\right],
\end{equation*}
where $X^{s,x;\theta,\gamma,n}$ satisfies the RSDE
\begin{equation*}
\left\{\begin{aligned}
\dd X_t^{s,x;\theta,\gamma,n}={}&b_n\left(W_{t_1 \wedge t}, \ldots, W_{t_n \wedge t}, t, X_t^{s,x;\theta,\gamma,n}, \theta_t, \gamma_t\right) \dd t \\
& +\sigma_n\left(W_{t_1 \wedge t}, \ldots, W_{t_n \wedge t}, t, X_t^{s,x;\theta,\gamma,n}, \theta_t, \gamma_t\right) \dd W_t \\
& +D\phi_D(X_t^{s,x;\theta,\gamma,n})\dd K_t^{s,x;\theta,\gamma,n},\\
\dd K_t^{s,x;\theta,\gamma,n}={}&\mathbf{1}_{\partial D}(X_t^{s,x;\theta,\gamma,n})\dd K_t^{s,x;\theta,\gamma,n},\\
X_s^{s,x;\theta,\gamma,n}={}&x.
\end{aligned}\right.
\end{equation*}
As in \cite{Qiu22AAP}, we include the Brownian motion $W$ in the state process and obtain a Markovian stochastic differential game in an augmented state space. By \cite{Li08sicon,Fan}, its value function $V^{\varepsilon}(s,x)$, for $s\in[t_{n-1},T]$, has the form
\begin{equation*}
V^{\varepsilon}(s, x)=\tilde{V}^{\varepsilon}\left(s, x, W_{t_1}, \ldots, W_{t_{n-1}}, W_s\right),
\end{equation*}
where $\tilde V^\varepsilon$ is the value function associated with the following augmented state dynamics on $[t_{n-1},T]$:
\begin{equation*}
\left\{\begin{aligned}
\dd X_t^{s,x,y;\theta,\gamma,n}
={}&b_n\left(W_{t_1},\ldots,W_{t_{n-1}},W_t^{s,y},t,
X_t^{s,x,y;\theta,\gamma,n},\theta_t,\gamma_t\right)\dd t\\
&\quad+\sigma_n\left(W_{t_1},\ldots,W_{t_{n-1}},W_t^{s,y},t,
X_t^{s,x,y;\theta,\gamma,n},\theta_t,\gamma_t\right)\dd W_t\\
&\quad+D\phi_D(X_t^{s,x,y;\theta,\gamma,n})
\dd K_t^{s,x,y;\theta,\gamma,n},\\
\dd W_t^{s,y}={}&\dd W_t,\qquad t\in[s,T],\\
X_s^{s,x,y;\theta,\gamma,n}={}&x,\qquad W_s^{s,y}=y,
\end{aligned}\right.
\end{equation*}
and it is given by
\begin{equation*}
\begin{aligned}
&\tilde{V}^{\varepsilon}
 \left(s,x,W_{t_1},\ldots,W_{t_{n-1}},y\right)\\
:={}&\underset{\mu\in\mathcal M}{\essinf}\
 \underset{\theta\in\Theta}{\esssup}\
 G_{s,T}^{n,s,x;\theta,\mu(\theta)}
 \left[\Phi_n\left(W_{t_1},\ldots,W_{t_{n-1}},W_T^{s,y},
 X_T^{s,x,y;\theta,\mu(\theta),n}\right)\right].
\end{aligned}
\end{equation*}
Because $\tilde{V}^{\varepsilon}$ is the lower value function of a Markovian stochastic differential game, it is continuous in time, and so is $V^{\varepsilon}$. By Lemma \ref{pro GRSDE}(i), $|V^\varepsilon(t,x)|\leq C$.
In particular, since $V^\varepsilon(\cdot,x)$ is continuous, it belongs to
$\mathcal S^4(\mathbb R)$.

By the RSDE stability result in Proposition \ref{SRSDE}, for all $(\theta,\mu)\in\Theta\times\mathcal{M}$,
\begin{equation*}
\begin{split}
&E\left[\sup\limits_{s \leq t \leq T}\left|X_t^{s,x;\theta,\mu(\theta),n}-X_t^{s,x;\theta,\mu(\theta)}\right|^2+\sup\limits_{s \leq t \leq T}\left|K_t^{s,x;\theta,\mu(\theta),n}-K_t^{s,x;\theta,\mu(\theta)}\right|^2\bigg|\mathcal{F}_s\right] \\
\leq &C_0 E\left[\int_s^T\left(\left|b_t^{\varepsilon}\right|^4+\left|\sigma_t^{\varepsilon}\right|^4\right) \dd t\bigg|\mathcal{F}_s\right]^\frac{1}{2},
\end{split}
\end{equation*}
where $C_0$ is independent of $n$, $\varepsilon$, $s$, $x$, and $(\theta,\mu)$. By the definitions of $V$ and $V^\varepsilon$, Proposition
\ref{stability GBSDE} and the preceding RSDE estimate give
\begin{equation*}
\begin{aligned}
&|V^{\varepsilon}(s,x)-V(s,x)|^2
\leq \underset{(\mu,\theta)\in\mathcal M\times\Theta}{\esssup}
\left|Y_s^{n,s,x;\theta,\mu(\theta)}
-Y_s^{s,x;\theta,\mu(\theta)}\right|^2\\
\leq &C_1\bigg\{
E\left[\int_0^T\bigl(|b_t^\varepsilon|^4+|\sigma_t^\varepsilon|^4\bigr)\dd t
\biggm|\mathcal F_s\right]^{\frac{1}{2}}+E\left[\int_0^T\bigl(|b_t^\varepsilon|^4+|\sigma_t^\varepsilon|^4\bigr)\dd t
\biggm|\mathcal F_s\right]^{\frac14}
+E\left[|\Phi^\varepsilon|^4\mid\mathcal F_s\right]^{\frac{1}{2}}\\
&+E\left[\sup\limits_{r\in[s,T]}|f_r^\varepsilon|^4\mid\mathcal F_s\right]^{\frac{1}{2}}
+E\left[\sup\limits_{r\in[s,T]}|g_r^\varepsilon|^4\mid\mathcal F_s\right]^{\frac{1}{2}}\bigg\}.
\end{aligned}
\end{equation*}
Here $C_1$ is a constant independent of $n$, $\varepsilon$, $s$, and $x$. By taking the supremum over $s$ and applying Doob's inequality, we obtain
\begin{equation*}
\begin{aligned}
&\|V^\varepsilon(\cdot,x)-V(\cdot,x)\|_{\mathcal S^4(\mathbb R)}^4\\
\leq &C\bigg\{
E[|\Phi^\varepsilon|^8]^{\frac{1}{2}}
+E[\sup\limits_{t\in[0,T]}|f_t^\varepsilon|^8]^{\frac{1}{2}}
+E[\sup\limits_{t\in[0,T]}|g_t^\varepsilon|^8]^{\frac{1}{2}}\\
&\quad+E\left[\int_0^T(|b_t^\varepsilon|^8+|\sigma_t^\varepsilon|^8)\dd t\right]^{\frac{1}{2}}
+E\left[\int_0^T(|b_t^\varepsilon|^8+|\sigma_t^\varepsilon|^8)\dd t\right]^{\frac14}\bigg\}\\
\leq &C(\varepsilon^4+\varepsilon^2)\longrightarrow0,
\qquad \varepsilon\longrightarrow0.
\end{aligned}
\end{equation*}
Hence, $V(\cdot, x) \in \mathcal{S}^4(\mathbb{R})$. In particular, $V(t, x)$ is continuous in $t$. Combined with its $\frac{1}{2}$-H\"{o}lder continuity in $x$, uniform with respect to $t$, this completes the proof.
\end{proof}
\begin{remark}
In the nonrecursive case with a homogeneous Neumann boundary condition,
\begin{equation*}
J(t,x;\theta,\gamma)=E\left[\int_{t}^{T}f(s,X_s^{t,x;\theta,\gamma},\theta_s,\gamma_s)\dd s + \Phi(X_T^{t,x;\theta,\gamma})\bigg|\mathcal{F}_t\right],
\end{equation*}
the continuity properties of the value functions hold under Assumption \ref{HD} alone. Indeed, the\linebreak payoff contains no integral with respect to $K$, and continuity in the initial state follows from the stability estimate for $X$, which does not require Assumption \ref{HD'}.
\end{remark}
We next state the DPP. Its proof is given in Appendix \ref{app:DPP}.
\begin{theorem}\label{DPP}
Suppose Assumption \ref{HSDG} holds. For all $\tau,\hat\tau\in\mathcal T^0$ with $\tau\leq\hat\tau$ and $\xi\in L^2(\mathcal F_\tau;\overline D)$, the lower value function satisfies
\begin{equation*}
V(\tau, \xi)=\underset{\mu \in \mathcal{M}_{\tau, \hat\tau}}{\essinf} \esssup\limits_{\theta \in \Theta_{\tau, \hat\tau}} G_{\tau, \hat\tau}^{\tau,\xi;\theta,\mu(\theta)}\left[V\left(\hat\tau, X_{\hat\tau}^{\tau,\xi;\theta,\mu(\theta)}\right)\right], \quad P\text{-a.s.},
\end{equation*}
and for the upper value function $U(\tau, \xi)$, we have
\begin{equation*}
U(\tau, \xi)=\underset{\alpha \in \mathcal{A}_{\tau, \hat\tau}}{\esssup} \essinf\limits_{\gamma \in \Gamma_{\tau, \hat\tau}} G_{\tau, \hat\tau}^{\tau,\xi;\alpha(\gamma),\gamma}\left[U\left(\hat\tau, X_{\hat\tau}^{\tau,\xi;\alpha(\gamma),\gamma}\right)\right], \quad P\text{-a.s.}
\end{equation*}
\end{theorem}

We now extend the definition of viscosity solutions in \cite{Qiu22AAP} to
Neumann boundary conditions. We first introduce the nonlinear operators used in what follows.
For any $M,N\geq0$ and stopping times $\tau_1$ and $\tau_2$ with
$\tau_1\leq\tau_2$, $P$-a.s., we set
\begin{equation*}
\overline{\mathcal{E}}_{\tau_1, \tau_2}^{M,N;x,\theta,\gamma}[\xi]=\overline{Y}_{\tau_1}, \quad \underline{\mathcal{E}}_{\tau_1, \tau_2}^{M,N;x,\theta,\gamma}[\xi]=\underline{Y}_{\tau_1}, \quad \text{ for } \xi \in L^2\left(\Omega, \mathcal{F}_{\tau_2}, P\right),
\end{equation*}
where $\left(\overline{Y}_s,\overline{Z}_s\right)_{\tau_1 \leq s \leq \tau_2}$ and $\left(\underline{Y}_s,\underline{Z}_s\right)_{\tau_1 \leq s \leq \tau_2}$ are the solutions of the following GBSDEs:
\begin{equation}\label{01}
\overline{Y}_t=\xi+\displaystyle\int_t^{\tau_2} M\left(\left|\overline{Y}_s\right|+\left|\overline{Z}_s\right|\right) \dd s + \displaystyle\int_t^{\tau_2} N \left|\overline{Y}_s\right|\dd K_s^{\tau_1,x;\theta,\gamma}-\displaystyle\int_t^{\tau_2} \overline{Z}_s \dd W_s, \quad \tau_1 \leq t \leq \tau_2,
\end{equation}
and
\begin{equation}\label{02}
\underline{Y}_t=\xi-\displaystyle\int_t^{\tau_2} M\left(\left|\underline{Y}_s\right|+\left|\underline{Z}_s\right|\right) \dd s - \displaystyle\int_t^{\tau_2} N \left|\underline{Y}_s\right|\dd K_s^{\tau_1,x;\theta,\gamma}-\displaystyle\int_t^{\tau_2} \underline{Z}_s \dd W_s, \quad \tau_1 \leq t \leq \tau_2,
\end{equation}
respectively. Thus, $M$ controls the $\dd s$ generator and $N$ controls the
$\dd K_s$ generator. Here $K^{\tau_1,x;\theta,\gamma}$ is the increasing
process in the reflected state equation (\ref{eq RSDE}). We then set
\begin{equation}\label{x1}
\overline{\mathcal{E}}_{\tau_1, \tau_2}^{M,N}[\xi]=\underset{(x,\gamma,\theta)\in\overline D\times\Gamma\times\Theta}{\esssup}\overline{\mathcal{E}}_{\tau_1, \tau_2}^{M,N;x,\theta,\gamma}[\xi],
\end{equation}
\begin{equation}\label{x2}
\underline{\mathcal{E}}_{\tau_1, \tau_2}^{M,N}[\xi]=\underset{(x,\gamma,\theta)\in\overline D\times\Gamma\times\Theta}{\essinf}\underline{\mathcal{E}}_{\tau_1, \tau_2}^{M,N;x,\theta,\gamma}[\xi].
\end{equation}
When $N=0$, the local-time generator disappears, and these operators reduce
to the nonlinear expectations used for the Cauchy problem in
\cite{Qiu22AAP}, with $M$ as their single parameter.
They satisfy the following properties.
\begin{lemma}\label{function}
For all stopping times $\tau_1\leq\tau_2$, the operators
$\overline{\mathcal{E}}_{\tau_1,\tau_2}^{M,N}$ and
$\underline{\mathcal{E}}_{\tau_1,\tau_2}^{M,N}$ map
$L^2(\Omega,\mathcal{F}_{\tau_2};\mathbb{R})$ into
$L^2(\Omega,\mathcal{F}_{\tau_1};\mathbb{R})$ and satisfy, for
$\xi,\xi_1,\xi_2\in L^2(\Omega,\mathcal{F}_{\tau_2};\mathbb{R})$:\\
\textnormal{(i)} $\overline{\mathcal{E}}_{\tau_1, \tau_2}^{M,N}[\xi]
=-\underline{\mathcal{E}}_{\tau_1, \tau_2}^{M,N}[-\xi]$, $P$-a.s.; for
bounded nonnegative $\mathcal{F}_{\tau_1}$-measurable random variables
$\rho$, both operators are positively homogeneous in the sense that
$\overline{\mathcal E}_{\tau_1,\tau_2}^{M,N}[\rho\xi]
=\rho\overline{\mathcal E}_{\tau_1,\tau_2}^{M,N}[\xi]$ and
$\underline{\mathcal E}_{\tau_1,\tau_2}^{M,N}[\rho\xi]
=\rho\underline{\mathcal E}_{\tau_1,\tau_2}^{M,N}[\xi]$.\\
\textnormal{(ii)} If $0\leq M\leq\widehat M$ and $0\leq N\leq\widehat N$, then
$\overline{\mathcal{E}}_{\tau_1, \tau_2}^{M,N}[\xi]
\leq\overline{\mathcal{E}}_{\tau_1, \tau_2}^{\widehat M,\widehat N}[\xi]$
and
$\underline{\mathcal{E}}_{\tau_1, \tau_2}^{M,N}[\xi]
\geq\underline{\mathcal{E}}_{\tau_1, \tau_2}^{\widehat M,\widehat N}[\xi]$,
$P$-a.s.\\
\textnormal{(iii)} The following inequalities hold, $P$-a.s.:
\begin{equation}\label{*}
\underline{\mathcal{E}}_{\tau_1, \tau_2}^{L_f,L_g}\left[\xi_1-\xi_2\right] \leq G_{\tau_1, \tau_2}^{\tau_1,x;\theta,\gamma}\left[\xi_1\right]-G_{\tau_1, \tau_2}^{\tau_1,x;\theta,\gamma}\left[\xi_2\right] \leq \overline{\mathcal{E}}_{\tau_1, \tau_2}^{L_f,L_g}\left[\xi_1-\xi_2\right].
\end{equation}
\end{lemma}
\begin{proof}
First, by standard estimates for GBSDEs (see \cite{Pardoux and Zhang}) and
the boundedness of $\overline{D}$, both operators are well defined in
$L^2(\Omega,\mathcal F_{\tau_1};\mathbb R)$. The uniqueness of the solutions gives (i), and the
comparison theorem for GBSDEs proves the componentwise monotonicity in (ii).
It remains to prove (iii). For $i=1,2$,
$G_{s, \tau_2}^{\tau_1,x;\theta,\gamma}[\xi_i]
=\overline{Y}_s^{\tau_1,x;\theta,\gamma,\xi_i}$, $s\in[\tau_1,\tau_2]$,
is the solution of the following GBSDE:
\begin{equation*}
\left\{\begin{aligned}
-\dd \overline{Y}_s^{\tau_1,x;\theta,\gamma,\xi_i}={}& f\left(s, X_s^{\tau_1,x;\theta,\gamma}, \overline{Y}_s^{\tau_1,x;\theta,\gamma,\xi_i}, \overline{Z}_s^{\tau_1,x;\theta,\gamma,\xi_i}, \theta_s, \gamma_s\right) \dd s-\overline{Z}_s^{\tau_1,x;\theta,\gamma,\xi_i}\dd W_s \\
&\quad+g\left(s, X_s^{\tau_1,x;\theta,\gamma}, \overline{Y}_s^{\tau_1,x;\theta,\gamma,\xi_i}, \theta_s, \gamma_s\right) \dd K_s^{\tau_1,x;\theta,\gamma},\quad s \in[\tau_1, \tau_2];\\
\overline{Y}_{\tau_2}^{\tau_1,x;\theta,\gamma,\xi_i}={}&\xi_i,
\end{aligned}\right.
\end{equation*}
where $(X^{\tau_1,x;\theta,\gamma},K^{\tau_1,x;\theta,\gamma})$ is the solution of (\ref{eq RSDE}). For all control processes $\theta,\gamma$, all $y_1,y_2\in\mathbb{R}$, and all $z_1,z_2\in\mathbb{R}^m$, the following inequalities hold $P$-a.s.:
\begin{equation*}
f\left(s, X_s^{\tau_1,x;\theta,\gamma}, y_1, z_1, \theta_s, \gamma_s\right)-f\left(s, X_s^{\tau_1,x;\theta,\gamma}, y_2, z_2, \theta_s, \gamma_s\right)\leq L_f(|y_1-y_2|+|z_1-z_2|),
\end{equation*}
\begin{equation*}
g\left(s, X_s^{\tau_1,x;\theta,\gamma}, y_1, \theta_s, \gamma_s\right)-g\left(s, X_s^{\tau_1,x;\theta,\gamma}, y_2, \theta_s, \gamma_s\right)\leq L_g|y_1-y_2|.
\end{equation*}
Thus, the comparison theorem for GBSDEs gives
\begin{equation*}
\begin{aligned}
&G_{\tau_1, \tau_2}^{\tau_1,x;\theta,\gamma}\left[\xi_1\right]
-G_{\tau_1, \tau_2}^{\tau_1,x;\theta,\gamma}\left[\xi_2\right]\\
\leq &\overline{\mathcal{E}}_{\tau_1, \tau_2}^{L_f,L_g;x,\theta,\gamma}[\xi_1-\xi_2]\\
\leq &\underset{(x,\gamma,\theta)\in\overline D\times\Gamma\times\Theta}{\esssup}
\overline{\mathcal{E}}_{\tau_1, \tau_2}^{L_f,L_g;x,\theta,\gamma}[\xi_1-\xi_2]
=\overline{\mathcal{E}}_{\tau_1, \tau_2}^{L_f,L_g}[\xi_1-\xi_2].
\end{aligned}
\end{equation*}
The first inequality in (\ref{*}) follows analogously.
\end{proof}
We now introduce the notion of viscosity solutions. For every $M,N\geq0$,
$(u,\tau)\in\mathcal S^2(C_b(\overline{D}))\times\mathcal T^0$,
$\Omega_\tau\in\mathcal F_\tau$ with $P(\Omega_\tau)>0$, and
$\xi\in L^0(\Omega_\tau,\mathcal F_\tau;\overline{D})$, we define the
classes of test fields
\begin{equation*}
\begin{aligned}
\underline{\mathcal{G}}^{M,N} u\left(\tau,\xi;\Omega_\tau\right)
&:=\bigg\{\phi\in
\mathcal C_{\mathcal F}^3([0,T]\times\mathbb R^d):
\text{ for some }\hat\tau\in\mathcal T_+^\tau,
(\phi-u)(\tau, \xi) \mathbf{1}_{\Omega_\tau}\\
&\quad=\underset{\bar\tau \in \mathcal{T}^\tau}{\essinf}\
\underline{\mathcal{E}}_{\tau, \bar\tau \wedge \hat\tau}^{M,N}
\left[\inf\limits_{y \in\overline{D}}(\phi-u)(\bar\tau \wedge \hat\tau, y)\right]
\mathbf{1}_{\Omega_\tau}=0,\quad P\text{-a.s.}\bigg\},
\end{aligned}
\end{equation*}
and
\begin{equation*}
\begin{aligned}
\overline{\mathcal{G}}^{M,N} u\left(\tau,\xi;\Omega_\tau\right)
&:=\bigg\{\phi\in
\mathcal C_{\mathcal F}^3([0,T]\times\mathbb R^d):
\text{ for some }\hat\tau\in\mathcal T_+^\tau,
(\phi-u)(\tau, \xi) \mathbf{1}_{\Omega_\tau}\\
&\quad=\underset{\bar\tau \in \mathcal{T}^\tau}{\esssup}\
\overline{\mathcal{E}}_{\tau, \bar\tau \wedge \hat\tau}^{M,N}
\left[\sup\limits_{y \in\overline{D}}(\phi-u)(\bar\tau \wedge \hat\tau, y)\right]
\mathbf{1}_{\Omega_\tau}=0,\quad P\text{-a.s.}\bigg\}.
\end{aligned}
\end{equation*}
We call $\hat\tau$ the stopping time associated with the corresponding test
relation. If either test class is nonempty, then necessarily
$0\leq\tau<T$ on $\Omega_\tau$.

\begin{definition}\label{vs}
We say that $u \in \mathcal{S}^2\left(C_b\left(\overline{D}\right)\right)$ is a viscosity subsolution (resp. supersolution) of BSPDE (\ref{HJBI1}) (equivalently (\ref{HJBI3})) if $u(T,x)\leq\Phi(x)$ (resp. $u(T,x)\geq\Phi(x)$) for all $x\in\overline D$, $P$-a.s., and there exist $M_0,N_0\geq0$ such that, for all $M\geq M_0$, $N\geq N_0$, $\tau\in\mathcal T^0$, $\Omega_\tau\in\mathcal F_\tau$ with $P(\Omega_\tau)>0$, $\xi\in L^0(\Omega_\tau,\mathcal F_\tau;\overline D)$, and all $\phi\in\underline{\mathcal G}^{M,N}u(\tau,\xi;\Omega_\tau)$ (resp. $\phi\in\overline{\mathcal G}^{M,N}u(\tau,\xi;\Omega_\tau)$), the following conditions hold:
\begin{equation}\label{vs1}
\underset{[\tau,T]\times D\ni(s, x) \rightarrow\left(\tau^{+}, \xi\right)}{\essliminf}\footnote{For a random field $F$ on $[0,T]\times D$, the essential lower limit at $(t_0,x_0)$ is $\lim_{\delta\downarrow0}\essinf\{F(s,x):t_0<s<t_0+\delta,\quad x\in B_\delta(x_0)\cap D\}$. The essential upper limit is defined analogously.}\left\{-\mathfrak{d}_s \phi-\mathbb{H}_{-}\left(D^2 \phi, D \mathfrak{d}_\omega \phi, D \phi, \phi, \mathfrak{d}_\omega \phi\right)\right\}(s, x) \leq 0,
\end{equation}
$P$-a.s. on $\Omega_{\tau}\cap\{\xi\in D\}$ $\Bigl($resp.,
\begin{equation}\label{vs2}
\underset{[\tau,T]\times D\ni(s, x) \rightarrow\left(\tau^{+}, \xi\right)}{\esslimsup}\left\{-\mathfrak{d}_s \phi-\mathbb{H}_{-}\left(D^2 \phi, D \mathfrak{d}_\omega \phi, D \phi, \phi, \mathfrak{d}_\omega \phi\right)\right\}(s, x) \geq 0,
\end{equation}
$P$-a.s. on $\Omega_{\tau}\cap\{\xi\in D\}$ $\Bigr)$
and
\begin{equation}\label{vs3}
\underset{[\tau,T]\times \partial D\ni(s, x) \rightarrow\left(\tau^{+}, \xi\right)}{\essliminf}\left\{-\mathbb{G}_-\left(\mathfrak{d}_s \phi,\frac{\partial \phi}{\partial \vec n}, D^2 \phi, D \mathfrak{d}_\omega\phi, D \phi, \phi, \mathfrak{d}_\omega\phi\right)\right\}(s,x)\leq 0,
\end{equation}
$P$-a.s. on $\Omega_{\tau}\cap\{\xi\in\partial D\}$ $\Bigl($resp.,
\begin{equation}\label{vs4}
\underset{[\tau,T]\times \partial D\ni(s, x) \rightarrow\left(\tau^{+}, \xi\right)}{\esslimsup}\left\{-\mathbb{G}_+\left(\mathfrak{d}_s \phi,\frac{\partial \phi}{\partial \vec n}, D^2 \phi, D \mathfrak{d}_\omega\phi, D \phi, \phi, \mathfrak{d}_\omega\phi\right)\right\}(s,x)\geq 0,
\end{equation}
$P$-a.s. on $\Omega_{\tau}\cap\{\xi\in\partial D\}$ $\Bigr)$, where $\mathbb H_-$, $\mathbb G_-$, and $\mathbb G_+$ are defined in (\ref{h-}).
The function $u$ is a viscosity solution of BSPDE (\ref{HJBI1}) (equivalently (\ref{HJBI3})) if it is both a viscosity subsolution and a viscosity supersolution of (\ref{HJBI1}). If the above conditions (\ref{vs1})--(\ref{vs4}) hold with thresholds $(M_0,N_0)$, we call $u$ an $(M_0,N_0)$-viscosity subsolution, supersolution, or solution, respectively.
\end{definition}
\begin{remark}\label{vs5}
As in \cite[Remark~4.3]{Qiu22AAP}, continuity of the coefficients allows
(\ref{vs1}) to be replaced by
\begin{equation*}
\underset{s \rightarrow \tau^{+}}{\essliminf}\left\{-\mathfrak{d}_s \phi-\mathbb{H}_{-}\left(D^2 \phi, D \mathfrak{d}_\omega \phi, D \phi, \phi, \mathfrak{d}_\omega \phi\right)\right\}(s, \xi) \leq 0,
\end{equation*}
and (\ref{vs2}), (\ref{vs3}), and (\ref{vs4}) can also be replaced similarly.
\end{remark}
The following lemma is proved as in \cite[Remark~4.2]{Qiu22AAP}.
\begin{lemma}\label{mon}
If $u$ is a viscosity subsolution (resp. supersolution) of BSPDE (\ref{HJBI3}), then, for every $\lambda\in\mathbb{R}$, $\tilde{u}(t, x):=e^{\lambda t} u(t, x)$ is a viscosity subsolution (resp. supersolution) of
\begin{equation*}
\left\{\begin{array}{l}
-\mathfrak{d}_t \tilde{u}(t, x)=\mathbb{K}\left(t, x, D^2 \tilde{u}(t, x), D \mathfrak{d}_\omega \tilde{u}(t, x), D \tilde{u}(t, x), \tilde{u}(t, x), \mathfrak{d}_\omega \tilde{u}(t, x)\right),
\quad(t, x) \in Q; \\
\mathbb{F}_-\left(t, x,\mathfrak{d}_t \tilde{u}(t, x),\frac{\partial\tilde{u}}{\partial\vec n}(t,x), D^2 \tilde{u}(t, x), D \mathfrak{d}_\omega \tilde{u}(t, x), D \tilde{u}(t, x), \tilde{u}(t, x), \mathfrak{d}_\omega \tilde{u}(t, x)\right) = 0,\quad x \in \partial D;\\
\mathbb{F}_+\left(t, x,\mathfrak{d}_t \tilde{u}(t, x),\frac{\partial\tilde{u}}{\partial\vec n}(t,x), D^2 \tilde{u}(t, x), D \mathfrak{d}_\omega \tilde{u}(t, x), D \tilde{u}(t, x), \tilde{u}(t, x), \mathfrak{d}_\omega \tilde{u}(t, x)\right) = 0,\quad x \in \partial D;\\
\tilde{u}(T, x)=e^{\lambda T} \Phi(x), \quad x \in \overline{D},
\end{array}\right.
\end{equation*}
where $(t,x,A,B,p,y,z)\in[0,T]\times\mathbb R^d\times\mathbb R^{d\times d}
\times\mathbb R^{d\times m}\times\mathbb R^d\times\mathbb R\times\mathbb R^m$, and
\begin{equation*}
\begin{aligned}
\mathbb{K}(t, x, A, B, p, y, z):={}&-\lambda y+e^{\lambda t} \mathbb{H}_{-}\left(t, x, e^{-\lambda t} A, e^{-\lambda t} B, e^{-\lambda t} p, e^{-\lambda t} y, e^{-\lambda t} z\right),
\end{aligned}
\end{equation*}
and $\mathbb F_-$ and $\mathbb F_+$ are defined in the same way as $\mathbb G_-$ and $\mathbb G_+$. Under this exponential transformation, only the $\dd s$-parameter of the nonlinear expectation must be augmented by at most $|\lambda|$; the $\dd K_s$-parameter $N$ is unchanged. Therefore, without loss of generality, we may assume that $\mathbb H_-(t,x,A,B,p,y,z)$ is decreasing in $y$.
\end{lemma}
\begin{remark}
In Lemma \ref{function}, the GBSDE comparison theorem is applied for the proof of (ii) and (iii) with the same
increasing process $K^{\tau_1,x;\theta,\gamma}$. Taking the essential supremum and the essential infimum in
(\ref{x1}) and (\ref{x2}), respectively, then removes the dependence on
$(x,\theta,\gamma)$.
\end{remark}
\section{The existence result for the stochastic HJBI equation}
In this section, we prove that the lower value function $V$ is a viscosity
solution of the stochastic HJBI equation (\ref{HJBI3}) in the two-parameter
sense of Definition \ref{vs}. The proof for $U$ is similar.
\begin{theorem}\label{existence}
Suppose Assumption \ref{HSDG} holds. Then, the lower value function $V$ is an
$(L_f,L_g)$-viscosity solution of the stochastic HJBI equation
(\ref{HJBI3}).
\end{theorem}
\begin{proof}
We prove Theorem \ref{existence} by combining BSDE techniques with the DPP.
For a global test field
$\varphi\in\mathcal C_{\mathcal F}^3([0,T]\times\mathbb R^d)$, we define
\begin{equation*}
\begin{aligned}
F(s, x, y, z, \theta, \gamma)
={}&\mathfrak{d}_s \varphi(s, x)+\operatorname{tr}\left(\frac{1}{2} \sigma \sigma^{*}(s, x, \theta, \gamma) D^2 \varphi(s, x)+\sigma(s, x, \theta, \gamma)\left(D \mathfrak{d}_\omega \varphi(s, x)\right)^{*}\right) \\
& \quad+b^{*}(s, x, \theta, \gamma) D \varphi(s, x) \\
& \quad+f\left(s, x, y+\varphi(s, x), z+(D \varphi(s, x))^{*} \sigma(s, x, \theta, \gamma)+\mathfrak{d}_\omega \varphi(s, x), \theta, \gamma\right),\\
G(s,x,y,\theta,\gamma)={}&\frac{\partial{\varphi}}{\partial\vec n}(s,x)+g(s,x,y+\varphi(s,x),\theta,\gamma).
\end{aligned}
\end{equation*}
By Assumption \ref{HSDG}, the functions $F$ and $G$, evaluated along the
state and the controls, are uniformly Lipschitz in $(y,z)$ and in $y$,
respectively. Their Lipschitz constants are independent of the controls.
For a controlled trajectory $(X,K)$ starting at $t\leq T_0<T$, put
$F_s^0:=F(s,X_s,0,0,\theta_s,\gamma_s)$ and
$G_s^0:=G(s,X_s,0,\theta_s,\gamma_s)$. By Definition
\ref{testfunction} and Proposition \ref{estimate RSDE}, for every $\lambda>0$,
\begin{equation}\label{FG-integrability}
 E\int_t^{T_0}e^{\lambda K_s}
 \bigl(|F_s^0|^2\dd s+|G_s^0|^2\dd K_s\bigr)<\infty.
\end{equation}

Fix $M\geq L_f$ and $N\geq L_g$. We assume that, for
$\xi\in L^0(\Omega_\tau,\mathcal F_\tau;\overline{D})$, the test field
$\varphi$ belongs to
$\underline{\mathcal G}^{M,N}V(\tau,\xi;\Omega_\tau)$.

\noindent\textbf{Step 1.} We show that $V$ is a viscosity subsolution.

By Proposition \ref{bound} and Theorem \ref{continuity}, we have
$V\in \mathcal S^4(C_b(\overline{D}))$, and $V(T,x)=\Phi(x)$ for all
$x\in\overline{D}$. We first work on the event $\{\xi\in D\}$. Note that
$\{\omega:\xi(\omega)\in D\}=\bigcup_{n\geq1}
\{\omega:\operatorname{dist}(\xi(\omega),\partial D)>1/n\}$. Thus it is
enough to work on
$\mathcal D_n=\{\operatorname{dist}(\xi,\partial D)>1/n\}$. By Remark
\ref{vs5}, it suffices to verify the following inequality for every
$\varphi\in\underline{\mathcal G}^{M,N}V(\tau,\xi;\Omega_\tau)$, with
$\tau\in\mathcal T^0$, $\Omega_\tau\in\mathcal F_\tau$ of positive
probability, and
$\xi\in L^0(\Omega_\tau,\mathcal F_\tau;\overline{D})$, $P$-a.s. on
$\Omega_\tau\cap\mathcal D_n$:
\begin{equation}\label{5*}
\essliminf_{s \rightarrow \tau^{+}}\left\{-\mathfrak{d}_s \varphi-\mathbb{H}_{-}\left(D^2 \varphi, D \mathfrak{d}_\omega \varphi, D \varphi, \varphi, \mathfrak{d}_\omega \varphi\right)\right\}(s, \xi) \leq 0.
\end{equation}
Suppose that (\ref{5*}) does not hold. Then there exist $k>0$, $\Omega_\tau'\in\mathcal{F}_\tau$, and $\delta'>0$ such that $\Omega_\tau'\subset\Omega_\tau\cap\mathcal D_n$, $P(\Omega_\tau')>0$, and
\begin{equation*}
\underset{s \in\left[\tau, (\tau+\delta') \wedge T\right]}{\esssup} F_0(s, \xi, 0,0)\leq -k<0,\quad P\text{-a.s. on }\Omega_\tau',
\end{equation*}
where
\begin{equation*}
F_0(s, \xi, 0,0)=\underset{\theta \in \Theta_0}{\esssup}\ \underset{\gamma \in \Gamma_0}{\essinf}\ F(s, \xi, 0,0, \theta, \gamma).
\end{equation*}
Again, we may choose $T_0\in(0,T)$ such that $P(\{\tau<T_0\}\cap\Omega'_\tau)>0$. Without loss of generality, we assume $\hat\tau=T_0$ and $\Omega_\tau=\Omega'_\tau=\Omega$.
The predictable measurable selection theorem and the continuity of $F$ yield a nonanticipative strategy $\widehat{\gamma}\in\mathcal M_{\tau,\tau+\delta}$ and $0<\delta<\frac{1}{2n}$ such that
\begin{equation}\label{newstar}
F(\tau', \xi', 0,0, \theta_{\tau'}, \widehat{\gamma}(\theta)_{\tau'}) \leq-\frac{k}{2} \quad P\text{-a.s. on }\{\tau'\in[\tau,\tau+\delta]\}\cap O_{\delta,\xi'}(\xi)\cap\mathcal{D}_n,
\end{equation}
for all $\tau'\in\mathcal T^\tau$, $\theta\in\Theta_{\tau,\tau+\delta}$, and
$\xi'\in L^0(\Omega,\mathcal F_{\tau'};\overline{D})$, where
$O_{\delta,\xi'}(\xi):=\{\omega:|\xi'(\omega)-\xi(\omega)|\leq\delta\}$.
Indeed, $\widehat\gamma$ can be chosen through a jointly measurable, adapted
function $\gamma(s,\omega,x,\theta)$ such that
$\widehat\gamma(\theta)_{\tau'}=\gamma(\tau',\omega,\xi',\theta_{\tau'})$.
Let $\delta_0=\delta'\wedge\delta\wedge\frac{1}{2}(T-T_0)$. For any $\alpha\in(0,\delta_0]$, we consider the following GBSDE:
\begin{equation}\label{1}
\left\{\begin{aligned}
-\dd Y_s^{1, \theta,\gamma}={}&F\left(s, X_s^{\tau,\xi;\theta,\gamma}, Y_s^{1, \theta,\gamma}, Z_s^{1, \theta,\gamma}, \theta_s,\gamma_s\right) \dd s+G\left(s, X_s^{\tau,\xi;\theta,\gamma}, Y_s^{1,\theta,\gamma},\theta_s,\gamma_s\right) \dd K_s^{\tau,\xi;\theta,\gamma} \\
& -Z_s^{1, \theta,\gamma} \dd W_s, \quad s \in[\tau, \tau+\alpha]; \\
Y_{\tau+\alpha}^{1, \theta,\gamma}={}&0,
\end{aligned}\right.
\end{equation}
where $(X_s^{\tau,\xi;\theta,\gamma},K_s^{\tau,\xi;\theta,\gamma})$ solves (\ref{eq RSDE}) and $(\theta,\gamma)\in\Theta_{\tau,\tau+\alpha}\times\Gamma_{\tau,\tau+\alpha}$. By (\ref{FG-integrability}) and the standard GBSDE theory, there exists a unique solution, which satisfies the following identity.
\begin{lemma}\label{relationship}
Let $\varphi\in\mathcal C_{\mathcal F}^3([0,T]\times\mathbb R^d)$. Then, for every $(\theta,\gamma)\in\Theta_{\tau,\tau+\alpha}\times\Gamma_{\tau,\tau+\alpha}$ and for all $s\in[\tau,\tau+\alpha]$,
\begin{equation*}
Y_s^{1, \theta,\gamma}=G_{s, \tau+\alpha}^{\tau,\xi;\theta,\gamma}\left[\varphi\left(\tau+\alpha, X_{\tau+\alpha}^{\tau,\xi;\theta,\gamma}\right)\right]-\varphi\left(s, X_s^{\tau,\xi;\theta,\gamma}\right), \quad P\text{-a.s.}
\end{equation*}
\end{lemma}
\begin{proof}
By definition, $\overline{Y}_s^{\tau,\xi,\theta,\gamma}:=G_{s, \tau+\alpha}^{\tau,\xi;\theta,\gamma}\left[\varphi\left(\tau+\alpha, X_{\tau+\alpha}^{\tau,\xi;\theta,\gamma}\right)\right]$, $s\in[\tau,\tau+\alpha]$, is the first component of the solution to the GBSDE
\begin{equation*}
\left\{\begin{aligned}
-\dd \overline{Y}_s^{\tau,\xi;\theta,\gamma}={}&f\left(s, X_s^{\tau,\xi;\theta,\gamma}, \overline{Y}_s^{\tau,\xi;\theta,\gamma}, \overline{Z}_s^{\tau,\xi;\theta,\gamma}, \theta_s, \gamma_s\right) \dd s-\overline{Z}_s^{\tau,\xi;\theta,\gamma}\dd W_s \\
& +g\left(s, X_s^{\tau,\xi;\theta,\gamma}, \overline{Y}_s^{\tau,\xi;\theta,\gamma}, \theta_s, \gamma_s\right) \dd K_s^{\tau,\xi;\theta,\gamma}, \quad s \in[\tau, \tau+\alpha], \\
\overline{Y}_{\tau+\alpha}^{\tau,\xi;\theta,\gamma}={}&\varphi\left(\tau+\alpha, X_{\tau+\alpha}^{\tau,\xi;\theta,\gamma}\right).
\end{aligned}\right.
\end{equation*}
It remains to prove $Y_s^{1,\theta,\gamma}=\overline{Y}_s^{\tau,\xi;\theta,\gamma}-\varphi\left(s, X_s^{\tau,\xi;\theta,\gamma}\right)$. Application of the It\^o--Wentzell formula to $\varphi\left(s, X_s^{\tau,\xi;\theta,\gamma}\right)$ shows that this difference satisfies GBSDE (\ref{1}) with terminal value $Y_{\tau+\alpha}^{1,\theta,\gamma}=0$.
\end{proof}

Since $\varphi\in\underline{\mathcal{G}}^{M,N}V(\tau,\xi;\Omega_\tau)$ and, by our convention, $\Omega_\tau=\Omega$, the DPP (Theorem \ref{DPP}) gives, for every $\alpha\in(0,\delta_0]$,
\begin{equation}\label{eq:subsolution-dpp}
\varphi(\tau, \xi)=V(\tau, \xi)=\underset{\mu \in \mathcal{M}_{\tau,\tau+\alpha}}{\essinf}\ \underset{\theta \in \Theta_{\tau,\tau+\alpha}}{\esssup}\  G_{\tau, \tau+\alpha}^{\tau,\xi;\theta,\mu(\theta)}\left[V\left(\tau+\alpha, X_{\tau+\alpha}^{\tau,\xi;\theta,\mu(\theta)}\right)\right], \quad P\text{-a.s.}
\end{equation}
Since $\varphi\in\underline{\mathcal{G}}^{M,N}V(\tau,\xi;\Omega_\tau)$,
it follows from Lemma \ref{function} that, for $M\geq L_f$, $N\geq L_g$,
\begin{equation}\label{eq:subsolution-tangency}
\begin{aligned}
&G_{\tau,\tau+\alpha}^{\tau,\xi;\theta,\mu(\theta)}
\left[\varphi\left(\tau+\alpha,X_{\tau+\alpha}^{\tau,\xi;\theta,\mu(\theta)}\right)\right]-G_{\tau,\tau+\alpha}^{\tau,\xi;\theta,\mu(\theta)}
\left[V\left(\tau+\alpha,X_{\tau+\alpha}^{\tau,\xi;\theta,\mu(\theta)}\right)\right]\\
\geq &\underline{\mathcal E}_{\tau,\tau+\alpha}^{L_f,L_g}
\left[\varphi\left(\tau+\alpha,X_{\tau+\alpha}^{\tau,\xi;\theta,\mu(\theta)}\right)
-V\left(\tau+\alpha,X_{\tau+\alpha}^{\tau,\xi;\theta,\mu(\theta)}\right)\right]\\
\geq &\underline{\mathcal E}_{\tau,\tau+\alpha}^{L_f,L_g}
\left[\inf\limits_{y\in\overline D}\bigl(\varphi(\tau+\alpha,y)-V(\tau+\alpha,y)\bigr)\right]\\
\geq &\underline{\mathcal E}_{\tau,\tau+\alpha}^{M,N}
\left[\inf\limits_{y\in\overline D}\bigl(\varphi(\tau+\alpha,y)-V(\tau+\alpha,y)\bigr)\right]\\
\geq &\varphi(\tau,\xi)-V(\tau,\xi)=0,\qquad P\text{-a.s.}
\end{aligned}
\end{equation}
Consequently, Lemma \ref{relationship} and the relations (\ref{eq:subsolution-dpp}), (\ref{eq:subsolution-tangency}) imply
\begin{equation*}
\begin{aligned}
&\underset{\mu\in\mathcal M_{\tau,\tau+\alpha}}{\essinf}
 \underset{\theta\in\Theta_{\tau,\tau+\alpha}}{\esssup}
 Y_\tau^{1,\theta,\mu(\theta)}\\
={}&\underset{\mu\in\mathcal M_{\tau,\tau+\alpha}}{\essinf}
 \underset{\theta\in\Theta_{\tau,\tau+\alpha}}{\esssup}
 \bigg\{G_{\tau,\tau+\alpha}^{\tau,\xi;\theta,\mu(\theta)}
 \left[\varphi\left(\tau+\alpha,X_{\tau+\alpha}^{\tau,\xi;\theta,\mu(\theta)}\right)\right]
 -\varphi(\tau,\xi)\bigg\}
\geq 0,\quad P\text{-a.s.}
\end{aligned}
\end{equation*}
Therefore, for the strategy $\mu=\widehat{\gamma}$ introduced in (\ref{newstar}), we have
\begin{equation*}
\underset{\theta \in \Theta_{\tau,\tau+\alpha}}{\esssup} Y_\tau^{1, \theta, \widehat{\gamma}(\theta)} \geq 0, \quad P\text{-a.s.}
\end{equation*}
As in the proof of the DPP (see \cite{Li08sicon,Qiu22AAP}), Theorem \ref{mst} allows to choose, for every $\varepsilon>0$, a control process $\theta^{\varepsilon}\in\Theta_{\tau,\tau+\alpha}$ such that, $P$-a.s.,
\begin{equation}\label{sanjiao}
Y_\tau^{1, \theta^\varepsilon, \widehat{\gamma}\left(\theta^\varepsilon\right)} \geq-\varepsilon \alpha.
\end{equation}
For this $\theta^{\varepsilon}\in\Theta_{\tau,\tau+\alpha}$, we define
$\tau^{*}:=\inf\left\{s>\tau:\,|X_s^{\tau,\xi;\theta^{\varepsilon},\widehat{\gamma}(\theta^{\varepsilon})}-\xi|\geq\delta_0\right\}\wedge(\tau+\alpha)$.
Then, for $s\in[\tau,\tau^{*}]$, we have
$|X_s^{\tau,\xi;\theta^{\varepsilon},\widehat{\gamma}(\theta^{\varepsilon})}-\xi|
\leq\delta_0$. Since
$\dd K_s^{\tau,\xi;\theta^{\varepsilon},\widehat{\gamma}(\theta^{\varepsilon})}=0$ on
$\mathcal{D}_n$, it follows that
\begin{equation*}
\begin{aligned}
Y_s^{1,\theta^{\varepsilon},\widehat{\gamma}(\theta^{\varepsilon})}={}&Y_{\tau^*}^{1,\theta^{\varepsilon},\widehat{\gamma}(\theta^{\varepsilon})}+\int_s^{\tau^*} F\left(r, X_r^{\tau,\xi;\theta^{\varepsilon},\widehat{\gamma}(\theta^{\varepsilon})}, Y_r^{1,\theta^{\varepsilon},\widehat{\gamma}(\theta^{\varepsilon})}, Z_r^{1,\theta^{\varepsilon},\widehat{\gamma}(\theta^{\varepsilon})},\theta_r^{\varepsilon},\widehat{\gamma}(\theta^{\varepsilon})_r\right) \dd r\\
&\quad-\int_s^{\tau^*} Z_r^{1,\theta^{\varepsilon},\widehat{\gamma}(\theta^{\varepsilon})} \dd W_r,\qquad s\in[\tau,\tau^{*}].
\end{aligned}
\end{equation*}
Let us consider the following BSDE:
\begin{equation}\label{5.5}
\left\{\begin{aligned}
-\dd Y_s^2 & =\left(C_L\left(\left|Y_s^2\right|+\left|Z_s^2\right|\right)-\frac{1}{2} k\right) \dd s-Z_s^2 \dd W_s,\quad s \in[\tau, \tau+\alpha], \\
Y_{\tau+\alpha}^2 & =0,
\end{aligned}\right.
\end{equation}
whose unique solution is given by
\begin{equation}\label{Y2}
Y_s^2=-\frac{k}{2 C_L}\left(1-e^{C_L(s-(\tau+\alpha))}\right), Z_s^2=0, \quad s \in[\tau, \tau+\alpha].
\end{equation}
We also need the BSDE
\begin{equation}\label{5.7}
\left\{\begin{aligned}
-\dd Y_s^3 & =\left(C_L\left(\left|Y_s^3\right|+\left|Z_s^3\right|\right)-\frac{1}{2} k\right) \dd s-Z_s^3 \dd W_s, \quad s \in[\tau, \tau^*], \\
Y_{\tau^*}^3 & =Y_{\tau^*}^{1,\theta^{\varepsilon},\widehat{\gamma}(\theta^{\varepsilon})}.
\end{aligned}\right.
\end{equation}
In both BSDEs the constant $C_L:=1+\max\{L_f,L_g\}$ controls the Lipschitz dependence of $F$ on
$(y,z)$ and that of $G$ on $y$.
\begin{lemma}\label{23}
We have $Y_{\tau}^{1,\theta^{\varepsilon},\widehat{\gamma}(\theta^{\varepsilon})}\leq Y_\tau^3$ and $|Y_\tau^2-Y_\tau^3|\leq C\alpha^{\frac{3}{2}}$, $P$-a.s. on $\mathcal{D}_n$, where $C$ is independent of $\theta^{\varepsilon}$, $\widehat{\gamma}$, and $\alpha$.
\end{lemma}
\begin{proof}
By (\ref{newstar}), we have
\begin{equation*}
\begin{split}
&F(s,X_s^{\tau,\xi;\theta^{\varepsilon},\widehat{\gamma}(\theta^{\varepsilon})},y,z,\theta^{\varepsilon}_s,\widehat{\gamma}(\theta^{\varepsilon})_s)\\
\leq &F(s,X_s^{\tau,\xi;\theta^{\varepsilon},\widehat{\gamma}(\theta^{\varepsilon})},0,0,\theta^{\varepsilon}_s,\widehat{\gamma}(\theta^{\varepsilon})_s)+C_L\left(|y|+|z|\right)\\
\leq &C_L\left(|y|+|z|\right)-\frac{1}{2}k,\quad\text{on }\mathcal{D}_n.
\end{split}
\end{equation*}
Since $\mathcal{D}_n\in\mathcal{F}_{\tau}$, the BSDE comparison theorem yields
\begin{equation*}
Y_s^{1,\theta^{\varepsilon},\widehat{\gamma}(\theta^{\varepsilon})}\leq Y_s^3,\quad s\in[\tau,\tau^*],\quad P\text{-a.s. on }\mathcal{D}_n.
\end{equation*}
Standard estimates for BSDE (\ref{1}) (see \cite{Pardoux and Peng,Peng1997}) show that, $P$-a.s.,
\begin{equation*}
\begin{aligned}
&E\left[\left|Y_{\tau^*}^{1,\theta^{\varepsilon},\widehat{\gamma}(\theta^{\varepsilon})}\right|^2
\bigg|\mathcal{F}_\tau\right]\\
\leq &CE\left[|\tau+\alpha-\tau^{*}|^2\bigm|\mathcal{F}_\tau\right]
+CE\left[\left(K_{\tau+\alpha}^{\tau,\xi;\theta^{\varepsilon},\widehat{\gamma}(\theta^{\varepsilon})}
-K_{\tau^*}^{\tau,\xi;\theta^{\varepsilon},\widehat{\gamma}(\theta^{\varepsilon})}\right)^2
\bigg|\mathcal{F}_\tau\right]\\
\leq &CE\left[|\tau+\alpha-\tau^{*}|^2\bigm|\mathcal{F}_\tau\right]
+E\left[|\tau+\alpha-\tau^{*}|\bigm|\mathcal{F}_\tau\right].
\end{aligned}
\end{equation*}
Moreover, from (\ref{estimate5}),
\begin{equation}\label{Y1}
\begin{split}
&E\left[\sup\limits_{\tau\leq s \leq \tau+\alpha}|X_s^{\tau,\xi;\theta^{\varepsilon},\widehat{\gamma}(\theta^{\varepsilon})}-\xi|^8\mid\mathcal{F}_\tau\right]\leq C\alpha^4,
\end{split}
\end{equation}
so that Markov's inequality yields
\begin{equation}\label{fangkuang}
P\left\{\tau^*<\tau+\alpha\mid\mathcal{F}_\tau\right\}\leq P\left\{\sup\limits_{\tau\leq s \leq \tau+\alpha}|X_s^{\tau,\xi;\theta^{\varepsilon},\widehat{\gamma}(\theta^{\varepsilon})}-\xi|\geq\delta_0\biggm|\mathcal{F}_\tau\right\}\leq\frac{C}{\delta_0^8} \alpha^4.
\end{equation}
By using (\ref{5.5})--(\ref{5.7}) and combining this estimate with (\ref{Y1}), we obtain
\begin{equation*}
E\left[\left|Y_{\tau^*}^{1,\theta^{\varepsilon},\widehat{\gamma}(\theta^{\varepsilon})}\right|^2\bigg|\mathcal{F}_\tau\right]\leq C\alpha P\left\{\tau^*<\tau+\alpha\mid\mathcal{F}_\tau\right\}\leq\frac{C}{\delta_0^8} \alpha^3.
\end{equation*}
The BSDE stability estimate therefore gives
\begin{equation}\label{5.8}
\begin{aligned}
&\left|Y_\tau^2-Y_\tau^3\right|\\
\leq &C\left(E\left[\left|Y_{\tau^{*}}^2-Y_{\tau^{*}}^3\right|^2
\bigg|\mathcal{F}_\tau\right]\right)^{\frac{1}{2}}
\leq C\left(E\left[\left|Y_{\tau^{*}}^2\right|^2
\bigg|\mathcal{F}_\tau\right]\right)^{\frac{1}{2}}
+C\left(E\left[\left|Y_{\tau^{*}}^3\right|^2
\bigg|\mathcal{F}_\tau\right]\right)^{\frac{1}{2}}\\
\leq &C\frac{k}{2}\left(1-e^{-C_L\alpha}\right)
\left(P\left\{\tau^*<\tau+\alpha\bigm|\mathcal{F}_\tau\right\}\right)^{\frac{1}{2}}
+C\left(E\left[\left|Y_{\tau^{*}}^{1,\theta^{\varepsilon},\widehat{\gamma}(\theta^{\varepsilon})}\right|^2
\bigg|\mathcal{F}_\tau\right]\right)^{\frac{1}{2}}\\
\leq &C\frac{k}{2}\left(1-e^{-C_L\alpha}\right)
\frac{1}{\delta_0^4}\alpha^2
+\frac{C}{\delta_0^4}\alpha^{\frac{3}{2}}
\leq C\alpha^{\frac{3}{2}}.
\end{aligned}
\end{equation}
This estimate holds for all $\alpha \in(0, \delta_0]$ and completes the proof.
\end{proof}
We now continue the proof of Theorem \ref{existence}. From the preceding lemma we have
\begin{equation}\label{123}
Y_\tau^{1,\theta^{\varepsilon},\widehat{\gamma}(\theta^{\varepsilon})}\leq Y_\tau^3\leq Y_{\tau}^2+|Y_{\tau}^3-Y_{\tau}^2|\leq Y_\tau^2 +C\alpha^{\frac{3}{2}},\quad P\text{-a.s. on }\mathcal{D}_n.
\end{equation}
Substituting (\ref{Y2}) into (\ref{123}), we obtain
\begin{equation}\label{1234}
Y_\tau^{1,\theta^{\varepsilon},\widehat{\gamma}(\theta^{\varepsilon})}\leq -\frac{k}{2 C_L}\left(1-e^{-C_L\alpha}\right) +C\alpha^{\frac{3}{2}},\quad P\text{-a.s. on }\mathcal{D}_n,
\end{equation}
and combining (\ref{sanjiao}) with (\ref{1234}) leads to
\begin{equation*}
-\varepsilon\leq\frac{-k}{2C_L}\frac{1-e^{-C_L\alpha}}{\alpha}+C \alpha^{\frac{1}{2}}.
\end{equation*}
Thus, letting $\alpha\downarrow0$ and afterwards $\varepsilon\downarrow0$ yields
$0\leq-k/2$, contradicting $k>0$. Consequently, for every
$\varphi\in\underline{\mathcal{G}}^{M,N}V(\tau,\xi;\Omega_\tau)$ with
$\tau\in\mathcal{T}^{0}$, $\Omega_{\tau}\in\mathcal{F}_\tau$,
$P(\Omega_\tau)>0$, and
$\xi\in L^0(\Omega_\tau,\mathcal{F}_\tau;D)$, the following holds $P$-a.s. on
$\Omega_\tau\cap\mathcal{D}_n$:
\begin{equation}\label{****}
\essliminf_{s \rightarrow \tau^{+}}\left\{-\mathfrak{d}_s \varphi-\mathbb{H}_{-}\left(D^2 \varphi, D \mathfrak{d}_\omega \varphi, D \varphi, \varphi, \mathfrak{d}_\omega \varphi\right)\right\}(s, \xi) \leq 0.
\end{equation}
Moreover, since $\{\omega:\xi(\omega)\in D\}=\bigcup\limits_{n\geq 1}\{\omega:\operatorname{dist}(\xi(\omega),\partial D)>\frac{1}{n}\}$, relation (\ref{****}) also holds on $\{\omega:\xi\in D\}$.

Next, on $\overline{\mathcal{D}}:=\{\xi\in \partial D\}$, we must prove
\begin{equation*}
\essliminf_{s\rightarrow \tau^{+}}\left\{-\mathbb{G}_{-}\left(\mathfrak{d}_s\varphi,\frac{\partial{\varphi}}{\partial\vec n},D^2 \varphi, D \mathfrak{d}_\omega \varphi, D \varphi, \varphi, \mathfrak{d}_\omega \varphi\right)\right\}(s,\xi)\leq 0.
\end{equation*}
Recall that $\mathbb{G}_{-}$ is defined in (\ref{h-}). If this inequality failed, then, as in the case of $\mathcal D_n$, there would exist $k>0$, $\delta'\in(0,1)$, and $\Omega_\tau'\in\mathcal F_\tau$ satisfying $\Omega_\tau'\subset\Omega_\tau$ and $P(\Omega_\tau'\cap\overline{\mathcal D})>0$ such that
\begin{equation*}
\underset{s \in\left[\tau, (\tau+\delta') \wedge T\right]}{\esssup} \bar{G}_0(s, \xi, 0,0)\leq -k<0,\quad P\text{-a.s. on }\Omega_\tau'\cap\overline{\mathcal{D}},
\end{equation*}
where
\begin{equation*}
\bar{G}_0(s, \xi, 0,0)=\underset{\theta \in \Theta_0}{\esssup}\ \underset{\gamma \in \Gamma_0}{\essinf}\ \bar{G}(s, \xi, 0, 0, \theta, \gamma)=\underset{\theta \in \Theta_0}{\esssup}\ \underset{\gamma \in \Gamma_0}{\essinf}\ \left(F(s, \xi, 0, 0, \theta, \gamma)\vee G(s, \xi, 0, \theta, \gamma)\right).
\end{equation*}
Again, we may choose $T_0\in(0,T)$ such that $P(\{\tau<T_0\}\cap\Omega'_\tau\cap\overline{\mathcal D})>0$. Without loss of generality, we assume $\hat\tau=T_0$ and $\Omega_\tau=\Omega'_\tau\cap\overline{\mathcal D}=\Omega$.
Then similarly to the case of $\mathcal D_n$, the predictable measurable selection theorem allows to choose a nonanticipative strategy $\widehat{\gamma}\in\mathcal M_{\tau,\tau+\delta}$ and $\delta>0$ such
that, for all $\tau'\in\mathcal T^\tau$, $\theta\in\Theta_{\tau,\tau+\delta}$, and all
$\xi'\in L^0(\Omega,\mathcal{F}_{\tau'};\overline{D})$,
\begin{equation*}
\bar{G}(\tau', \xi', 0,0, \theta_{\tau'}, \widehat{\gamma}(\theta)_{\tau'}) \leq-\frac{k}{2}
\quad P\text{-a.s. on }\{\tau'\in[\tau,\tau+\delta]\}
\cap O_{\delta,\xi'}(\xi)\cap\overline{\mathcal{D}}.
\end{equation*}
Now let $\alpha\in(0,\delta_0]$, where
$\delta_0=\delta\wedge\delta'\wedge\frac{1}{2}(T-T_0)$. We consider
GBSDE (\ref{1}), and analogously to (\ref{sanjiao}) we can select for any $\varepsilon>0$, a control process
$\theta^{\varepsilon}\in\Theta_{\tau,\tau+\alpha}$ such that
\begin{equation}\label{quan1}
Y_{\tau}^{1,\theta^{\varepsilon},\widehat{\gamma}(\theta^{\varepsilon})}\geq-\varepsilon\alpha,\qquad P\text{-a.s.}
\end{equation}
Let us consider again
$\tau^{*}:=\inf\left\{s>\tau:\,|X_s^{\tau,\xi;\theta^{\varepsilon},\widehat{\gamma}(\theta^{\varepsilon})}-\xi|\geq\delta_0\right\}\wedge(\tau+\alpha)$.
Since
$|X_s^{\tau,\xi;\theta^{\varepsilon},\widehat{\gamma}(\theta^{\varepsilon})}-\xi|
\leq\delta_0$ for $s\in[\tau,\tau^{*}]$, we have on
$\overline{\mathcal{D}}$
\begin{equation*}
F(s,X_s^{\tau,\xi;\theta^{\varepsilon},\widehat{\gamma}(\theta^{\varepsilon})},y,z,\theta_s^{\varepsilon},\widehat{\gamma}(\theta^{\varepsilon})_s)\leq C_L\left(|y|+|z|\right)-\frac{1}{2}k,
\end{equation*}
and
\begin{equation*}
G(s,X_s^{\tau,\xi;\theta^{\varepsilon},\widehat{\gamma}(\theta^{\varepsilon})},y,\theta_s^{\varepsilon},\widehat{\gamma}(\theta^{\varepsilon})_s)\leq C_L|y|-\frac{1}{2}k.
\end{equation*}
Inspired by the case of $\mathcal D_n$, we shall use the following GBSDE:
\begin{equation*}
\left\{\begin{array}{l}
-\dd Y_s^4=\left(C_L\left(\left|Y_s^4\right|+\left|Z_s^4\right|\right)-\frac{1}{2} k\right) \dd s+\left(C_L\left|Y_s^4\right|-\frac{1}{2} k\right) \dd K_s^{\tau,\xi;\theta^{\varepsilon},\widehat{\gamma}(\theta^{\varepsilon})}-Z_s^4 \dd W_s,\\
Y_{\tau^*}^4=Y_{\tau^*}^{1,\theta^{\varepsilon},\widehat{\gamma}(\theta^{\varepsilon})},
\end{array}\right.
\end{equation*}
and thanks to the comparison theorem for GBSDEs, we have
\begin{equation}\label{5.12}
Y_s^{1, \theta^{\varepsilon},\widehat{\gamma}(\theta^{\varepsilon})} \leq Y_s^4,\quad s \in[\tau, \tau^*],\quad P\text{-a.s. on }\overline{\mathcal{D}}.
\end{equation}
We shall also consider the following GBSDE:
\begin{equation}\label{5.13}
\left\{\begin{array}{l}
-\dd Y_s^5=\left(C_L\left(\left|Y_s^5\right|+\left|Z_s^5\right|\right)-\frac{1}{2} k\right) \dd s+\left(C_L\left|Y_s^5\right|-\frac{1}{2} k\right) \dd K_s^{\tau,\xi;\theta^{\varepsilon},\widehat{\gamma}(\theta^{\varepsilon})}-Z_s^5 \dd W_s, \\
Y_{\tau+\alpha}^5=0,
\end{array}\right.
\end{equation}
and recall BSDE (\ref{5.5}) too. Since $C_L\left|Y_s^2\right|-\frac{1}{2} k<0$ for $s\in[\tau,\tau+\alpha]$, the comparison theorem shows that $Y_s^5 \leq Y_s^2$ for $s \in[\tau, \tau+\alpha]$, $P$-a.s. Thus, as in (\ref{fangkuang}) and (\ref{5.8}), we have
\begin{equation}\label{new520}
\begin{aligned}
&\left|Y_\tau^4-Y_\tau^5\right|\\
\leq &C\left(E\left[\left|Y_{\tau^*}^4-Y_{\tau^*}^5\right|^2 \bigg| \mathcal{F}_\tau\right]\right)^{\frac{1}{2}}\\
\leq &C\left(E\left[\left|Y_{\tau^*}^4\right|^2 \bigg| \mathcal{F}_\tau\right]\right)^{\frac{1}{2}}
+C\left(E\left[\left|Y_{\tau^*}^5\right|^2 \bigg| \mathcal{F}_\tau\right]\right)^{\frac{1}{2}}\\
\leq &C\left(E\left[\left|Y_{\tau^*}^{1,\theta^{\varepsilon},\widehat{\gamma}(\theta^{\varepsilon})}\right|^2 \bigg| \mathcal{F}_\tau\right]\right)^{\frac{1}{2}}
+C\left(E\left[\left|Y_{\tau^*}^2\right|^2 \bigg| \mathcal{F}_\tau\right]\right)^{\frac{1}{2}}+C\left(E\left[\left|Y_{\tau^*}^5-Y_{\tau^*}^2\right|^2 \bigg| \mathcal{F}_\tau\right]\right)^{\frac{1}{2}}\\
\leq &C\alpha^{\frac{3}{2}}
+C\left(E\left[\left|Y_{\tau^*}^5-Y_{\tau^*}^2\right|^2 \bigg| \mathcal{F}_\tau\right]\right)^{\frac{1}{2}}.
\end{aligned}
\end{equation}
Here, by definition,
$Y_{\tau^*}^4=Y_{\tau^*}^{1,\theta^\varepsilon,\widehat\gamma(\theta^\varepsilon)}$.
This estimate holds for all $\alpha\in(0,\delta_0]$. Moreover, by standard estimates for GBSDEs (see \cite[Lemma~2.3]{Li and Tang}), Lemma \ref{pro RSDE}, and (\ref{fangkuang}), we obtain
\begin{equation*}
\begin{aligned}
&E\left[\left|Y_{\tau^*}^5-Y_{\tau^*}^2\right|^2 \bigg| \mathcal{F}_\tau\right]
\leq C E\left[\int_{\tau^*}^{\tau+\alpha}
e^{2K_s^{\tau,\xi;\theta^{\varepsilon},\widehat{\gamma}(\theta^{\varepsilon})}}
\left(C_L\left|Y_s^2\right|-\frac{k}{2}\right)^2
\dd K_s^{\tau,\xi;\theta^{\varepsilon},\widehat{\gamma}(\theta^{\varepsilon})}
\bigg|\mathcal{F}_\tau\right]\\
={}&C E\left[\int_{\tau^*}^{\tau+\alpha}
e^{2K_s^{\tau,\xi;\theta^{\varepsilon},\widehat{\gamma}(\theta^{\varepsilon})}}
\frac{k^2}{4}e^{2C_L(s-(\tau+\alpha))}
\dd K_s^{\tau,\xi;\theta^{\varepsilon},\widehat{\gamma}(\theta^{\varepsilon})}
\bigg|\mathcal{F}_\tau\right]\\
\leq &C\frac{k^2}{4}E\left[
\mathbf{1}_{\{\tau^*<\tau+\alpha\}}
e^{2K_{\tau+\alpha}^{\tau,\xi;\theta^{\varepsilon},\widehat{\gamma}(\theta^{\varepsilon})}}
\left(K_{\tau+\alpha}^{\tau,\xi;\theta^{\varepsilon},\widehat{\gamma}(\theta^{\varepsilon})}
-K_{\tau^*}^{\tau,\xi;\theta^{\varepsilon},\widehat{\gamma}(\theta^{\varepsilon})}\right)
\bigg|\mathcal{F}_\tau\right]\\
\leq &C\frac{k^2}{4}
\left(P\left\{\tau^*<\tau+\alpha\bigm|\mathcal{F}_\tau\right\}\right)^{\frac14}
\left(E\left[e^{8K_{\tau+\alpha}^{\tau,\xi;\theta^{\varepsilon},\widehat{\gamma}(\theta^{\varepsilon})}}
\bigg|\mathcal{F}_\tau\right]\right)^{\frac14}
\left(E\left[\left|K_{\tau+\alpha}^{\tau,\xi;\theta^{\varepsilon},\widehat{\gamma}(\theta^{\varepsilon})}
\mathbf{1}_{\{\tau^*<\tau+\alpha\}}\right|^2
\bigg|\mathcal{F}_\tau\right]\right)^{\frac{1}{2}}\\
\leq &Ck^2
\left(P\left\{\tau^*<\tau+\alpha\bigm|\mathcal{F}_\tau\right\}\right)^{\frac14}
\left(\alpha^2P\left\{\tau^*<\tau+\alpha\bigm|\mathcal{F}_\tau\right\}\right)^{\frac14}\\
\leq &Ck^2\alpha^{\frac52}.
\end{aligned}
\end{equation*}
Hence, (\ref{new520}) gives
\begin{equation}\label{5.16}
\left|Y_\tau^4-Y_\tau^5\right| \leq C \alpha^{\frac{3}{2}}+C k \alpha^{\frac{5}{4}},
\end{equation}
and combining (\ref{quan1}), (\ref{5.12}), (\ref{5.13}), and (\ref{5.16}) we obtain
\begin{equation*}
\begin{aligned}
&-\varepsilon\alpha
\leq Y_\tau^{1,\theta^{\varepsilon},\widehat{\gamma}(\theta^{\varepsilon})}
\leq Y_\tau^4
\leq Y_\tau^5+\left|Y_\tau^4-Y_\tau^5\right|
\leq Y_\tau^2+C\alpha^{\frac{3}{2}}+Ck\alpha^{\frac{5}{4}},
\qquad P\text{-a.s. on }\overline{\mathcal{D}}.
\end{aligned}
\end{equation*}
Consequently, (\ref{Y2}) implies
\begin{equation*}
-\varepsilon \leq-\frac{k}{2 C_L} \frac{1-e^{-C_L \alpha}}{\alpha}+C \alpha^{\frac{1}{2}}+C k \alpha^{\frac{1}{4}}.
\end{equation*}
Letting first $\alpha\downarrow0$ and then $\varepsilon\downarrow0$ gives
$0\leq-k/2$, contradicting $k>0$. Thus $V$ is a viscosity subsolution of
BSPDE (\ref{HJBI3}).

\noindent\textbf{Step 2.} We show that $V$ is a viscosity supersolution.

As in Step 1, fix $M\geq L_f$ and $N\geq L_g$. We first work on the event
$\{\xi\in D\}$. By Remark \ref{vs5}, it is sufficient to verify
the following inequality for every
$\varphi\in\overline{\mathcal G}^{M,N}V(\tau,\xi;\Omega_\tau)$, with
$\tau\in\mathcal T^0$, $\Omega_\tau\in\mathcal F_\tau$ of positive
probability and
$\xi\in L^0(\Omega_\tau,\mathcal F_\tau;\overline{D})$, $P$-a.s. on
$\Omega_\tau\cap\{\xi\in D\}$:
\begin{equation}\label{*kuang}
\esslimsup_{s \rightarrow \tau^{+}}\left\{-\mathfrak{d}_s \varphi-\mathbb{H}_{-}\left(D^2 \varphi, D \mathfrak{d}_\omega \varphi, D \varphi, \varphi, \mathfrak{d}_\omega \varphi\right)\right\}(s, \xi) \geq 0.
\end{equation}
As in Step 1, we prove the assertion on
$\Omega_\tau\cap\mathcal{D}_n$.
If (\ref{*kuang}) were not fulfilled on $\Omega_\tau\cap\mathcal D_n$, then there
would exist $\varepsilon>0$, $\Omega_\tau'\in\mathcal F_\tau$, and
$0<\delta'<1/(2n)$, with
$\Omega_\tau'\subset\Omega_\tau\cap\mathcal D_n$ and
$P(\Omega_\tau')>0$, such that
\begin{equation*}
\underset{s \in\left[\tau, (\tau+\delta') \wedge T\right]}{\essinf} F_0(s, \xi, 0,0)\geq \varepsilon>0,\quad P\text{-a.s. on }\Omega_\tau'.
\end{equation*}
Again, we may choose $T_0\in(0,T)$ such that $P(\{\tau<T_0\}\cap\Omega'_\tau)>0$. Without loss of generality, we assume $\hat\tau=T_0$ and $\Omega_\tau=\Omega'_\tau=\Omega$.
Then the predictable
measurable selection theorem and the continuity of $F$ show that there are
$\theta^{\varepsilon}\in\Theta_{\tau,\tau+\delta}$ and
$0<\delta\leq\delta'\wedge(T-T_0)$ such that, for every
$\alpha\in(0,\delta]$,
$\tau'\in\mathcal T_{\tau,\tau+\alpha}$,
$\gamma\in\Gamma_{\tau,\tau+\alpha}$, and
$\xi'\in L^0(\Omega,\mathcal{F}_{\tau'};\overline{D})$,
\begin{equation*}
F(\tau', \xi', 0,0, \theta^{\varepsilon}_{\tau'}, \gamma_{\tau'})
\geq\frac{\varepsilon}{2}
\quad P\text{-a.s. on }O_{\delta,\xi'}(\xi)\cap\mathcal{D}_n.
\end{equation*}
Let
$\delta_0=\delta\wedge\delta'\wedge\frac{1}{2}(T-T_0)$ and
$\alpha\in(0,\delta_0]$. We use the same notation $\theta^\varepsilon$
for its restriction to $[\tau,\tau+\alpha]$.
From Lemma \ref{function} and
$\varphi\in\overline{\mathcal{G}}^{M,N}V(\tau,\xi;\Omega_\tau)$ we deduce
\begin{equation*}
\begin{aligned}
&G_{\tau, \tau+\alpha}^{\tau,\xi;\theta,\mu(\theta)}\left[\varphi\left(\tau+\alpha, X_{\tau+\alpha}^{\tau,\xi;\theta,\mu(\theta)}\right)\right]-G_{\tau, \tau+\alpha}^{\tau,\xi;\theta,\mu(\theta)}\left[V\left(\tau+\alpha, X_{\tau+\alpha}^{\tau,\xi;\theta,\mu(\theta)}\right)\right] \\
\leq &\overline{\mathcal{E}}_{\tau, \tau+\alpha}^{L_f,L_g}\left[\varphi\left(\tau+\alpha, X_{\tau+\alpha}^{\tau,\xi;\theta,\mu(\theta)}\right)-V\left(\tau+\alpha, X_{\tau+\alpha}^{\tau,\xi;\theta,\mu(\theta)}\right)\right] \\
\leq &\overline{\mathcal{E}}_{\tau, \tau+\alpha}^{L_f,L_g}\left[\sup\limits_{y \in \overline{D}}\left(\varphi\left(\tau+\alpha, y\right)-V\left(\tau+\alpha, y\right)\right)\right] \\
\leq &\overline{\mathcal{E}}_{\tau, \tau+\alpha}^{M,N}\left[\sup\limits_{y \in \overline{D}}\left(\varphi\left(\tau+\alpha, y\right)-V\left(\tau+\alpha, y\right)\right)\right] \\
\leq &\varphi(\tau, \xi)-V(\tau, \xi)=0,
\end{aligned}
\end{equation*}
for $M\geq L_f$, $N\geq L_g$, and all
$(\theta,\mu)\in\Theta_{\tau,\tau+\alpha}\times
\mathcal M_{\tau,\tau+\alpha}$.
Thanks to Lemma \ref{relationship} and the assumption $\varphi\in\overline{\mathcal{G}}^{M,N}V(\tau,\xi;\Omega_\tau)$, we have the following inequality, $P$-a.s. on $\Omega$,
\begin{equation*}
\begin{aligned}
&Y_{\tau}^{1,\theta,\mu(\theta)}\\
={}&G_{\tau,\tau+\alpha}^{\tau,\xi;\theta,\mu(\theta)}
\left[\varphi\left(\tau+\alpha,X_{\tau+\alpha}^{\tau,\xi;\theta,\mu(\theta)}\right)\right]
-\varphi(\tau,\xi)\\
\leq &G_{\tau,\tau+\alpha}^{\tau,\xi;\theta,\mu(\theta)}
\left[V\left(\tau+\alpha,X_{\tau+\alpha}^{\tau,\xi;\theta,\mu(\theta)}\right)\right]
-V(\tau,\xi).
\end{aligned}
\end{equation*}
Then the DPP (Theorem \ref{DPP}) implies
\begin{equation}\label{5.25}
\underset{\mu \in \mathcal{M}_{\tau,\tau+\alpha}}{\essinf} \underset{\theta \in \Theta_{\tau,\tau+\alpha}}{\esssup} Y_\tau^{1, \theta,\mu(\theta)} \leq 0,\qquad P\text{-a.s. on }\Omega_\tau'.
\end{equation}
Put
\begin{equation*}
 A_\mu:=\esssup_{\theta\in\Theta_{\tau,\tau+\alpha}}
 Y_\tau^{1,\theta,\mu(\theta)},\qquad
 a:=\essinf_{\mu\in\mathcal M_{\tau,\tau+\alpha}}A_\mu.
\end{equation*}
Pasting on $\{A_{\mu_1}\leq A_{\mu_2}\}\in\mathcal F_\tau$ shows that
this family is downward directed. Choose $A_{\mu_n}\downarrow a$ and set
\begin{equation*}
 B_n:=\{A_{\mu_n}\leq a+\varepsilon'\}\setminus\bigcup_{i<n}B_i,
 \qquad \mu^{\varepsilon'}(\theta):=\sum_{n\geq1}\mathbf1_{B_n}\mu_n(\theta).
\end{equation*}
By nonanticipativity and (\ref{5.25}), $\mu^{\varepsilon'}\in\mathcal M_{\tau,\tau+\alpha}$ and
\begin{equation*}
 Y_\tau^{1,\theta,\mu^{\varepsilon'}(\theta)}
 \leq A_{\mu^{\varepsilon'}}\leq a+\varepsilon'\leq\varepsilon',
 \qquad\theta\in\Theta_{\tau,\tau+\alpha},\quad P\text{-a.s. on }\Omega_\tau'.
\end{equation*}
For this strategy, we define
\begin{equation*}
 \tau^{*}:=\inf\left\{s>\tau:\
 |X_s^{\tau,\xi;\theta^{\varepsilon},
 \mu^{\varepsilon'}(\theta^{\varepsilon})}-\xi|\geq\delta_0\right\}
 \wedge(\tau+\alpha).
\end{equation*}
Since
$\dd K_s^{\tau,\xi;\theta^{\varepsilon},
\mu^{\varepsilon'}(\theta^{\varepsilon})}=0$ on $\mathcal D_n$ for
$s\in[\tau,\tau^*]$, we have, for all such $s$,
\begin{equation*}
\begin{split}
Y_s^{1,\theta^{\varepsilon},\mu^{\varepsilon'}(\theta^{\varepsilon})}={}&Y_{\tau^*}^{1,\theta^{\varepsilon},\mu^{\varepsilon'}(\theta^{\varepsilon})}+\int_s^{\tau^*} F\left(r, X_r^{\tau,\xi;\theta^{\varepsilon},\mu^{\varepsilon'}(\theta^{\varepsilon})}, Y_r^{1,\theta^{\varepsilon},\mu^{\varepsilon'}(\theta^{\varepsilon})}, Z_r^{1,\theta^{\varepsilon},\mu^{\varepsilon'}(\theta^{\varepsilon})},\theta_r^{\varepsilon},\mu^{\varepsilon'}(\theta^{\varepsilon})_r\right) \dd r\\
&-\int_s^{\tau^*} Z_r^{1,\theta^{\varepsilon},\mu^{\varepsilon'}(\theta^{\varepsilon})} \dd W_r.
\end{split}
\end{equation*}
Let us consider the following BSDE:
\begin{equation*}
\left\{\begin{aligned}
-\dd \hat{Y}_s^2 & =\left(-C_L\left(\left|\hat{Y}_s^2\right|+\left|\hat{Z}_s^2\right|\right)+\frac{1}{2}\varepsilon\right) \dd s-\hat{Z}_s^2 \dd W_s,\quad s\in[\tau,\tau+\alpha], \\
\hat{Y}_{\tau+\alpha}^2 & =0,
\end{aligned}\right.
\end{equation*}
whose unique solution is given by
\begin{equation}\label{5.18}
\hat{Y}_s^2=\frac{\varepsilon}{2 C_L}\left(1-e^{C_L(s-(\tau+\alpha))}\right),\qquad \hat{Z}_s^2=0,\quad s \in[\tau, \tau+\alpha].
\end{equation}
We also need the BSDE
\begin{equation*}
\left\{\begin{aligned}
-\dd \hat{Y}_s^3 & =\left(-C_L\left(\left|\hat{Y}_s^3\right|+\left|\hat{Z}_s^3\right|\right)+\frac{1}{2} \varepsilon\right) \dd s-\hat{Z}_s^3 \dd W_s, \quad s \in\left[\tau, \tau^*\right], \\
\hat{Y}_{\tau^*}^3 & =Y_{\tau^*}^{1, \theta^\varepsilon, \mu^{\varepsilon'}\left(\theta^\varepsilon\right)}.
\end{aligned}\right.
\end{equation*}
Note that on $\mathcal{D}_n$, for $s\in[\tau,\tau^*]$,
\begin{equation*}
\begin{split}
&F(s,X_s^{\tau,\xi;\theta^{\varepsilon},\mu^{\varepsilon'}(\theta^{\varepsilon})},y,z,\theta_s^{\varepsilon},\mu^{\varepsilon'}(\theta^{\varepsilon})_s)\\
\geq &F(s,X_s^{\tau,\xi;\theta^{\varepsilon},\mu^{\varepsilon'}(\theta^{\varepsilon})},0,0,\theta_s^{\varepsilon},\mu^{\varepsilon'}(\theta^{\varepsilon})_s)-C_L\left(|y|+|z|\right)\\
\geq &-C_L\left(|y|+|z|\right)+\frac{1}{2}\varepsilon.
\end{split}
\end{equation*}
By arguing as in the proof of Lemma \ref{23}, we see that $Y_\tau^{1, \theta^\varepsilon, \mu^{\varepsilon'}\left(\theta^\varepsilon\right)} \geq \hat{Y}_\tau^3$ and $\left|\hat{Y}_\tau^2-\hat{Y}_\tau^3\right| \leq C \alpha^{\frac{3}{2}}$, $P$-a.s. on $\mathcal{D}_n$. Together with $Y_\tau^{1, \theta^\varepsilon, \mu^{\varepsilon'}(\theta^\varepsilon)} \leq \varepsilon'$ this gives
\begin{equation*}
\begin{aligned}
&\varepsilon'
\geq Y_\tau^{1,\theta^{\varepsilon},\mu^{\varepsilon'}(\theta^{\varepsilon})}
\geq \hat{Y}_\tau^3
\geq \hat{Y}_\tau^2-\left|\hat{Y}_\tau^2-\hat{Y}_\tau^3\right|
\geq \hat{Y}_\tau^2-C\alpha^{\frac{3}{2}},
\qquad P\text{-a.s. on }\mathcal{D}_n.
\end{aligned}
\end{equation*}
Set $\varepsilon'=\alpha^2$. Then (\ref{5.18}) gives
\begin{equation*}
\alpha \geq \frac{\varepsilon}{2 C_L} \frac{1-e^{-C_L \alpha}}{\alpha}-C \alpha^{\frac{1}{2}}.
\end{equation*}
Hence, letting $\alpha\downarrow0$ leads to the contradiction $0\geq\varepsilon/2$.
This proves (\ref{*kuang}) on $\mathcal{D}_n$, and so on
$\{\xi\in D\}$.

Next, on $\{\xi\in\partial D\}$, we must prove
\begin{equation*}
\esslimsup_{s\rightarrow \tau^{+}}\left\{-\mathbb{G}_{+}\left(\mathfrak{d}_s\varphi,\frac{\partial{\varphi}}{\partial\vec n},D^2 \varphi, D \mathfrak{d}_\omega \varphi, D \varphi, \varphi, \mathfrak{d}_\omega \varphi\right)\right\}(s,\xi)\geq 0.
\end{equation*}
If this were not satisfied, then there would exist $\varepsilon>0$, $\delta'>0$,
and $\Omega_\tau'\in\mathcal F_\tau$, with
$\Omega_\tau'\subset\Omega_\tau$ and
$P(\Omega_\tau'\cap\overline{\mathcal D})>0$, such that
\begin{equation*}
\underset{s \in\left[\tau, (\tau+\delta') \wedge T\right]}{\essinf}  \bar{\bar{G}}_0(s, \xi, 0,0)\geq \varepsilon>0,
\qquad P\text{-a.s. on }\Omega_\tau'\cap\overline{\mathcal D},
\end{equation*}
where
\begin{equation*}
\bar{\bar{G}}_0(s, \xi, 0,0)=\underset{\theta \in \Theta_0}{\esssup}\ \underset{\gamma \in \Gamma_0}{\essinf}\ \left(F(s, \xi, 0, 0, \theta, \gamma)\wedge G(s, \xi, 0, \theta, \gamma)\right).
\end{equation*}
Again, we may choose $T_0\in(0,T)$ such that $P(\{\tau<T_0\}\cap\Omega'_\tau\cap\overline{\mathcal D})>0$. Without loss of generality, we assume $\hat\tau=T_0$ and $\Omega_\tau=\Omega'_\tau\cap\overline{\mathcal D}=\Omega$.
By the same argument as further up, the predictable measurable selection theorem gives
$\theta^\varepsilon\in\Theta_{\tau,\tau+\delta}$ and
$0<\delta\leq\delta'\wedge(T-T_0)$ such that, for every
$\alpha\in(0,\delta]$,
$\tau'\in\mathcal T_{\tau,\tau+\alpha}$,
$\gamma\in\Gamma_{\tau,\tau+\alpha}$, and
$\xi'\in L^0(\Omega,\mathcal F_{\tau'};\overline{D})$,
\begin{equation}\label{new*}
F\left(\tau', \xi', 0,0, \theta^{\varepsilon}_{\tau'},
\gamma_{\tau'}\right)\wedge
G\left(\tau', \xi', 0, \theta^{\varepsilon}_{\tau'},
\gamma_{\tau'}\right)
\geq\frac{\varepsilon}{2}
\quad P\text{-a.s. on }
O_{\delta,\xi'}(\xi)\cap\overline{\mathcal{D}}.
\end{equation}
Let $\delta_0=\delta\wedge\delta'\wedge\frac{1}{2}(T-T_0)$ and
$\alpha\in(0,\delta_0]$. We use the same notation $\theta^\varepsilon$
for its restriction to $[\tau,\tau+\alpha]$. Repeating the tangency and DPP
argument leading to (\ref{5.25}) on this localized event, we see that, for every
$\varepsilon'>0$, there exists
$\mu^{\varepsilon'}\in\mathcal{M}_{\tau,\tau+\alpha}$ such that
$Y_\tau^{1, \theta^\varepsilon, \mu^{\varepsilon'}(\theta^\varepsilon)}
\leq \varepsilon'$.
For this strategy, we define
\begin{equation*}
\tau^{*}:=\inf\left\{s>\tau:\
|X_s^{\tau,\xi;\theta^{\varepsilon},\mu^{\varepsilon'}(\theta^{\varepsilon})}-\xi|
\geq\delta_0\right\}\wedge(\tau+\alpha).
\end{equation*}
Then (\ref{new*}) allows to derive the following inequalities on
$\overline{\mathcal{D}}$, for $s\in[\tau,\tau^*]$:
\begin{equation*}
\begin{split}
F(s,X_s^{\tau,\xi;\theta^{\varepsilon},\mu^{\varepsilon'}(\theta^{\varepsilon})},y,z,\theta_s^\varepsilon,\mu^{\varepsilon'}(\theta^{\varepsilon})_s)\geq &-C_L\left(|y|+|z|\right)+\frac{1}{2}\varepsilon,\\
G(s,X_s^{\tau,\xi;\theta^{\varepsilon},\mu^{\varepsilon'}(\theta^{\varepsilon})},y,\theta_s^\varepsilon,\mu^{\varepsilon'}(\theta^{\varepsilon})_s)\geq &-C_L |y|+\frac{1}{2}\varepsilon.
\end{split}
\end{equation*}
Consequently, the comparison theorem for GBSDEs gives
$Y_s^{1,\theta^{\varepsilon},\mu^{\varepsilon'}(\theta^{\varepsilon})}
\geq\hat{Y}_s^4$ for $s\in[\tau,\tau^*]$, $P$-a.s. on
$\overline{\mathcal{D}}$, where $\hat{Y}^4$ is defined by the following
GBSDE:
\begin{equation*}
\left\{\begin{array}{l}
-\dd \hat{Y}_s^4=\left(-C_L\left(\left|\hat{Y}_s^4\right|+\left|\hat{Z}_s^4\right|\right)+\frac{1}{2} \varepsilon\right) \dd s+\left(-C_L\left|\hat{Y}_s^4\right|+\frac{1}{2} \varepsilon\right) \dd K_s^{\tau,\xi;\theta^{\varepsilon},\mu^{\varepsilon'}(\theta^{\varepsilon})}-\hat{Z}_s^4 \dd W_s, \\
\hat{Y}_{\tau^*}^4=Y_{\tau^*}^{1,\theta^{\varepsilon},\mu^{\varepsilon'}(\theta^{\varepsilon})}.
\end{array}\right.
\end{equation*}
We want to compare the preceding GBSDE with the following one:
\begin{equation*}
\left\{\begin{array}{l}
-\dd \hat{Y}_s^5=\left(-C_L\left(\left|\hat{Y}_s^5\right|+\left|\hat{Z}_s^5\right|\right)+\frac{1}{2} \varepsilon\right) \dd s+\left(-C_L\left|\hat{Y}_s^5\right|+\frac{1}{2} \varepsilon\right) \dd K_s^{\tau,\xi;\theta^{\varepsilon},\mu^{\varepsilon'}(\theta^{\varepsilon})}-\hat{Z}_s^5 \dd W_s, \\
\hat{Y}_{\tau+\alpha}^5=0.
\end{array}\right.
\end{equation*}
Since $-C_L\left|\hat{Y}_s^2\right|+\frac{1}{2} \varepsilon>0$, the comparison theorem for GBSDEs gives $\hat{Y}_s^5 \geq \hat{Y}_s^2$ for $s \in[\tau, \tau+\alpha]$, $P$-a.s. On the other hand, arguing as in (\ref{5.16}) we obtain
\begin{equation*}
\left|\hat{Y}_\tau^4-\hat{Y}_\tau^5\right| \leq C \alpha^{\frac{3}{2}}+C \varepsilon \alpha^{\frac{5}{4}}.
\end{equation*}
Hence,
\begin{equation*}
\begin{aligned}
\varepsilon'
\geq Y_\tau^{1,\theta^{\varepsilon},\mu^{\varepsilon'}(\theta^{\varepsilon})}
\geq \hat{Y}_\tau^4
\geq \hat{Y}_\tau^5-\left|\hat{Y}_\tau^4-\hat{Y}_\tau^5\right|
\geq \hat{Y}_\tau^2-C\alpha^{\frac{3}{2}}-C\varepsilon\alpha^{\frac{5}{4}},
\qquad P\text{-a.s. on }\overline{\mathcal{D}}.
\end{aligned}
\end{equation*}
Therefore, substituting (\ref{5.18}) for $s=\tau$ and choosing $\varepsilon'=\alpha^2$, we have
\begin{equation*}
\alpha \geq \frac{\varepsilon}{2 C_L} \frac{1-e^{-C_L \alpha}}{\alpha}-C \alpha^{\frac{1}{2}}-C \varepsilon  \alpha^{\frac{1}{4}}.
\end{equation*}
Finally, taking the limit $\alpha\downarrow0$ leads to a contradiction and completes the proof.
\end{proof}

\section{The uniqueness result in a one-dimensional linear case}\label{sec:uniqueness}

In this section, we prove a comparison theorem for a one-dimensional linear
stochastic HJBI equation with a nonlinear Neumann boundary condition. This
gives the uniqueness of bounded $(L_y,0)$-viscosity solutions. Assumption
\ref{HU} is imposed only in this section, in addition to Assumption
\ref{HSDG}, and Theorem \ref{existence} still applies. The controls enter the
running payoff but not the state equation.

\subsection{Additional assumptions and the main result}

We consider the one-dimensional case. For notational simplicity,
we take $D=(0,1)$ and $\overline{D}=[0,1]$, although the arguments below remain
valid for an arbitrary bounded interval. Consistent with the sign convention in
the reflected equation (\ref{eq RSDE}), we take the inward normals to be
\begin{equation}\label{uniq-normal}
 \vec n(0)=1,\qquad \vec n(1)=-1.
\end{equation}
Let $W=(W^0,W^1)$, where $W^0$ and $W^1$ are two independent one-dimensional
Brownian motions, and let $\{\mathcal F_t^0\}_{t\in[0,T]}$ be the augmented
filtration generated by $W^0$.

\begin{assumption}\label{HU}
We impose the following conditions.

\noindent{\rm(i)} The state coefficients are independent of
$(\theta,\gamma)$, and the reflected state equation is as follows
\begin{equation*}
\left\{\begin{aligned}
\dd X_s={}&b(s,X_s)\dd s+\sigma^0(s,X_s)\dd W_s^0
       +\sigma^1(s,X_s)\dd W_s^1+\vec n(X_s)\dd K_s,\\
\dd K_s={}&\mathbf{1}_{\{0,1\}}(X_s)\dd K_s,\qquad K\text{ is nondecreasing},\\
X_t={}&x\in[0,1].
\end{aligned}\right.
\end{equation*}

\noindent{\rm(ii)} The drift $b$ is $\{\mathcal F_t^0\}$-progressively
measurable, uniformly bounded and uniformly Lipschitz continuous in $x$. The
functions $\sigma^0,\sigma^1$ are deterministic, belong to
$C_b^3([0,T]\times\mathbb R)$, and are uniformly Lipschitz continuous in $x$.
There is a constant $\kappa>0$ such that
\begin{equation*}
 |\sigma^1(t,x)|^2\geq\kappa,\qquad (t,x)\in[0,T]\times[0,1].
\end{equation*}

\noindent{\rm(iii)} The generators have the form
\begin{equation}\label{uniq-linear-data}
 f(t,x,y,z,\theta,\gamma)=L_y y+f_0(t,x,\theta,\gamma),\qquad
 g(t,x,y,\theta,\gamma)=g_0(t,x).
\end{equation}
Here $L_y\geq0$ is a deterministic constant and $f_0$ is
$\{\mathcal F_t^0\}$-progressively measurable, bounded,
continuous in $(x,\theta,\gamma)$ and uniformly Lipschitz continuous in $x$.
The random field $g_0$ is $\{\mathcal F_t^0\}$-progressively measurable,
continuous in $x$, and belongs to
$\mathcal C_{\mathcal F}^3([0,T]\times\mathbb R)$.

\noindent{\rm(iv)} The terminal condition $\Phi$ is
$\mathcal F_T^0$-measurable, bounded and uniformly Lipschitz continuous in
$x$.
\end{assumption}

Putting
\begin{equation*}
 a(t,x):=\frac{1}{2}\left(|\sigma^0(t,x)|^2+|\sigma^1(t,x)|^2\right),\qquad
 F_0(t,x):=\underset{\theta\in\Theta_0}{\esssup}\
 \underset{\gamma\in\Gamma_0}{\essinf}
 f_0(t,x,\theta,\gamma),
\end{equation*}
the lower Hamiltonian in (\ref{h-}) is affine and takes the form
\begin{equation*}
 \mathbb H_-(t,x,A,B,p,y,z)
 =a(t,x)A+\sigma^0(t,x)B^0+\sigma^1(t,x)B^1
   +b(t,x)p+L_y y+F_0(t,x),
\end{equation*}
where $B=(B^0,B^1)$. Thus the equation is linear after optimizing over the
controls.

\begin{theorem}\label{uniq-main}
Suppose that the Assumptions \ref{HSDG} and \ref{HU} hold. If
$u\in\mathcal S^\infty(C([0,1]))$ is an
$(L_y,0)$-viscosity subsolution and
$v\in\mathcal S^\infty(C([0,1]))$ is an $(L_y,0)$-viscosity supersolution of
(\ref{HJBI3}) with terminal data $\Phi$, then
\begin{equation*}
 u(t,x)\leq v(t,x),\qquad (t,x)\in[0,T]\times[0,1],\quad P\text{-a.s.}
\end{equation*}
The exceptional set can be chosen independently of $(t,x)$. Consequently,
(\ref{HJBI3}) has at most one bounded $(L_y,0)$-viscosity solution in
$\mathcal S^\infty(C([0,1]))$. Under (\ref{uniq-linear-data}), $f$ is
$L_y$-Lipschitz in $y$ and $g$ is independent of $y$, so $L_f=L_y$ and
$L_g=0$. Proposition \ref{bound} and Theorem \ref{existence} therefore show
that the lower value function $V$ is a bounded $(L_y,0)$-viscosity solution
and hence is the unique solution in this class. If the Isaacs condition holds,
then $V=U$ and the stochastic differential game has a value.
\end{theorem}

The proof of Theorem \ref{uniq-main} proceeds in three steps. First, we normalize the
nonhomogeneous boundary condition and the zeroth-order term. Second, we prove
a weak comparison theorem between a viscosity semisolution and a smooth
comparison field. Third, we construct smooth upper and lower comparison
fields using deterministic uniformly parabolic systems.

\subsection{Nonlinear expectations with \texorpdfstring{$N=0$}{N=0} and normalization}

When $N=0$, the local-time terms in (\ref{01})--(\ref{02}) disappear. More
precisely, for $M\geq0$, stopping times $\tau\leq\rho$, and
$\xi\in L^2(\Omega,\mathcal F_\rho,P)$, (\ref{01}) and (\ref{02}) take the form
\begin{align}
 \overline Y_s
 &=\xi+\int_s^\rho M\bigl(|\overline Y_r|+|\overline Z_r|\bigr)\dd r
       -\int_s^\rho\overline Z_r\dd W_r,\label{uniq-upper-bsde}\\
 \underline Y_s
 &=\xi-\int_s^\rho M\bigl(|\underline Y_r|+|\underline Z_r|\bigr)\dd r
       -\int_s^\rho\underline Z_r\dd W_r,
 \qquad \tau\leq s\leq\rho.\label{uniq-lower-bsde}
\end{align}
Then
\begin{equation*}
 \overline{\mathcal E}_{\tau,\rho}^{M,0}[\xi]:=\overline Y_\tau,
 \qquad
 \underline{\mathcal E}_{\tau,\rho}^{M,0}[\xi]:=\underline Y_\tau.
\end{equation*}
Observe that these operators are independent of $(x,\theta,\gamma)$. Throughout this section, ``subsolution'' and ``supersolution'' refer to
bounded viscosity semisolutions with zero $\dd K_s$-coefficient and a finite
$\dd s$-coefficient.

\begin{lemma}\label{uniq-E-properties}
Fix $M\geq0$. The operators $\overline{\mathcal E}^{M,0}$ and
$\underline{\mathcal E}^{M,0}$ are monotone and time-consistent, and
\begin{equation}\label{uniq-E-duality}
 \underline{\mathcal E}_{\tau,\rho}^{M,0}[\xi]
 =-\overline{\mathcal E}_{\tau,\rho}^{M,0}[-\xi].
\end{equation}
Moreover, for $\xi_1,\xi_2\in L^2(\Omega,\mathcal F_\rho; \mathbb R)$,
\begin{equation}\label{uniq-E-superadditive}
 \underline{\mathcal E}_{\tau,\rho}^{M,0}[\xi_1+\xi_2]
 \geq \underline{\mathcal E}_{\tau,\rho}^{M,0}[\xi_1]
      +\underline{\mathcal E}_{\tau,\rho}^{M,0}[\xi_2].
\end{equation}
If $\eta\geq0$ is $\mathcal F_\tau$-measurable and $\tau\leq\rho$, then
\begin{equation}\label{uniq-E-explicit-lower}
 \underline{\mathcal E}_{\tau,\rho}^{M,0}
 [\eta e^{M(\rho-\tau)}]=\eta.
\end{equation}
If, in addition, $\rho$ is $\mathcal F_\tau$-measurable (in particular, if
$\rho$ and $\tau$ are deterministic), then
\begin{equation}\label{uniq-E-explicit-upper}
 \overline{\mathcal E}_{\tau,\rho}^{M,0}[\eta]
 =\eta e^{M(\rho-\tau)}.
\end{equation}
\end{lemma}

\begin{proof}
The existence and the uniqueness of the solution, together with monotonicity and
time consistency, are standard properties of Lipschitz BSDEs. Replacing
$(Y,Z)$ by $(-Y,-Z)$ in the lower equation gives (\ref{uniq-E-duality}). For two solutions of the lower
equation, the triangular inequality provides
\begin{equation*}
 -M(|Y^1+Y^2|+|Z^1+Z^2|)
 \geq-\sum_{i=1}^2M(|Y^i|+|Z^i|).
\end{equation*}
Thus, we get from the comparison theorem for BSDEs (\ref{uniq-E-superadditive}). The pair
$Y_s=\eta e^{M(s-\tau)}, Z_s=0$, $\tau\leq s\leq\rho$, is, obviously, a solution of (\ref{uniq-lower-bsde}) and proves
(\ref{uniq-E-explicit-lower}). If $\rho$ is $\mathcal F_\tau$-measurable,
then $\overline Y_s=\eta e^{M(\rho-s)},\overline Z_s=0$ is a solution of (\ref{uniq-upper-bsde}) and we get
(\ref{uniq-E-explicit-upper}).
\end{proof}

Let $w\in\mathcal S^\infty(0,T;C([0,1]))$ be a viscosity subsolution
or supersolution of the original problem
\begin{equation*}
\left\{\begin{aligned}
&-\mathfrak d_tw-aD^2w-\sigma^0D\mathfrak d_{\omega^0}w
 -\sigma^1D\mathfrak d_{\omega^1}w-bDw-L_yw-F_0=0,\\
&\partial_{\vec n}w+g_0=0,\qquad w(T)=\Phi.
\end{aligned}\right.
\end{equation*}
We next make the boundary condition homogeneous. We put
$g_L(t):=g_0(t,0)$ and $g_R(t):=g_0(t,1)$, and define
\begin{equation*}
 \ell(t,x):=-g_L(t)x+\frac{1}{2}\big(g_L(t)+g_R(t)\big)x^2.
\end{equation*}
Then
\begin{equation*}
 D\ell(t,x)=-g_L(t)+\big(g_L(t)+g_R(t)\big)x.
\end{equation*}
By (\ref{uniq-normal}),
\begin{equation*}
 \frac{\partial\ell}{\partial\vec n}(t,0)=-g_L(t),\qquad
 \frac{\partial\ell}{\partial\vec n}(t,1)=-g_R(t).
\end{equation*}
Consequently, if $w$ satisfies
$\partial w/\partial\vec n+g_0=0$, then $\overline w:=w-\ell$ satisfies the
homogeneous Neumann condition.

Put
\begin{equation*}
 F^\ell:=F_0+\mathfrak d_t\ell+aD^2\ell
 +\sigma^0D\mathfrak d_{\omega^0}\ell
 +\sigma^1D\mathfrak d_{\omega^1}\ell+bD\ell+L_y\ell.
\end{equation*}
For notational simplicity, the arguments $(t,x)$ of $F^\ell$ and of the
coefficients in this formula are omitted.
Since $\ell$ is $\{\mathcal F_t^0\}$-adapted,
$\mathfrak d_{\omega^1}\ell=0$. The field $\overline w=w-\ell$ satisfies
\begin{equation*}
 -\mathfrak d_t\overline w-aD^2\overline w
 -\sigma^0D\mathfrak d_{\omega^0}\overline w
 -\sigma^1D\mathfrak d_{\omega^1}\overline w
 -bD\overline w-L_y\overline w-F^\ell=0.
\end{equation*}

Fix $\lambda>0$, put $\lambda_*:=\lambda+L_y$, and define
\begin{equation}\label{uniq-exponential}
 \widetilde w(t,x)=e^{\lambda_* t}\overline w(t,x),\qquad
 \widetilde F=e^{\lambda_* t}F^\ell,\qquad
 \widetilde\Phi=e^{\lambda_* T}(\Phi-\ell(T)).
\end{equation}
Then the normalized equation is
\begin{equation}\label{uniq-normalized-equation}
 \mathcal R[\widetilde w]:=
 -\mathfrak d_t\widetilde w-aD^2\widetilde w
 -\sigma^0D\mathfrak d_{\omega^0}\widetilde w
 -\sigma^1D\mathfrak d_{\omega^1}\widetilde w
 -bD\widetilde w+\lambda\widetilde w-\widetilde F=0,
\end{equation}
with
\begin{equation}\label{uniq-normalized-boundary}
 \frac{\partial\widetilde w}{\partial\vec n}=0,\qquad
 \widetilde w(T)=\widetilde\Phi.
\end{equation}

Here normalization means that the boundary condition is homogeneous and
the zeroth-order coefficient is $\lambda>0$.

Due to Lemma \ref{mon}, the test-field transformations used here, namely
the lifting and exponential transformations, preserve viscosity subsolutions
and supersolutions. The transformed threshold is
$\widehat M_0:=L_y+\lambda_*$, while $N=0$ remains unchanged.

For the normalized equation, put
\begin{equation*}
 \widetilde{\mathbb H}_{-,\lambda}(A,B,p,y,z)
 :=aA+\sigma^0B^0+\sigma^1B^1+bp-\lambda y+\widetilde F,
\end{equation*}
and let $\widetilde{\mathbb G}_{\pm,\lambda}$ be the corresponding boundary
operators with boundary generator equal to zero. For
$\phi\in\mathcal C_{\mathcal F}^3([0,T]\times\mathbb R)$, we have
\begin{equation*}
 \mathfrak d_t\phi+
 \widetilde{\mathbb H}_{-,\lambda}
 (D^2\phi,D\mathfrak d_\omega\phi,D\phi,\phi,
 \mathfrak d_\omega\phi)
 =-\mathcal R[\phi].
\end{equation*}
Hence we obtain
\begin{equation}\label{uniq-boundary-residual}
 -\widetilde{\mathbb G}_{-,\lambda}(\phi)
 =\mathcal R[\phi]\wedge\left(-\frac{\partial\phi}{\partial\vec n}\right),
 \qquad
 -\widetilde{\mathbb G}_{+,\lambda}(\phi)
 =\mathcal R[\phi]\vee\left(-\frac{\partial\phi}{\partial\vec n}\right).
\end{equation}
Here the $B$-argument in $\widetilde{\mathbb G}_{\pm,\lambda}$ is
$B=(D\mathfrak d_{\omega^0}\phi,D\mathfrak d_{\omega^1}\phi)$.

\subsection{Weak comparison with smooth random fields}

We first compare an arbitrary viscosity semisolution with a smooth random
field. A random field $\varphi$ will
be called a smooth upper comparison field if it is adapted and continuous on
$[0,T]\times[0,1]$, uniformly Lipschitz continuous in $x$, satisfies
$\varphi^-\in\mathcal S^\infty(C([0,1]))$, and has a finite deterministic
partition $0=s_0<s_1<\cdots<s_{N_p}=T$
such that $\varphi$ belongs to $\mathcal C_{\mathcal F}^3$ on every regular
time interval $(s_{j-1},s_j)$ and
\begin{equation}\label{uniq-upper-field}
 \varphi(T,x)\geq\widetilde\Phi(x),\qquad
 \mathcal R[\varphi](t,x)\geq0,\qquad
 \frac{\partial\varphi}{\partial\vec n}(t,0)
 =\frac{\partial\varphi}{\partial\vec n}(t,1)=0.
\end{equation}
Here the residual inequality is imposed on
$(s_{j-1},s_j)\times(0,1)$ for every $j$. A smooth lower comparison field is
defined by reversing the terminal and residual
inequalities in (\ref{uniq-upper-field}) and requiring
$\varphi^+\in\mathcal S^\infty(C([0,1]))$.

\begin{lemma}\label{uniq-strictification}
Let $\varphi$ be a smooth upper comparison field and let
$[r,S]\subset[0,T]$ be a deterministic interval whose interior is contained
in one of the regular time intervals of $\varphi$. Fix $q\in C_b^\infty(\mathbb R)$ so
that $q(x)=(x-\frac{1}{2})^2$ on a neighbourhood of $[0,1]$, and put
\begin{equation}\label{uniq-Cq}
 C_q:=2\|a\|_\infty+\|b\|_\infty+\frac{\lambda}{4}.
\end{equation}
For every $\eta>0$ and every
\begin{equation*}
 0<\varepsilon\leq\frac{\eta}{2C_q},
\end{equation*}
the field
\begin{equation}\label{uniq-direct-strict-field}
 \psi^{\eta,\varepsilon}(t,x)
 :=\varphi(t,x)+\eta(S-t)+\varepsilon q(x),
 \qquad (t,x)\in[r,S]\times[0,1],
\end{equation}
has the regularity required in Definition \ref{testfunction} on every compact
subinterval of $(r,S)$ and satisfies
\begin{equation}\label{uniq-strict-field}
 \psi^{\eta,\varepsilon}(S,x)\geq\varphi(S,x),\qquad
 \mathcal R[\psi^{\eta,\varepsilon}](t,x)\geq\frac{\eta}{2},
 \qquad
 \frac{\partial\psi^{\eta,\varepsilon}}{\partial\vec n}(t,i)
 =-\varepsilon,\quad i=0,1.
\end{equation}
Moreover,
\begin{equation}\label{uniq-strict-uniform-error}
 \|\psi^{\eta,\varepsilon}-\varphi\|_
 {L^\infty(\Omega\times[r,S]\times[0,1])}
 \leq\eta(S-r)+\frac{\varepsilon}{4}.
\end{equation}
In particular, $\psi^{\eta,\varepsilon}$ converges uniformly to $\varphi$
as $\eta,\varepsilon\downarrow0$.

If $\varphi$ is a smooth lower comparison field, then
\begin{equation*}
 \psi_-^{\eta,\varepsilon}(t,x)
 :=\varphi(t,x)-\eta(S-t)-\varepsilon q(x)
\end{equation*}
satisfies the corresponding inequalities
\begin{equation*}
 \psi_-^{\eta,\varepsilon}(S,x)\leq\varphi(S,x),\qquad
 \mathcal R[\psi_-^{\eta,\varepsilon}](t,x)\leq-\frac{\eta}{2},
 \qquad
 \frac{\partial\psi_-^{\eta,\varepsilon}}{\partial\vec n}(t,i)
 =\varepsilon,\quad i=0,1.
\end{equation*}
\end{lemma}

\begin{proof}
Since $S$ is deterministic, the field in
(\ref{uniq-direct-strict-field}) is adapted and has the same stochastic
regularity as $\varphi$. Moreover,
\begin{equation*}
 \mathfrak d_t\psi^{\eta,\varepsilon}
 =\mathfrak d_t\varphi-\eta,\qquad
 \mathfrak d_{\omega^k}\psi^{\eta,\varepsilon}
 =\mathfrak d_{\omega^k}\varphi,\quad k=0,1.
\end{equation*}
Since
\begin{equation*}
 Dq(x)=2x-1,\qquad D^2q(x)=2,\qquad\text{on a neighbourhood of }[0,1],
\end{equation*}
substitution into the definition of $\mathcal R$ gives
\begin{equation*}
\begin{split}
 \mathcal R[\psi^{\eta,\varepsilon}]
 ={}&\mathcal R[\varphi]+\eta+\lambda\eta(S-t)\\
 &+\varepsilon\left[
 -2a(t,x)-b(t,x)(2x-1)
 +\lambda\left(x-\frac{1}{2}\right)^2\right].
\end{split}
\end{equation*}
For $x\in[0,1]$,
\begin{equation*}
 \left|-2a(t,x)-b(t,x)(2x-1)+\lambda\left(x-\frac{1}{2}\right)^2\right|
 \leq 2\|a\|_\infty+\|b\|_\infty+\frac{\lambda}{4}=C_q.
\end{equation*}
Since $\lambda>0$, $S-t\geq0$, and
$\mathcal R[\varphi]\geq0$, it follows that
\begin{equation*}
 \mathcal R[\psi^{\eta,\varepsilon}]
 \geq\eta-C_q\varepsilon\geq\frac{\eta}{2}.
\end{equation*}

By (\ref{uniq-normal}),
\begin{equation*}
 \frac{\partial q}{\partial\vec n}(0)
 =\vec n(0)Dq(0)=-1,\qquad
 \frac{\partial q}{\partial\vec n}(1)
 =\vec n(1)Dq(1)=-1.
\end{equation*}
The homogeneous Neumann condition for $\varphi$ therefore yields the boundary
identity in (\ref{uniq-strict-field}). Since $q\geq0$,
\begin{equation*}
 \psi^{\eta,\varepsilon}(S,x)
 =\varphi(S,x)+\varepsilon q(x)\geq\varphi(S,x).
\end{equation*}
Finally, $0\leq q\leq1/4$ on $[0,1]$, which proves
(\ref{uniq-strict-uniform-error}). The lower comparison field follows by
changing the signs of the two perturbations. The opposite signs provide strict
barriers for the two semisolution inequalities, allowing the same weak
comparison argument to be used symmetrically below.
\end{proof}

Next we compare a viscosity subsolution (supersolution) with a smooth comparison field. Fix $M\geq\widehat M_0$ and $N=0$, and choose $\delta_M>0$ sufficiently small such that
\begin{equation}\label{uniq-short-time}
	(2M+1)\delta_Me^{M\delta_M}<\frac{1}{2},\qquad M(2M+1)\delta_M\leq\frac{M+1}{2}.
\end{equation}

The next nonlinear optimal stopping lemma defines an admissible test field at a
suitable stopping time.

\begin{lemma}\label{uniq-Snell}
Let $\mathscr Y$ be a bounded continuous adapted process on $[t,S]$ and put
\begin{equation}\label{uniq-Snell-def}
 \mathscr Z_s:=\underset{\rho\in\mathcal T_{s,S}}{\esssup}
 \overline{\mathcal E}_{s,\rho}^{M,0}[\mathscr Y_\rho].
\end{equation}
Then $\mathscr Z$ has a continuous version and is the smallest
$\overline{\mathcal E}^{M,0}$-supermartingale dominating $\mathscr Y$. If
\begin{equation*}
 \tau:=\inf\{s\geq t:\mathscr Z_s=\mathscr Y_s\}\wedge S,
\end{equation*}
then
\begin{equation}\label{uniq-Snell-contact}
 \mathscr Z_t=\overline{\mathcal E}_{t,\tau}^{M,0}[\mathscr Y_\tau],
 \qquad
 \mathscr Y_\tau\geq
 \overline{\mathcal E}_{\tau,\rho}^{M,0}[\mathscr Y_\rho]
 \quad\text{for every }\rho\in\mathcal T_{\tau,S}.
\end{equation}
\end{lemma}

\begin{proof}
For $k\geq1$, let $(\mathscr Z^k,\Gamma^k)$ solve the penalized BSDE
\begin{equation*}
\begin{split}
 \mathscr Z_s^k={}&\mathscr Y_S+
 \int_s^S\{M(|\mathscr Z_r^k|+|\Gamma_r^k|)
        +k(\mathscr Y_r-\mathscr Z_r^k)^+\}\dd r
 -\int_s^S\Gamma_r^k \dd W_r.
\end{split}
\end{equation*}
The comparison theorem implies that $\mathscr Z^k$ increases with $k$.
Put
\begin{equation*}
 L_s^k:=k\int_t^s(\mathscr Y_r-\mathscr Z_r^k)^+\dd r,
 \qquad s\in[t,S].
\end{equation*}
It\^o's formula, the Burkholder--Davis--Gundy inequality, and the boundedness
of $\mathscr Y$ give
\begin{equation*}
 E\left[\sup\limits_{t\leq s\leq S}|\mathscr Z_s^k|^2
       +\int_t^S|\Gamma_s^k|^2\dd s+|L_S^k|^2\right]
 \leq C E\left[\sup\limits_{t\leq s\leq S}|\mathscr Y_s|^2\right],
\end{equation*}
where $C$ depends only on $M$ and $T$, and is independent of $k$.
The monotone stability argument therefore yields that $\mathscr Z:=\lim_{k\to\infty}\mathscr Z^k$ is the first component of the solution $(\mathscr Z,\Gamma,L)$ of the reflected BSDE
\begin{equation}\label{uniq-reflected-Snell}
\left\{\begin{aligned}
&\mathscr Z_s=\mathscr Y_S+\int_s^S M(|\mathscr Z_r|+|\Gamma_r|)\dd r+L_S-L_s-\int_s^S\Gamma_r\dd W_r,\\
&\mathscr Z_s\geq\mathscr Y_s,\qquad \int_t^S(\mathscr Z_s-\mathscr Y_s)\dd L_s=0,\\
&L\text{ is continuous and nondecreasing},\quad L_t=0,\quad E[L_S^2]<\infty.
\end{aligned}\right.
\end{equation}
Since the obstacle and the Brownian filtration are continuous, $\mathscr Z$
and $L$ have continuous versions. Equation (\ref{uniq-reflected-Snell}) shows
that $\mathscr Z$ is a nonlinear supermartingale. Conversely, every such
supermartingale dominating $\mathscr Y$ dominates the right-hand side of
(\ref{uniq-Snell-def}); hence $\mathscr Z$ is its Snell envelope. Before
$\tau$, the latter condition in (\ref{uniq-reflected-Snell}) implies
that $L$ is constant. Stopping the equation at $\tau$ proves the first identity
in (\ref{uniq-Snell-contact}). The second follows from the supermartingale
property and $\mathscr Z\geq\mathscr Y$.
\end{proof}

\begin{lemma}\label{uniq-hitting-time-test}
Let $0\leq r<S\leq T$ be deterministic and assume
$S-r\leq\delta_M$. Let $\widetilde u$ be a bounded viscosity
subsolution of the normalized equation, and let $\psi\in\mathcal{C}_\mathcal{F}^3$  have the regularity
of Definition \ref{testfunction} on every compact subinterval of $(r,S)$
and, for some $\eta>0$ and $\varepsilon>0$, satisfy the residual and boundary conditions in
(\ref{uniq-strict-field}) on $[r,S]$. Assume in addition that
\begin{equation}\label{uniq-contact-terminal-order}
 \widetilde u(S,x)\leq\psi(S,x),
 \qquad x\in[0,1],\quad P\text{-a.s.}
\end{equation}
Then, if
\begin{equation*}
 P\left(
 \sup\limits_{(s,x)\in[r,S]\times[0,1]}
 (\widetilde u(s,x)-\psi(s,x))>0
 \right)>0,
\end{equation*}
there exist a stopping time $\tau\leq S$, an event
$\Omega_\tau\in\mathcal F_\tau$ of positive probability with
$\Omega_\tau\subset\{\tau<S\}$, a random variable
$\xi_\tau\in L^0(\Omega_\tau,\mathcal F_\tau;[0,1])$, and a test field
\begin{equation}\label{uniq-contact-class}
 \varphi_c\in\underline{\mathcal G}^{M,0}\widetilde u
 (\tau,\xi_\tau;\Omega_\tau)
\end{equation}
with associated $\hat\tau$ such that $\hat\tau>\tau$ on $\Omega_\tau$ and
\begin{equation}\label{uniq-contact-contradiction}
\begin{aligned}
&\essliminf_{s\downarrow\tau}\mathcal R[\varphi_c](s,\xi_\tau)>0,\qquad\text{on }\Omega_\tau\cap\{\xi_\tau\in(0,1)\},\\
&\frac{\partial\varphi_c}{\partial\vec n}(s,\xi_\tau)=-\varepsilon,\quad\tau\leq s\leq\hat\tau,\quad\text{on }\Omega_\tau\cap\{\xi_\tau\in\{0,1\}\}.
\end{aligned}
\end{equation}
\end{lemma}

\begin{proof}
Let us define
\begin{equation*}
 m_s:=\max_{x\in[0,1]}(\widetilde u(s,x)-\psi(s,x))^+,\qquad
 \Lambda:=\left\|\sup\limits_{r\leq s\leq S}m_s\right\|_{L^\infty}.
\end{equation*}
The process $m$ is continuous and adapted, and
(\ref{uniq-contact-terminal-order}) gives $m_S=0$. The hypothesis gives
$\Lambda>0$. By continuity of $m$ and the rational-time separability, there is a
deterministic $t\in(r,S)$ such that
$\Omega_t:=\{m_t>\Lambda/2\}$ has positive probability. Since
$t\in(r,S)$, the assumption on the interval gives $S-t\leq\delta_M$.
Put $c=(2M+1)\Lambda$ and
$\mathscr Y_s=m_s+c(s-t)$. Let $\mathscr Z$ be the Snell envelope introduced in Lemma
\ref{uniq-Snell}, with its first hitting time denoted by $\tau$. We claim that
$P(\Omega_t\cap\{\tau<S\})>0$. Otherwise, we would have $\tau=S$ on $\Omega_t$, and the locality of
the nonlinear expectation, Lemma
\ref{uniq-E-properties}, (\ref{uniq-short-time}), and (\ref{uniq-Snell-contact}) give, on
$\Omega_t$,
\begin{equation*}
 \mathscr Z_t=\overline{\mathcal E}_{t,S}^{M,0}[c(S-t)]
 =c(S-t)e^{M(S-t)}<\Lambda/2<m_t=\mathscr Y_t.
\end{equation*}
But this contradicts $\mathscr Z\geq\mathscr Y$. Thus
$\Omega_\tau:=\Omega_t\cap\{\tau<S\}$ is of strictly positive probability. Moreover, on
$\{\tau<S\}$, (\ref{uniq-E-explicit-upper}) and
(\ref{uniq-Snell-contact}) give
\begin{equation*}
\begin{aligned}
m_{\tau}+c(\tau-t)
={}\mathscr Y_{\tau}
={}\mathscr Z_{\tau}
\geq \overline{\mathcal E}_{\tau,S}^{M,0}[\mathscr Y_S]
={}\overline{\mathcal E}_{\tau,S}^{M,0}[c(S-t)]
={}c(S-t)e^{M(S-\tau)}.
\end{aligned}
\end{equation*}
Consequently,
\begin{equation*}
 m_{\tau}\geq
 c\big[(S-t)e^{M(S-\tau)}-(\tau-t)\big]>0
 \quad\text{on }\{\tau<S\}.
\end{equation*}
In particular, $m_\tau>0$ on $\Omega_\tau$.

By the measurable selection theorem, we can choose
$\xi_\tau\in L^0(\Omega_\tau,\mathcal F_\tau;[0,1])$ such that
\begin{equation}\label{uniq-measurable-max}
 m_\tau=\widetilde u(\tau,\xi_\tau)-\psi(\tau,\xi_\tau).
\end{equation}
Put
\begin{equation*}
 \chi:=\mathscr Y_\tau e^{-M\tau}\in L^\infty(\mathcal F_\tau),
 \qquad \mathscr M_s:=E[\chi\mid\mathcal F_s].
\end{equation*}
By the martingale representation theorem and the fact that $\chi$ is
$\mathcal F_\tau$-measurable,
\begin{equation*}
 \mathscr M_s=\mathscr M_0+\int_0^{s\wedge\tau}Z_r^\chi\dd W_r,
 \qquad 0\leq s\leq T.
\end{equation*}
We define the following adapted process, which does not involve the variable $x$:
\begin{equation*}
 \beta^{c}_s=e^{Ms}\mathscr M_s-c(s-t),\qquad 0\leq s\leq T.
\end{equation*}
Since $\chi$ is bounded, the Burkholder--Davis--Gundy inequality gives
\begin{equation*}
 E\left[\left(\int_0^T|Z_s^\chi|^2\dd s\right)^2\right]
 \leq C E[|\chi-E\chi|^4]<\infty.
\end{equation*}
Thus $\beta^c\in\mathcal C_{\mathcal F}^3([0,T]\times\mathbb R)$.
For $s\geq\tau$, $\mathscr M_s=\mathscr Y_\tau e^{-M\tau}$
and $Z_s^\chi=0$, $\dd s\otimes\dd P$-a.e., and so
\begin{equation*}
 \beta^{c}_s=\mathscr Y_\tau e^{M(s-\tau)}-c(s-t),\qquad
 \mathfrak d_t\beta^{c}(s)=M\mathscr Y_\tau e^{M(s-\tau)}-c,\qquad
 \mathfrak d_\omega\beta^{c}(s)=0.
\end{equation*}
Set $\varphi_c:=\psi+\beta^c$ and $\hat\tau:=S$.
By (\ref{uniq-measurable-max}), $\varphi_c(\tau,\xi_\tau)=\widetilde
u(\tau,\xi_\tau)$ on $\Omega_\tau$. All identities and inequalities in the
remaining part of this proof are understood on $\Omega_\tau$. For
$\rho\in\mathcal T_{\tau,S}$, the definition of $m$ and that of $\mathscr Y$ give
\begin{equation*}
 \inf\limits_{x\in[0,1]}(\varphi_c-\widetilde u)(\rho,x)
 \geq\mathscr Y_\tau e^{M(\rho-\tau)}-\mathscr Y_\rho
 \quad\text{on }\Omega_\tau.
\end{equation*}
Using (\ref{uniq-E-duality}), (\ref{uniq-E-superadditive}),
(\ref{uniq-E-explicit-lower}) and (\ref{uniq-Snell-contact}), we obtain
\begin{equation*}
 \underline{\mathcal E}_{\tau,\rho}^{M,0}
 \left[\inf\limits_{x\in[0,1]}(\varphi_c-\widetilde u)(\rho,x)\right]
 \geq\mathscr Y_\tau-
 \overline{\mathcal E}_{\tau,\rho}^{M,0}[\mathscr Y_\rho]\geq0
 \quad\text{on }\Omega_\tau,
\end{equation*}
and equality holds between all three terms of the above inequality for $\rho=\tau$. This proves
(\ref{uniq-contact-class}).

For $\tau<s\leq\hat\tau$,
\begin{equation}\label{uniq-contact-residual}
 \mathcal R[\varphi_c]
 =\mathcal R[\psi]+c-M\mathscr Y_\tau e^{M(s-\tau)}+\lambda\beta^{c}_s.
\end{equation}
Because $\mathscr Y_\tau\leq\Lambda+c(S-t)$, condition
(\ref{uniq-short-time}) implies
\begin{equation*}
 c-M\mathscr Y_\tau\geq\frac{M+1}{2}\Lambda.
\end{equation*}
Letting $s\downarrow\tau$ in (\ref{uniq-contact-residual}) and using $\beta^c_\tau=m_\tau>0$, we obtain
\begin{equation*}
\begin{aligned}
\essliminf_{s\downarrow\tau}\mathcal R[\varphi_c](s,\xi_\tau)
&\geq\frac{\eta}{2}+c-M\mathscr Y_\tau+\lambda m_\tau\\
&\geq\frac{\eta}{2}+\frac{M+1}{2}\Lambda+\lambda m_\tau>0,\qquad\text{on }\Omega_\tau\cap\{\xi_\tau\in(0,1)\}.
\end{aligned}
\end{equation*}
Moreover, since $\beta^c$ does not depend on $x$, the normal derivative of $\varphi_c$ agrees with that of $\psi$. Hence
\begin{equation*}
\frac{\partial\varphi_c}{\partial\vec n}(s,\xi_\tau)=-\varepsilon,\quad\tau\leq s\leq\hat\tau,\quad\text{on }\Omega_\tau\cap\{\xi_\tau\in\{0,1\}\}.
\end{equation*}
This proves (\ref{uniq-contact-contradiction}).
\end{proof}

\begin{theorem}\label{uniq-weak-comparison}
Fix a finite $\widehat M_0\geq0$, and let $\widetilde u$ be a bounded
$(\widehat M_0,0)$-viscosity subsolution of
(\ref{uniq-normalized-equation})--(\ref{uniq-normalized-boundary}) and
$\varphi$ be a smooth upper comparison field. Then
$\widetilde u\leq\varphi$. Symmetrically, if $\widetilde v$ is a bounded
$(\widehat M_0,0)$-viscosity supersolution and $\varphi$ is a smooth lower
comparison field, then
$\varphi\leq\widetilde v$.
\end{theorem}

\begin{proof}
We prove the subsolution statement. Let
$0=s_0<s_1<\cdots<s_{N_p}=T$ be a deterministic refinement of the regularity
partition of $\varphi$ such that every interval $[s_{j-1},s_j]$ satisfies
(\ref{uniq-short-time}). We proceed by a backward induction on $j$.

The terminal condition for viscosity subsolutions and
(\ref{uniq-upper-field}) give
\begin{equation*}
 \widetilde u(T,x)\leq\widetilde\Phi(x)
 \leq\varphi(T,x),\qquad x\in[0,1].
\end{equation*}
Suppose, for some $j$, that
\begin{equation*}
 \widetilde u(s_j,x)\leq\varphi(s_j,x),
 \qquad x\in[0,1],\quad P\text{-a.s.}
\end{equation*}
For $\eta>0$, choose
$0<\varepsilon=\varepsilon(\eta)\leq\eta/(2C_q)$ (recall that $C_q$ is defined in (\ref{uniq-Cq})), and let
$\psi^{\eta,\varepsilon}$ be the field introduced in
(\ref{uniq-direct-strict-field}) with $S=s_j$. By
(\ref{uniq-strict-field}),
\begin{equation*}
 \widetilde u(s_j,x)\leq\varphi(s_j,x)
 \leq\psi^{\eta,\varepsilon}(s_j,x).
\end{equation*}
If $\widetilde u>\psi^{\eta,\varepsilon}$ somewhere on
$[s_{j-1},s_j]\times[0,1]$ with positive probability, Lemma
\ref{uniq-hitting-time-test} provides a test field $\varphi_c$ satisfying
(\ref{uniq-contact-contradiction}).

For $\xi_\tau$ defined in (\ref{uniq-measurable-max}), Remark \ref{vs5}
and (\ref{uniq-contact-contradiction}) give, on $\Omega_\tau\cap\{\xi_\tau\in(0,1)\}$,
\begin{equation*}
 0\geq\essliminf_{s\downarrow\tau}\mathcal R[\varphi_c](s,\xi_\tau)>0.
\end{equation*}
On $\Omega_\tau\cap\{\xi_\tau\in\{0,1\}\}$, (\ref{uniq-boundary-residual}) gives instead
\begin{equation*}
 0\geq\essliminf_{s\downarrow\tau}
 \bigl(-\widetilde{\mathbb G}_{-,\lambda}(\varphi_c)\bigr)(s,\xi_\tau)
 =\essliminf_{s\downarrow\tau}
 \bigl(\mathcal R[\varphi_c](s,\xi_\tau)\wedge\varepsilon\bigr)>0.
\end{equation*}
Both contradict the viscosity subsolution inequalities. Hence
\begin{equation*}
 \widetilde u\leq\psi^{\eta,\varepsilon}
 \quad\text{on }[s_{j-1},s_j]\times[0,1].
\end{equation*}
Letting $\eta\downarrow0$ and using
(\ref{uniq-strict-uniform-error}) yields
\begin{equation*}
 \widetilde u\leq\varphi
 \quad\text{on }[s_{j-1},s_j]\times[0,1].
\end{equation*}
Thus the inequality needed at $s_{j-1}$ holds. Repeating this argument over
the finitely many intervals proves $\widetilde u\leq\varphi$ on
$[0,T]\times[0,1]$.

The supersolution statement follows in the same way, using
$\psi_-^{\eta,\varepsilon}$ from Lemma \ref{uniq-strictification} and the
upper test classes.
\end{proof}

\subsection{Deterministic systems and smooth adapted comparison fields}

The weak comparison theorem cannot be applied directly to two nonsmooth
viscosity semisolutions. We now construct smooth upper and lower comparison
fields converging to the same limit. Only the common noise $W^0$ enters the
random coefficients. The independent noise $W^1$ remains in the state
equation and provides uniform ellipticity for the augmented deterministic
system.

In this subsection the number of intervals in a cylindrical approximation is
denoted by $n$. Thus the symbol $N$ remains reserved for the local-time
parameter in $\mathcal E^{M,N}$.

\paragraph{Cylindrical approximation.}
By the same cylindrical approximation argument as in the proof of Theorem \ref{continuity}, applied to $b,\widetilde F,\widetilde\Phi$ with respect to the common filtration $\{\mathcal F_t^0\}$, for every $\varepsilon>0$, there exist a partition $0=t_0<t_1<\cdots<t_n=T$ and bounded smooth deterministic functions
\begin{equation*}
\begin{aligned}
b_n&=b_n(y_1,\ldots,y_n,t,x),\qquad\widetilde F_n=\widetilde F_n(y_1,\ldots,y_n,t,x),\\
\widetilde\Phi_n&=\widetilde\Phi_n(y_1,\ldots,y_n,x),
\end{aligned}
\end{equation*}
such that the approximation errors
\begin{equation*}
\begin{aligned}
\delta b_t^\varepsilon&:=\sup\limits_{x\in[0,1]}|b_n(W^0_{t_1\wedge t},\ldots,W^0_{t_n\wedge t},t,x)-b(t,x)|,\\
\delta F_t^\varepsilon&:=\sup\limits_{x\in[0,1]}|\widetilde F_n(W^0_{t_1\wedge t},\ldots,W^0_{t_n\wedge t},t,x)-\widetilde F(t,x)|,\\
\delta\Phi^\varepsilon&:=\sup\limits_{x\in[0,1]}|\widetilde\Phi_n(W^0_{t_1},\ldots,W^0_{t_n},x)-\widetilde\Phi(x)|
\end{aligned}
\end{equation*}
satisfy
\begin{equation*}
\|\delta\Phi^\varepsilon\|_{L^8(\Omega)}+\|\delta b^\varepsilon\|_{L^8(\Omega;L^1(0,T))}+\|\delta F^\varepsilon\|_{L^8(\Omega;L^1(0,T))}<\varepsilon.
\end{equation*}
Moreover, $\delta b^\varepsilon$ and $\delta F^\varepsilon$ are $\{\mathcal F_t^0\}$-adapted, and $\delta\Phi^\varepsilon$ is $\mathcal F_T^0$-measurable. The functions $b_n,\widetilde F_n,\widetilde\Phi_n$ are uniformly bounded and uniformly Lipschitz continuous in $x$, with constants independent of $n$ and $\varepsilon$. The existence of the above approximating functions can be established by arguments similar to those used in \cite[Lemma~5.5]{Qiu18sicon} and \cite[Lemma~2.1]{Qiu22AAP}. Set
\begin{equation*}
\begin{aligned}
b^\varepsilon(t,x)&:=b_n(W^0_{t_1\wedge t},\ldots,W^0_{t_n\wedge t},t,x),\\
\widetilde F^\varepsilon(t,x)&:=\widetilde F_n(W^0_{t_1\wedge t},\ldots,W^0_{t_n\wedge t},t,x),\\
\widetilde\Phi^\varepsilon(x)&:=\widetilde\Phi_n(W^0_{t_1},\ldots,W^0_{t_n},x).
\end{aligned}
\end{equation*}

For the subsequent analysis, we need the approximations $\widetilde\Phi_n$
to satisfy the Neumann compatibility condition
\begin{equation}\label{uniq-terminal-compatibility}
 D_x\widetilde\Phi_n(\mathbf y,0)=D_x\widetilde\Phi_n(\mathbf y,1)=0,
 \qquad \mathbf y=(y_1,\ldots,y_n)\in\mathbb R^n.
\end{equation}
To this end, we use the following construction. Set
\begin{equation*}
 r(x)=\begin{cases}x,&0\leq x\leq1,\\2-x,&1\leq x\leq2,\end{cases}
 \qquad r(x+2k)=r(x),\quad k\in\mathbb Z,
\end{equation*}
and put $\overline\Phi_n(\mathbf y,x):=\widetilde\Phi_n(\mathbf y,r(x))$. Notice that $\overline\Phi_n$ and $\widetilde\Phi_n$ coincide for $x\in[0,1]$.
Take the even kernel
\begin{equation*}
 \zeta(z):=\begin{cases}c_0\exp\bigl(-1/(1-z^2)\bigr),&|z|<1,\\
 0,&|z|\geq1,\end{cases}
\end{equation*}
with the constant $c_0>0$ chosen so that $\int_{\mathbb R}\zeta(z)\dd z=1$ and define
\begin{equation*}
 \widetilde\Phi_{n,\delta}(\mathbf y,x):=\int_{\mathbb R}\zeta(z)
 \overline\Phi_n(\mathbf y,x-\delta z)\dd z.
\end{equation*}
If $L_\Phi$ is a common spatial Lipschitz constant of $\widetilde\Phi_n$, then
\begin{equation*}
 \|\widetilde\Phi_{n,\delta}-\widetilde\Phi_n\|_{\infty,\mathbb R^n\times[0,1]}
 \leq L_\Phi\delta\int_{\mathbb R}|z|\zeta(z)\dd z,
 \qquad \|D_x\widetilde\Phi_{n,\delta}\|_\infty\leq L_\Phi.
\end{equation*}
Moreover, we have
\begin{equation*}
 \widetilde\Phi_{n,\delta}(\mathbf y,-x)=\widetilde\Phi_{n,\delta}(\mathbf y,x),
 \qquad
 \widetilde\Phi_{n,\delta}(\mathbf y,1+x)=\widetilde\Phi_{n,\delta}(\mathbf y,1-x).
\end{equation*}
Indeed, by the evenness of $\overline\Phi_n(\mathbf y,\cdot)$ and $\zeta$, and the change of variables $z\mapsto-z$,
\begin{equation*}
	\begin{aligned}
		\widetilde\Phi_{n,\delta}(\mathbf y,-x)
		&=\int_{\mathbb R}\zeta(z)\overline\Phi_n(\mathbf y,x+\delta z)\dd z\\
		&=\int_{\mathbb R}\zeta(z)\overline\Phi_n(\mathbf y,x-\delta z)\dd z
		=\widetilde\Phi_{n,\delta}(\mathbf y,x).
	\end{aligned}
\end{equation*}
The second identity follows similarly. Differentiating these identities with respect to $x$ and setting $x=0$, we obtain
\begin{equation*}
 -D_x\widetilde\Phi_{n,\delta}(\mathbf y,0)=D_x\widetilde\Phi_{n,\delta}(\mathbf y,0),
 \qquad
 D_x\widetilde\Phi_{n,\delta}(\mathbf y,1)=-D_x\widetilde\Phi_{n,\delta}(\mathbf y,1).
\end{equation*}
Hence
\begin{equation*}
 D_x\widetilde\Phi_{n,\delta}(\mathbf y,0)=D_x\widetilde\Phi_{n,\delta}(\mathbf y,1)=0.
\end{equation*}
We now set
\begin{equation*}
 \widetilde\Phi_n:=\widetilde\Phi_{n,1/n}.
\end{equation*}
For $t\in[t_{j-1},t_j)$, the variables
$y_1,\ldots,y_{j-1}$ represent the past common-noise coordinates and $y$
represents the current one, as in the finite-dimensional construction of
\cite{Qiu22AAP}. Define
\begin{equation*}
\begin{split}
 b_{n,j}(y_1,\ldots,y_{j-1},y,t,x)&:=b_n\bigl(y_1,\ldots,y_{j-1},
 \underbrace{y,\ldots,y}_{n-j+1},t,x\bigr),\\
 \widetilde F_{n,j}(y_1,\ldots,y_{j-1},y,t,x)&:=\widetilde F_n\bigl(y_1,\ldots,y_{j-1},
 \underbrace{y,\ldots,y}_{n-j+1},t,x\bigr).
\end{split}
\end{equation*}
We look for deterministic functions
$v_{n,j}(t,x,y_1,\ldots,y_{j-1},y)$ satisfying
\begin{equation}\label{uniq-augmented-PDE}
\begin{split}
&-\partial_tv_{n,j}-\frac{1}{2}D_{yy}v_{n,j}-aD_{xx}v_{n,j}
 -\sigma^0D_{xy}v_{n,j}-b_{n,j}D_xv_{n,j}
 +\lambda v_{n,j}-\widetilde F_{n,j}=0,
\end{split}
\end{equation}
with the homogeneous boundary condition
\begin{equation}\label{uniq-augmented-boundary}
 D_xv_{n,j}(t,0,y_1,\ldots,y_{j-1},y)
 =D_xv_{n,j}(t,1,y_1,\ldots,y_{j-1},y)=0.
\end{equation}
On the last interval the terminal condition is
\begin{equation}\label{uniq-last-terminal}
 v_{n,n}(T,x,y_1,\ldots,y_{n-1},y)
 =\widetilde\Phi_n(y_1,\ldots,y_{n-1},y,x),
\end{equation}
and on the preceding intervals it is given recursively by
\begin{equation}\label{uniq-interface}
\begin{split}
&v_{n,j}(t_j,x,y_1,\ldots,y_{j-1},y)\\
&\qquad=v_{n,j+1}(t_j,x,y_1,\ldots,y_{j-1},y,y),
\qquad 1\leq j<n.
\end{split}
\end{equation}

\begin{proposition}
For every cylindrical approximation and every $j$, the problem
(\ref{uniq-augmented-PDE})--(\ref{uniq-interface}) has a unique bounded
classical solution. There exists $\alpha\in(0,1)$ such that, for every compact
interval $I\Subset(t_{j-1},t_j)$,
\begin{equation}\label{uniq-higher-regularity}
 v_{n,j}\in C^{2+\alpha/2,4+\alpha}
 \bigl(I\times[0,1]\times\mathbb R\bigr)
\end{equation}
in the variables $(t,x,y)$. The corresponding norms are locally bounded
uniformly with respect to the past parameters. Moreover, $v_{n,j}$ is smooth
in those parameters. No regularity across the time interfaces is required.
\end{proposition}

\begin{proof}
Fix $j$ and the past parameters. To match the time convention of
\cite[Theorem~5.19]{Lieberman1996}, set $s=t_j-t$, $z=(x,y)$ and
$u(s,z)=v_{n,j}(t_j-s,z)$. After multiplying the time-reversed equation by
$-1$, it takes the form
\begin{equation*}
 Lu=-\widehat F_{n,j},\qquad
 L=-\partial_s+\sum_{i,k=1}^2A^{ik}\partial_{z_i z_k}
 +\sum_{i=1}^2B^i\partial_{z_i}+c,
\end{equation*}
where $z_1=x$, $z_2=y$, the hats denote the time-reversed coefficients, and
\begin{equation*}
 A=(A^{ik})_{1\leq i,k\leq2}=\frac{1}{2}
 \begin{pmatrix}
 |\widehat\sigma^0|^2+|\widehat\sigma^1|^2&\widehat\sigma^0\\
 \widehat\sigma^0&1
 \end{pmatrix},\qquad
 B=(\widehat b_{n,j},0),\qquad c=-\lambda.
\end{equation*}
We verify the hypotheses of that theorem.

First, with
\begin{equation*}
 C_\sigma:=1+\|\sigma^0\|_\infty^2+\|\sigma^1\|_\infty^2,
\end{equation*}
uniform ellipticity follows directly from the structure of the augmented
matrix. Indeed, for $\zeta=(\xi,\eta)\in\mathbb R^2$,
\begin{equation*}
 2\zeta A\zeta^*
 =|\widehat\sigma^0\xi+\eta|^2
  +|\widehat\sigma^1|^2\xi^2.
\end{equation*}
Put $Q:=2A$. Then
$\det Q=|\widehat\sigma^1|^2\geq\kappa$, while
$\operatorname{tr}Q\leq C_\sigma$. Since a positive definite $2\times2$ matrix
satisfies
$\lambda_{\min}(Q)\geq\det Q/\operatorname{tr}Q$, we obtain
\begin{equation*}
 \frac{\kappa}{2C_\sigma}|\zeta|^2
 \leq \zeta A\zeta^*
 \leq \frac{C_\sigma}{2}|\zeta|^2,
 \qquad \operatorname{tr}A\leq\frac{C_\sigma}{2}.
\end{equation*}
The constants are uniform because the diffusion coefficients are not
approximated.
Thus the uniform parabolicity condition
\cite[(5.2a)]{Lieberman1996} holds, with constants independent of
$\varepsilon,n,j$ and the past parameters. Here we use the assumption
$|\sigma^1|^2\geq\kappa$. Since
$\sigma^0,\sigma^1\in C_b^3$ and the cylindrical functions
$\widehat b_{n,j}$ and $\widehat F_{n,j}$ are smooth with bounded derivatives,
the coefficient bounds in \cite[(5.8)]{Lieberman1996} hold for every fixed
$\alpha\in(0,1)$; the zeroth-order coefficient $c=-\lambda$ is constant.

Next, let $\gamma$ denote the inward unit normal in the notation of
\cite{Lieberman1996}. On $x=0$ take $\beta=\gamma=(1,0)$, while on $x=1$
take $\beta=\gamma=(-1,0)$, and in both cases put $\beta^0=0$. The boundary
operator $Mu=\beta\cdot Du+\beta^0u$ is then equivalent to the homogeneous
condition (\ref{uniq-augmented-boundary}). Moreover,
\begin{equation*}
 \beta\cdot\gamma=1,
\end{equation*}
so the obliqueness requirement \cite[(5.4)]{Lieberman1996} holds with
$\chi=1$. Condition \cite[(5.22)]{Lieberman1996} holds because
$|\beta^0|=0$. The required $H_{\delta-1}$ bounds also hold because $\beta$
is constant on each boundary component.

Finally, for every $1<\delta<2$, the smooth bounded forcing
$-\widehat F_{n,j}$ belongs to the intermediate space
$H_\alpha^{(2-\delta)}$ required in Theorem 5.19, the boundary data satisfy
$\psi=0\in H_{\delta-1}$, and the terminal data belong to $H_\delta$. On the
last interval the corner compatibility
$M\phi=\psi$ is exactly (\ref{uniq-terminal-compatibility}); on each preceding
interval it follows recursively from (\ref{uniq-interface}) and the Neumann
condition already obtained on the next interval. Hence
\cite[Theorem~5.19]{Lieberman1996} applies on every smooth bounded
truncation of the strip. Its Schauder estimate, followed by the standard
boundary bootstrap for the smooth coefficients, yields
(\ref{uniq-higher-regularity}). The existence and uniqueness of a bounded
classical solution on the whole $\mathbb{R}$ follow from \cite[Section~3]{AngiuliLorenzi2015}
and the variation-of-constants formula. Differentiating the equation with
respect to the past parameters and applying the same estimate proves smooth dependence on the past parameters. Starting with $j=n$ and using
(\ref{uniq-interface}), backward induction constructs the full deterministic
system. Uniqueness among bounded solutions follows from the maximum
principle, since $c=-\lambda<0$.
\end{proof}

Define the reconstructed random field by
\begin{equation*}
 \widetilde V^\varepsilon(t,x):=
 v_{n,j}(t,x,W^0_{t_1},\ldots,W^0_{t_{j-1}},W^0_t),\qquad
 t_{j-1}\leq t<t_j.
\end{equation*}
At the terminal time, set
$\widetilde V^\varepsilon(T,x):=\widetilde\Phi^\varepsilon(x)$.

\begin{lemma}\label{uniq-reconstructed-field}
The field $\widetilde V^\varepsilon$ is $\{\mathcal F_t^0\}$-adapted (and
hence $\{\mathcal F_t\}$-adapted) and continuous on $[0,T]\times[0,1]$. It
belongs to $\mathcal C_{\mathcal F}^3$ on each open time interval, satisfies the homogeneous Neumann
condition, and has terminal value
$\widetilde V^\varepsilon(T)=\widetilde\Phi^\varepsilon$. On every interval
$[t_{j-1},t_j]$, with the derivatives taken for $t\in(t_{j-1},t_j)$,
\begin{equation}\label{uniq-reconstructed-decomposition}
 \mathfrak d_t\widetilde V^\varepsilon
 =\partial_tv_{n,j}+\frac{1}{2}D_{yy}v_{n,j},\qquad
 \mathfrak d_{\omega^0}\widetilde V^\varepsilon=D_yv_{n,j},\qquad
 \mathfrak d_{\omega^1}\widetilde V^\varepsilon=0.
\end{equation}
Moreover, $\widetilde V^\varepsilon$ solves
\begin{equation}\label{uniq-approximating-BSPDE}
\begin{split}
&-\mathfrak d_t\widetilde V^\varepsilon-aD^2\widetilde V^\varepsilon
-\sigma^0D\mathfrak d_{\omega^0}\widetilde V^\varepsilon
-\sigma^1D\mathfrak d_{\omega^1}\widetilde V^\varepsilon
-b^\varepsilon D\widetilde V^\varepsilon
+\lambda\widetilde V^\varepsilon-\widetilde F^\varepsilon=0
\end{split}
\end{equation}
in the classical sense. If $X^{t,x,\varepsilon}$ denotes the reflected state
process with drift $b^\varepsilon$, then
\begin{equation}\label{uniq-global-representation}
\begin{split}
\widetilde V^\varepsilon(t,x)=E\bigg[&e^{-\lambda(T-t)}
\widetilde\Phi^\varepsilon(X_T^{t,x,\varepsilon})
+\int_t^Te^{-\lambda(s-t)}
\widetilde F^\varepsilon(s,X_s^{t,x,\varepsilon})\dd s
\,\bigg|\mathcal F_t\bigg].
\end{split}
\end{equation}
Moreover, for some constant $C_*$ independent of $\varepsilon,n,j$ it holds
\begin{equation}\label{uniq-uniform-Lipschitz}
 |\widetilde V^\varepsilon(t,x)-\widetilde V^\varepsilon(t,x')|
 \leq C_*|x-x'|,\qquad t\in[0,T],\quad x,x'\in[0,1],
 \quad P\text{-a.s.}
\end{equation}
In particular, $|D_x\widetilde V^\varepsilon|\leq C_*$ wherever the derivative
is defined.
\end{lemma}

\begin{proof}
$\widetilde V^\varepsilon$ is adapted. The continuity of Brownian paths and
(\ref{uniq-interface}) ensure continuity at each $t_j$. Applying It\^o's
formula to
$v_{n,j}(t,x,W^0_{t_1},\ldots,W^0_{t_{j-1}},W_t^0)$ now yields
\begin{equation*}
\dd \widetilde V^\varepsilon(t,x)
=\left(\partial_tv_{n,j}+\frac{1}{2}D_{yy}v_{n,j}\right)\dd t
+D_yv_{n,j}\dd W_t^0,
\end{equation*}
which proves (\ref{uniq-reconstructed-decomposition}); substitution into
(\ref{uniq-augmented-PDE}) gives (\ref{uniq-approximating-BSPDE}). The
regularity follows from (\ref{uniq-higher-regularity}), and the Neumann
condition follows from (\ref{uniq-augmented-boundary}). Finally, the
It\^o--Wentzell formula gives (\ref{uniq-global-representation}). The uniform
spatial Lipschitz bounds of the approximating data and the standard
stability estimate for RSDEs then imply
(\ref{uniq-uniform-Lipschitz}) directly.
\end{proof}

\subsection{Smooth barriers and completion of the comparison proof}

By the cylindrical approximation described above, we have
\begin{equation}\label{uniq-error-convergence}
 \|\delta\Phi^\varepsilon\|_{L^8}
 +\|\delta b^\varepsilon\|_{L^8(\Omega;L^1(0,T))}
 +\|\delta F^\varepsilon\|_{L^8(\Omega;L^1(0,T))}
 \longrightarrow0.
\end{equation}
Let $r_t^\varepsilon=C_*\delta b_t^\varepsilon+\delta F_t^\varepsilon$, where
$C_*$ is the constant introduced in Lemma \ref{uniq-reconstructed-field}. Define
\begin{equation}\label{uniq-error-BSDE}
 Y_t^\varepsilon:=E\left[\delta\Phi^\varepsilon
 +\int_t^T r_s^\varepsilon\dd s\,\bigg|\mathcal F_t^0\right].
\end{equation}
Then, by definition, $Y^\varepsilon\geq0$, and standard BSDE estimates together with
(\ref{uniq-error-convergence}) give
\begin{equation}\label{uniq-error-vanishes}
 \|Y^\varepsilon\|_{\mathcal S^8}\longrightarrow0.
\end{equation}
The martingale representation theorem applied to $Y^\varepsilon$ in (\ref{uniq-error-BSDE}) yields a process $Z^\varepsilon$ such
that
\begin{equation}\label{uniq-error-decomposition}
 \mathfrak d_tY^\varepsilon=-r^\varepsilon,\qquad
 \mathfrak d_{\omega^0}Y^\varepsilon=Z^\varepsilon,\qquad
 \mathfrak d_{\omega^1}Y^\varepsilon=0.
\end{equation}
As $Y^\varepsilon$ is independent of $x$, all its spatial derivatives vanish,
and, as a random field, it belongs to $\mathcal C_{\mathcal F}^3$.

\begin{lemma}\label{uniq-smooth-barriers}
The random fields
\begin{equation*}
 \overline\phi^\varepsilon
 :=\widetilde V^\varepsilon+Y^\varepsilon,
 \qquad
 \underline\phi^\varepsilon
 :=\widetilde V^\varepsilon-Y^\varepsilon
\end{equation*}
are, respectively, a smooth upper comparison field and a smooth lower
comparison field for the normalized equation.
\end{lemma}

\begin{proof}
From (\ref{uniq-normalized-equation}) and (\ref{uniq-approximating-BSPDE}), the residual of
$\widetilde V^\varepsilon$ with respect to the original coefficients is
\begin{equation*}
 \mathcal R[\widetilde V^\varepsilon]
 =(b^\varepsilon-b)D\widetilde V^\varepsilon
  +(\widetilde F^\varepsilon-\widetilde F).
\end{equation*}
Therefore, by Lemma \ref{uniq-reconstructed-field},
\begin{equation*}
 |\mathcal R[\widetilde V^\varepsilon]|
 \leq C_*\delta b^\varepsilon+\delta F^\varepsilon
 =r^\varepsilon.
\end{equation*}
Since $Y^\varepsilon$ does not involve $x$, all its spatial derivatives,
including $D\mathfrak d_\omega Y^\varepsilon$, vanish. By
(\ref{uniq-normalized-equation}) and (\ref{uniq-error-decomposition}),
\begin{equation*}
 \mathcal R[\overline\phi^\varepsilon]
 =\mathcal R[\widetilde V^\varepsilon]
  +r^\varepsilon+\lambda Y^\varepsilon\geq0,
\end{equation*}
and
\begin{equation*}
 \mathcal R[\underline\phi^\varepsilon]
 =\mathcal R[\widetilde V^\varepsilon]
  -r^\varepsilon-\lambda Y^\varepsilon\leq0.
\end{equation*}
At $T$, $Y_T^\varepsilon=\delta\Phi^\varepsilon$, and hence
\begin{equation*}
 \overline\phi^\varepsilon(T,x)
 =\widetilde\Phi^\varepsilon(x)+\delta\Phi^\varepsilon
 \geq\widetilde\Phi(x),
 \qquad
 \underline\phi^\varepsilon(T,x)
 =\widetilde\Phi^\varepsilon(x)-\delta\Phi^\varepsilon
 \leq\widetilde\Phi(x).
\end{equation*}
Finally, adding or subtracting the space-independent field $Y^\varepsilon$
does not change the homogeneous Neumann condition. The required piecewise
regularity follows from Lemma \ref{uniq-reconstructed-field} and the preceding
properties of $Y^\varepsilon$. In addition,
\begin{equation*}
 (\overline\phi^\varepsilon)^-
 \leq(\widetilde V^\varepsilon)^-,\qquad
 (\underline\phi^\varepsilon)^+
 \leq(\widetilde V^\varepsilon)^+,
\end{equation*}
so the required $\mathcal S^\infty$ conditions in the definition of smooth
upper and lower comparison fields are satisfied.
\end{proof}

\begin{theorem}\label{uniq-normalized-comparison}
Fix a finite $\widehat M_0\geq0$. Let $\widetilde u$ be a bounded
$(\widehat M_0,0)$-viscosity subsolution and $\widetilde v$ a bounded
$(\widehat M_0,0)$-viscosity supersolution of the normalized equation. Then
\begin{equation}\label{uniq-normalized-order}
 \widetilde u(t,x)\leq\widetilde v(t,x),\qquad
 (t,x)\in[0,T]\times[0,1],\quad P\text{-a.s.}
\end{equation}
\end{theorem}

\begin{proof}
We apply Theorem \ref{uniq-weak-comparison} to the upper and lower fields introduced in
Lemma \ref{uniq-smooth-barriers}, so we obtain
\begin{equation*}
 \widetilde u\leq\widetilde V^\varepsilon+Y^\varepsilon,
 \qquad
 \widetilde V^\varepsilon-Y^\varepsilon\leq\widetilde v.
\end{equation*}
Taking the difference of both inequalities gives
\begin{equation*}
 \widetilde u(t,x)-\widetilde v(t,x)\leq2Y_t^\varepsilon.
\end{equation*}
By (\ref{uniq-error-vanishes}), we may choose a sequence
$\varepsilon_k\downarrow0$ such that
\begin{equation*}
 \sup\limits_{0\leq t\leq T}Y_t^{\varepsilon_k}\longrightarrow0,\qquad P\text{-a.s.}
\end{equation*}
Letting $k$ tend to infinity proves (\ref{uniq-normalized-order}).
\end{proof}

\begin{proof}[Proof of Theorem \ref{uniq-main}]
Set
\begin{equation*}
 \widetilde u=e^{\lambda_* t}(u-\ell),
 \qquad \widetilde v=e^{\lambda_* t}(v-\ell).
\end{equation*}
By the discussion following
(\ref{uniq-normalized-boundary}), $\widetilde u$ and $\widetilde v$ are,
respectively, a bounded $(\widehat M_0,0)$-viscosity subsolution and a bounded
$(\widehat M_0,0)$-viscosity supersolution of the normalized equation, with
$\widehat M_0=L_y+\lambda_*$ (Recall that $\lambda_*:=\lambda+L_y$ is defined before (\ref{uniq-exponential})). By Theorem
\ref{uniq-normalized-comparison}, we have $\widetilde u\leq\widetilde v$.
Since $e^{\lambda_* t}>0$ and the same lifting field is subtracted on both sides,
\begin{equation*}
 e^{\lambda_* t}(u-v)=\widetilde u-\widetilde v\leq0.
\end{equation*}
Thus $u\leq v$. If, moreover, both fields are bounded $(L_y,0)$-viscosity solutions,
interchanging their roles gives the reverse inequality, and therefore
uniqueness. Under (\ref{uniq-linear-data}), $L_f=L_y$ and $L_g=0$; hence
Proposition \ref{bound} and Theorem \ref{existence} give a bounded
$(L_y,0)$-viscosity solution, namely the lower value function. Under the Isaacs
condition, the lower and upper Hamiltonians and boundary operators coincide,
so $v$ and $u$ solve the same terminal-boundary problem. Uniqueness then
implies $v=u$.
\end{proof}

\begin{remark}
The finite-dimensional reduction above follows the approximation made in
\cite[Theorem~6.2]{Qiu22AAP}. However, the approaches differ in how uniform
parabolicity is obtained. In \cite{Qiu22AAP}, the random
coefficients and terminal data are first replaced by smooth cylindrical functions of finitely
many Brownian samples $W_{t_1},\ldots,W_{t_n}$. On each interval
$[t_{j-1},t_j)$, the samples from the preceding intervals are frozen as
parameters, whereas the current Brownian position $y=W_t$ is introduced as
an additional deterministic state variable. This converts the stochastic
equation into a deterministic parabolic HJBI equation in the joint variables
$(x,y)$.

Since the diffusion $\sigma(t)$ in \cite{Qiu22AAP} may be degenerate, the
probability space is enlarged by adding an independent $d$-dimensional Brownian
motion $B$, and the state equation is perturbed by $\delta_n\dd B_t$. The
second-order matrix of the resulting equation is
\begin{equation*}
 \mathbf A_{\delta_n}=\frac{1}{2}
 \begin{pmatrix}
  \sigma\sigma^*+\delta_n^2I_d&\sigma\\
  \sigma^*&I_m
 \end{pmatrix},
\end{equation*}
and hence, for $(\xi,\eta)\in\mathbb R^d\times\mathbb R^m$,
\begin{equation*}
 2\left\langle\mathbf A_{\delta_n}(\xi,\eta),(\xi,\eta)\right\rangle
 =|\sigma^*\xi+\eta|^2+\delta_n^2|\xi|^2.
\end{equation*}
Thus the approximating deterministic equation is uniformly elliptic for
every fixed $\delta_n>0$, although its ellipticity constant may vanish
as $\delta_n\downarrow0$. After solving these equations backward over the
cylindrical partition, the random comparison fields are reconstructed by
evaluating at $y=W_t$ and translating the spatial argument to
$x-\delta_nB_t$ (However, this construction cannot be adapted directly to
our interval $\overline D=[0,1]$: $x-\delta_nB_t$ may leave it,
and even after an extension the translated field generally does not satisfy the Neumann condition
at the fixed endpoints). The construction then adds and subtracts an error BSDE for
the cylindrical approximation. It also uses a correction of order $\delta_n$
driven by $|B_t|$ to control the errors created by this translation. The
resulting stability estimate is of order $\varepsilon^2+\delta_n^2$, so the original, possibly
degenerate problem is recovered by letting both the cylindrical error
$\varepsilon$ and the artificial-noise amplitude $\delta_n$ tend to zero. In
particular, $\delta_n\dd B_t$ is a regularizing device, not an assumption on
the original game.

In our present setting, only the common noise $W^0$ enters the random
coefficients, while the independent noise $W^1$ already acts in the original
state equation through $\sigma^1\dd W_t^1$. After adding the current
common-noise coordinate $y=W_t^0$, the joint second-order matrix is
\begin{equation*}
 \mathbf A_{\mathrm{aug}}=\frac{1}{2}
 \begin{pmatrix}
  |\sigma^0|^2+|\sigma^1|^2&\sigma^0\\
  \sigma^0&1
 \end{pmatrix},
\end{equation*}
with quadratic form
$|\sigma^0\xi+\eta|^2+|\sigma^1|^2\xi^2$. The assumption
$|\sigma^1|^2\geq\kappa$ therefore makes the augmented matrix uniformly
elliptic, with constants independent of the cylindrical approximation. Thus
no enlargement of the probability space, artificial diffusion, spatial
translation, or $\delta_n\downarrow0$ limit is needed. We use the cylindrical
approximation and deterministic reconstruction of \cite{Qiu22AAP}, while the
independent state noise provides uniform parabolicity directly.
\end{remark}

\begin{remark}[Possible higher-dimensional extension]
Let $D\subset\mathbb R^d$, let $W^0$ and $W^1$ take values in
$\mathbb R^{m_0}$ and $\mathbb R^{m_1}$, respectively, and let
$\sigma^0(t,x)\in\mathbb R^{d\times m_0}$ and
$\sigma^1(t,x)\in\mathbb R^{d\times m_1}$. The current common-noise
coordinate in the cylindrical approximation is then
$y\in\mathbb R^{m_0}$. In the joint variables $(x,y)$, the second-order
coefficient matrix of the augmented deterministic equation is
\begin{equation*}
 \mathbf A_{\mathrm{aug}}(t,x)=\frac{1}{2}
 \begin{pmatrix}
  \sigma^0(\sigma^0)^*+\sigma^1(\sigma^1)^*&\sigma^0\\
  (\sigma^0)^*&I_{m_0}
 \end{pmatrix}.
\end{equation*}
Indeed, for every $\xi\in\mathbb R^d$ and
$\eta\in\mathbb R^{m_0}$,
\begin{equation*}
 2\left\langle\mathbf A_{\mathrm{aug}}(\xi,\eta),(\xi,\eta)\right\rangle
 =\left|(\sigma^0)^*\xi+\eta\right|^2
  +\left|(\sigma^1)^*\xi\right|^2.
\end{equation*}
Consequently, a sufficient condition for uniform ellipticity is
\begin{equation*}
 \sigma^1(t,x)(\sigma^1(t,x))^*\geq\kappa I_d,
 \qquad \|\sigma^0(t,x)\|\leq K_0,
\end{equation*}
uniformly in $(t,x)$, where the matrix norm is defined by
$\|A\|:=\sup_{|z|=1}|Az|$. In fact,
\begin{equation*}
 2\mathbf A_{\mathrm{aug}}
 =\begin{pmatrix}I_d&\sigma^0\\0&I_{m_0}\end{pmatrix}
  \begin{pmatrix}\sigma^1(\sigma^1)^*&0\\0&I_{m_0}\end{pmatrix}
  \begin{pmatrix}I_d&0\\(\sigma^0)^*&I_{m_0}\end{pmatrix}.
\end{equation*}
The inverse of the first factor has an operator norm less than or equal to $1+K_0$.
Therefore, with $\mu:=\min\{\kappa,1\}$,
\begin{equation*}
 \left\langle\mathbf A_{\mathrm{aug}}(\xi,\eta),(\xi,\eta)\right\rangle
 \geq\frac{\mu}{2(1+K_0)^2}
 \bigl(|\xi|^2+|\eta|^2\bigr).
\end{equation*}
The matching uniform upper bound follows from the boundedness of
$\sigma^0$ and $\sigma^1$. This calculation also explains the role of the
independent noise: if $(\sigma^1)^*\xi=0$ for some nonzero $\xi$, then taking
$\eta=-(\sigma^0)^*\xi$ makes the quadratic form vanish. Thus the common
noise alone cannot make the augmented equation uniformly elliptic;
nondegeneracy of $\sigma^1$ in every state direction is the natural
condition and, in particular, requires $m_1\geq d$.

Under this condition, our deterministic-system argument extends to
$D\times\mathbb R^{m_0}$. On
$\partial D\times\mathbb R^{m_0}$ the augmented boundary direction is
$(\vec n,0)$ and remains uniformly oblique. The higher-dimensional parabolic
theory may require stronger smoothness and compatibility assumptions on the
boundary, coefficients and data.
\end{remark}

\begin{remark}[Role of our additional assumptions]
The restriction $N=0$ is used throughout the comparison proof. Since
$g(t,x,y,\theta,\gamma)=g_0(t,x)$, its Lipschitz constant in $y$ is zero.
Hence the nonlinear expectations contain no integral with respect to $K$ and
are independent of the state and the control parameters. This property is used in the stopping-time argument and in the construction of the smooth comparison fields. Let us emphasize that they reduce to those used in
\cite{Qiu22AAP}.

The other assumptions are also needed for the method we use. Since
$b,\sigma^0,\sigma^1$ do not depend on the controls and
$f=L_y y+f_0$ has no $z$-term, optimization over the controls changes only the
source term $F_0$. The resulting equation is therefore linear in the
derivatives of the unknown field. If the state coefficients depended on the
controls, the deterministic PDE would, in general, be fully nonlinear. The
condition $|\sigma^1|^2\geq\kappa$ gives uniform parabolicity of the augmented
equation, and the regularity assumptions allow us to apply classical Neumann
PDE theory. Finally, the form of $g_0$ allows the boundary condition to be made
homogeneous by a single lifting. These assumptions are used in our present
proof, but we do not claim that they are necessary for uniqueness.
\end{remark}

\appendix
\section{Further discussion of the domain assumptions}\label{sec:domain-discussion}
Previous studies of nonlinear Neumann boundary conditions often assume either
that the defining function $\phi_D$ belongs to $C_b^2(\mathbb{R}^d)$ or that
$D$ is convex. Both conditions imply Assumption \ref{HD}. This assumption also
allows certain nonsmooth and nonconvex domains.
\begin{proposition}\label{D}
Let $D$ be an open, connected, bounded subset of $\mathbb R^d$, and let
$\phi_D\in C_b^2(\mathbb R^d)$. Suppose
\begin{equation*}
D=\{x\in\mathbb R^d:\phi_D(x)>0\},\qquad
\partial D=\{x\in\mathbb R^d:\phi_D(x)=0\},\qquad
|D\phi_D(x)|=1\quad(x\in\partial D).
\end{equation*}
Then $D$ satisfies Assumption \ref{HD}.
\end{proposition}
\begin{proof}
By \cite{Wang}, $N_x=\{D\phi_D(x)\}$ for every $x\in\partial D$. It remains
to find $r_0>0$ such that $D\phi_D(x)\in N_{x,r}$ for every
$x\in\partial D$ and $0<r\leq r_0$.
We put
\begin{equation*}
 \lambda:=\sup\limits_{\xi\in\mathbb R^d}\max_{|y|=1}
 \left\langle D^2\phi_D(\xi)y,y\right\rangle<\infty,
\end{equation*}
and we choose $r_0>0$ arbitrarily if $\lambda\leq0$, and we choose it so that
$\lambda r_0<1$ if $\lambda>0$.
For arbitrary fixed $x\in\partial D$, we set $\vec n=D\phi_D(x)$, and let
$x_0=x-r_0\vec n$. For $x'\in B(x_0,r_0)$, let us write
$h=x'-x_0$ and $r=|h|<r_0$. Taylor's theorem along the line segment joining
$x$ and $x'$, together with $\phi_D(x)=0$, gives
\begin{equation*}
\begin{aligned}
&\phi_D(x')
\leq \left\langle \vec n,h-r_0\vec n\right\rangle
+\frac{\lambda}{2}|h-r_0\vec n|^2
={}-r_0+(1-\lambda r_0)\langle\vec n,h\rangle
+\frac{\lambda}{2}(r_0^2+r^2).
\end{aligned}
\end{equation*}
If $\lambda\leq0$, the Taylor estimate before expansion gives
\begin{equation*}
 \phi_D(x')\leq-r_0+\langle\vec n,h\rangle<0.
\end{equation*}
If $\lambda>0$, then $1-\lambda r_0>0$ and
$\langle\vec n,h\rangle\leq r$, so
\begin{equation*}
\begin{aligned}
&\phi_D(x')
\leq -r_0+(1-\lambda r_0)r+\frac{\lambda}{2}(r_0^2+r^2)
={}(r-r_0)\left(1+\frac{\lambda}{2}(r-r_0)\right)
<{}0.
\end{aligned}
\end{equation*}
Thus $x'\notin D$, and hence
$B(x-r_0D\phi_D(x),r_0)\cap D=\varnothing$. Moreover, for every
$0<r\leq r_0$,
\begin{equation*}
 B(x-rD\phi_D(x),r)\subset B(x-r_0D\phi_D(x),r_0),
\end{equation*}
so $D\phi_D(x)\in N_{x,r}$ for all $r\leq r_0$.
Next, we prove Assumption \ref{HD}(ii). For any $x\in\partial D$, set $l_x:=D\phi_D(x)$. Then, with $L=\|D^2\phi_D\|_{L^\infty(\mathbb R^d)}$, we have
\begin{equation*}
\begin{aligned}
\left\langle l_x,\vec n(y)\right\rangle
&=\left\langle D\phi_D(x),D\phi_D(y)\right\rangle\\
&=\left\langle D\phi_D(x),D\phi_D(y)-D\phi_D(x)\right\rangle+|D\phi_D(x)|^2\\
&\geq 1-L|y-x|,
\end{aligned}
\end{equation*}
for all $y\in B(x,\delta)\cap\partial D$. Choosing $\delta>0$ such that $1-L\delta>0$ and $\beta=\frac{1}{1-L\delta}$ verifies Assumption \ref{HD}(ii). For Assumption \ref{HD}(iii), choose $f_D=\phi_D$ and $\gamma_D=\frac{2}{3}r_0$.
\end{proof}
\begin{remark}\label{HDD}
\textnormal{(i)} The proof of Proposition \ref{D} also applies when
$\phi_D\in C_{\mathrm{Lip}}^1(\mathbb R^d)$. Let
$\chi\in C_c^\infty(\mathbb R^d)$ be equal to one on a neighborhood of the
compact set $\overline{D}$. For sufficiently small $\varepsilon>0$, choose
\begin{equation*}
 f_D=\chi\phi_D^\varepsilon,
 \qquad
 \phi_D^\varepsilon(x):=\int_{\mathbb R^d}
 \phi_D(x-\varepsilon y)\eta(y)\,\dd y,
\end{equation*}
where $\eta\in C_c^\infty(\mathbb R^d)$ is nonnegative, supported in the
unit ball, and has integral one. The\linebreak cutoff ensures that
$f_D\in C_b^2(\mathbb R^d)$, while $D f_D=D\phi_D^\varepsilon$ on $\partial D$.

Indeed, take $\gamma_D=\frac23r_0$. Since $D\phi_D$ is Lipschitz,
$D\phi_D^\varepsilon$ converges uniformly to $D\phi_D$. Thus, uniformly for
$x\in\partial D$,
\begin{equation*}
 \left\langle D f_D(x),D\phi_D(x)\right\rangle
 =\left\langle D\phi_D^\varepsilon(x),D\phi_D(x)\right\rangle
 \longrightarrow |D\phi_D(x)|^2=1.
\end{equation*}
Since $1$ lies in the interior of
$[\frac{\gamma_D}{r_0},\frac{2\gamma_D}{r_0}]=[\frac{2}{3},\frac{4}{3}]$, Assumption \ref{HD}(iii) follows for
$\varepsilon$ sufficiently small.

\textnormal{(ii)} If Assumption \ref{HD} holds, then for every
$x\in\partial D$, every $\vec n\in N_x$, and every
$x'\in\overline{D}$,
\begin{equation*}
 -\frac{1}{\gamma_D}\langle D f_D(x),\vec n\rangle|x-x'|^2
 +2\langle x-x',\vec n\rangle\leq0,
\end{equation*}
where $\gamma_D$ is the constant from Assumption \ref{HD}. Indeed, by
Assumption \ref{HD}(i), $\vec n\in N_{x,r_0}$. Put
$y=x-r_0\vec n$. Then $B(y,r_0)\cap D=\varnothing$ and hence
$|x'-y|\geq r_0$ for every $x'\in\overline{D}$. Furthermore,
\begin{equation*}
\begin{split}
&-\frac{1}{\gamma_D}\langle D f_D(x),\vec n\rangle|x-x'|^2
 +2\langle x-x',\vec n\rangle\\
\leq &-\frac1{r_0}|x-x'|^2+2\langle x-x',\vec n\rangle\\
={}&-\frac{1}{r_0}|x-y+y-x'|^2
 +2\left\langle x-y+y-x',\frac{x-y}{r_0}\right\rangle\\
={}&-\frac{1}{r_0}|x-y|^2-\frac{1}{r_0}|y-x'|^2
 +\frac{2}{r_0}|x-y|^2\leq-2r_0+2r_0=0.
\end{split}
\end{equation*}
\end{remark}
We finish with two examples showing that nonconvex or nonsmooth domains can
also satisfy Assumption \ref{HD}.
\begin{example}
Let
\begin{equation*}
 D=\left\{z\in\mathbb R^2:1<|z|^2<2\right\}.
\end{equation*}
This is a bounded, connected and nonconvex domain; for instance,
$(6/5,0)$ and $(-6/5,0)$ belong to $D$, whereas their midpoint does
not. We verify Assumption \ref{HD} directly.

Take $r_0=1/4$. On the inner boundary $|z|=1$, the inward unit normal is
$\vec n(z)=z$, whereas on the outer boundary $|z|=\sqrt2$, it is
$\vec n(z)=-z/\sqrt2$. The corresponding balls
\begin{equation*}
 B\left(z-r_0\vec n(z),r_0\right)
\end{equation*}
are contained in $\{|w|<1\}$ and $\{|w|>\sqrt2\}$, respectively.
Since both boundary components are smooth, their inward normals are unique;
hence $N_z=N_{z,r_0}\neq\varnothing$.

For Assumption \ref{HD}(ii), take $\delta=1/8$, $\beta=4$, and
$l_z=\vec n(z)$. Since $\sqrt2-1>\delta$, two boundary points at distance
less than $\delta$ lie on the same circle. Thus, for every
$z'\in B(z,\delta)\cap\partial D$,
\begin{equation*}
 \left\langle l_z,\vec n(z')\right\rangle
 =1-\frac{|z-z'|^2}{2|z|^2}
 >1-\frac1{128}>\frac14=\frac1\beta.
\end{equation*}

Finally, choose $\chi\in C_c^\infty(\mathbb R)$ such that
$\chi\equiv1$ on $[3/4,9/4]$, and set
\begin{equation*}
 f_D(z)=\chi(|z|^2)(|z|^2-1)(2-|z|^2).
\end{equation*}
Then $f_D\in C_b^\infty(\mathbb R^2)$ and
\begin{equation*}
 \left\langle D f_D(z),\vec n(z)\right\rangle
 =\begin{cases}
 2, & |z|=1,\\
 2\sqrt2, & |z|=\sqrt2.
 \end{cases}
\end{equation*}
Taking $\gamma_D=3/8$, we obtain the strict inequalities
\begin{equation*}
 \frac{\gamma_D}{r_0}=\frac32
 <\left\langle D f_D(z),\vec n(z)\right\rangle
 <3=\frac{2\gamma_D}{r_0},
\end{equation*}
which verify Assumption \ref{HD}(iii).
\end{example}

\begin{example}
Let
\begin{equation*}
 c_+=\left(\frac{1}{2},0\right),\qquad
 c_-=\left(-\frac{1}{2},0\right),\qquad
 D=B(c_+,1)\cap B(c_-,1).
\end{equation*}
The domain $D$ is bounded, open, connected and convex. Its boundary is
not differentiable at
\begin{equation*}
 p_\pm=\left(0,\pm\frac{\sqrt3}{2}\right),
\end{equation*}
where the two circular arcs meet with distinct tangent directions, so
$p_\pm$ are corner points of $\partial D$.

For a convex domain, every inward unit normal $\vec n\in N_z$ satisfies
\begin{equation*}
 \langle w-z,\vec n\rangle\geq0,
 \qquad w\in\overline{D}.
\end{equation*}
Consequently, for every $r>0$,
\begin{equation*}
 |w-(z-r\vec n)|^2
 =|w-z|^2+r^2+2r\langle w-z,\vec n\rangle
 \geq r^2.
\end{equation*}
Conversely, the metric-projection characterization for the closed convex
set $\overline{D}$ shows that every element of $N_z$ is such a supporting
normal. Thus $N_z=N_{z,r}$ for every $r>0$, and Assumption
\ref{HD}(i) holds with $r_0=1$.

Now we choose $\vartheta\in C_c^\infty(\mathbb R^2)$ such that
$\vartheta\equiv1$ on $B(0,1)$, and we put
\begin{equation*}
 f_D(z)=-\vartheta(z)|z|^2.
\end{equation*}
Since $\overline{D}\subset B(0,1)$, we have $D f_D(z)=-2z$ on
$\partial D$. On the regular arc contained in $\partial B(c_+,1)$,
the inward unit normal is $\vec n(z)=c_+-z$. This arc has $z_1\leq0$,
and
\begin{equation*}
 |z|^2-z_1=\frac34,\qquad
 \left\langle D f_D(z),\vec n(z)\right\rangle
 =|z|^2+\frac34\in\left[1,\frac32\right].
\end{equation*}
The same estimate holds on the other regular arc. At either corner
$p_\pm$, every $\vec n\in N_{p_\pm}$ is a unit vector in the cone
generated by $c_+-p_\pm$ and $c_--p_\pm$. It follows that
\begin{equation*}
 \frac32
 \leq\left\langle D f_D(p_\pm),\vec n\right\rangle
 \leq\sqrt3.
\end{equation*}
Therefore, for every $z\in\partial D$ and every $\vec n\in N_z$,
\begin{equation*}
 1\leq\left\langle D f_D(z),\vec n\right\rangle\leq\sqrt3.
\end{equation*}
Taking $\gamma_D=9/10$ and $r_0=1$ gives the strict bounds
\begin{equation*}
 \frac{\gamma_D}{r_0}=\frac9{10}
 <\left\langle D f_D(z),\vec n\right\rangle
 <\frac95=\frac{2\gamma_D}{r_0},
\end{equation*}
which verify Assumption \ref{HD}(iii).

It remains to verify Assumption \ref{HD}(ii). Take
\begin{equation*}
 \delta=\frac18,\qquad
 \beta=\frac4{\sqrt3},\qquad
 l_z=\frac{D f_D(z)}{|D f_D(z)|}.
\end{equation*}
The vector $l_z$ is well-defined on $\partial D$, since
$\frac{1}{2}\leq|z|\leq\sqrt3/2$. If
$z'\in B(z,\delta)\cap\partial D$ and $\vec n'\in N_{z'}$, then
\begin{equation*}
\begin{aligned}
&\left\langle D f_D(z),\vec n'\right\rangle
={}\left\langle D f_D(z'),\vec n'\right\rangle
+\left\langle D f_D(z)-D f_D(z'),\vec n'\right\rangle
\geq 1-2|z-z'|
>{}&\frac34.
\end{aligned}
\end{equation*}
Moreover, $|D f_D(z)|\leq\sqrt3$ on $\partial D$. Hence
\begin{equation*}
 \left\langle l_z,\vec n'\right\rangle
 >\frac{3}{4\sqrt3}
 =\frac{\sqrt3}{4}=\frac1\beta,
\end{equation*}
so Assumption \ref{HD}(ii) also holds with a strict inequality. This verifies
that Assumption \ref{HD} also covers bounded convex domains with genuine
corner points, where a globally smooth normal field is unavailable.
\end{example}

\section{Proof of the DPP (Theorem 4.2)}\label{app:DPP}

For stopping times $\tau,\hat\tau\in\mathcal T^0$ with $\tau\leq\hat\tau$ and controls $\theta_1\in\Theta_{\tau,\hat\tau}$ and $\theta_2\in\Theta_{\hat\tau,T}$, write
\begin{equation*}
	\theta_1\oplus\theta_2:=\theta_1\mathbf1_{[\tau,\hat\tau]}+\theta_2\mathbf1_{(\hat\tau,T]}.
\end{equation*}
We use the same notation for controls of player II.

For $\tau\in\mathcal T^0$ and $\xi\in L^2(\mathcal F_\tau;\overline D)$, put
\begin{equation*}
	\mathcal V(\tau,\xi):=\essinf_{\mu\in\mathcal M_{\tau,T}}\esssup_{\theta\in\Theta_{\tau,T}}Y_\tau^{\tau,\xi;\theta,\mu(\theta)}.
\end{equation*}

We first establish the following auxiliary result:
\begin{lemma}\label{app:stopping-value}
	Suppose Assumption \ref{HSDG} holds. Then, for every $\tau\in\mathcal T^0$ and $\xi\in L^2(\mathcal F_\tau;\overline D)$,
	\begin{equation*}
		\mathcal V(\tau,\xi)=V(\tau,\xi),\qquad P\text{-a.s.}
	\end{equation*}
\end{lemma}

\begin{proof}
	First, fix a deterministic time $t\in[0,T]$. For $x\in\overline D$, we have $\mathcal V(t,x)=V(t,x)$. Thus, for a simple random variable $\xi=\sum_{i=1}^N x_i\mathbf1_{A_i}$, where $(A_i)_{i=1}^N$ is an $\mathcal F_t$-measurable partition and $x_i\in\overline D$,
	\begin{equation*}
		\mathcal V(t,\xi)=\sum_{i=1}^N\mathbf1_{A_i}\mathcal V(t,x_i)=\sum_{i=1}^N\mathbf1_{A_i}V(t,x_i)=V(t,\xi).
	\end{equation*}
	For a general $\xi\in L^2(\mathcal F_t;\overline D)$, the same identity follows by approximation with simple random variables, using Lemma \ref{pro GRSDE}(ii) and Proposition \ref{bound}.
	
	Next, suppose that $\tau$ takes values in $\{t_1,\ldots,t_N\}$. Fix $x_0\in\overline D$. Set $A_i:=\{\tau=t_i\}\in\mathcal F_{t_i}$ and define $\xi_i:=\xi\mathbf1_{A_i}+x_0\mathbf1_{A_i^c}$. Then $\xi_i\in L^2(\mathcal F_{t_i};\overline D)$ and
	\begin{equation*}
		\mathcal V(\tau,\xi)=\sum_{i=1}^N\mathbf1_{A_i}\mathcal V(t_i,\xi_i)=\sum_{i=1}^N\mathbf1_{A_i}V(t_i,\xi_i)=V(\tau,\xi).
	\end{equation*}
	
	Finally, for a general stopping time $\tau$ and $\xi\in L^2(\mathcal F_\tau;\overline D)$, we set $h_n:=T2^{-n}$ and $\tau_n:=h_n\lceil\tau/h_n\rceil\wedge T$. Then $\tau_n\downarrow\tau$ and $0\leq\tau_n-\tau\leq h_n$. Since $\mathcal F_\tau\subset\mathcal F_{\tau_n}$, the preceding case gives $\mathcal V(\tau_n,\xi)=V(\tau_n,\xi)$. The estimates in Lemma \ref{pro RSDE} and \ref{pro RSDE} also hold on stochastic intervals. Thus, taking conditional expectations in the GBSDE over $[\tau,\tau_n]$ and applying Lemmas \ref{pro RSDE}(i), (iv) and \ref{pro GRSDE}(i), (ii) gives, uniformly w.r.t the controls $\theta,\gamma$,
	\begingroup
	\allowdisplaybreaks[4]
	\begin{align}
		&\left|Y_\tau^{\tau,\xi;\theta,\gamma}-E\left[Y_{\tau_n}^{\tau_n,\xi;\theta,\gamma}\mid\mathcal F_\tau\right]\right|\nonumber\\
		&\leq\left|Y_\tau^{\tau,\xi;\theta,\gamma}-E\left[Y_{\tau_n}^{\tau,\xi;\theta,\gamma}\mid\mathcal F_\tau\right]\right|+E\left[\left|Y_{\tau_n}^{\tau,\xi;\theta,\gamma}-Y_{\tau_n}^{\tau_n,\xi;\theta,\gamma}\right|\mid\mathcal F_\tau\right]\nonumber\\
		&\leq E\left[\int_\tau^{\tau_n}\left|f\left(r,X_r^{\tau,\xi;\theta,\gamma},Y_r^{\tau,\xi;\theta,\gamma},Z_r^{\tau,\xi;\theta,\gamma},\theta_r,\gamma_r\right)\right|\dd r\mid\mathcal F_\tau\right]\nonumber\\
		&\quad+E\left[\int_\tau^{\tau_n}\left|g\left(r,X_r^{\tau,\xi;\theta,\gamma},Y_r^{\tau,\xi;\theta,\gamma},\theta_r,\gamma_r\right)\right|\dd K_r^{\tau,\xi;\theta,\gamma}\mid\mathcal F_\tau\right]\nonumber\\
		&\quad+E\left[\left|Y_{\tau_n}^{\tau_n,X_{\tau_n}^{\tau,\xi;\theta,\gamma};\theta,\gamma}-Y_{\tau_n}^{\tau_n,\xi;\theta,\gamma}\right|\mid\mathcal F_\tau\right]\nonumber\\
		&\leq CE\left[\int_\tau^{\tau_n}\left(1+|Y_r^{\tau,\xi;\theta,\gamma}|+|Z_r^{\tau,\xi;\theta,\gamma}|\right)\dd r\mid\mathcal F_\tau\right]\nonumber\\
		&\quad+CE\left[\left(1+\sup\limits_{\tau\leq r\leq\tau_n}|Y_r^{\tau,\xi;\theta,\gamma}|\right)K_{\tau_n}^{\tau,\xi;\theta,\gamma}\mid\mathcal F_\tau\right]+CE\left[|X_{\tau_n}^{\tau,\xi;\theta,\gamma}-\xi|^{\frac12}\mid\mathcal F_\tau\right]\nonumber\\
		&\leq C(h_n+h_n^{\frac12})+C\left(E\left[|X_{\tau_n}^{\tau,\xi;\theta,\gamma}-\xi|^2\mid\mathcal F_\tau\right]\right)^{\frac14}\nonumber\\
		&\leq C(h_n+h_n^{\frac12}+h_n^{\frac14})\leq Ch_n^{\frac14},\qquad P\text{-a.s.}\label{app:short-time-payoff}
	\end{align}
	\endgroup
	Let us prove that
	\begin{equation}\label{app:value-approximation}
		\left|\mathcal V(\tau,\xi)-E[V(\tau_n,\xi)\mid\mathcal F_\tau]\right|\leq Ch_n^{\frac14}.
	\end{equation}
	
	We fix arbitrarily $\varepsilon>0$. Using the definition of $\mathcal V(\tau_n,\xi)$ combined with standard arguments allows to choose $\nu_n^\varepsilon\in\mathcal M_{\tau_n,T}$ such that
	\begin{equation*}
		\esssup_{\vartheta\in\Theta_{\tau_n,T}}Y_{\tau_n}^{\tau_n,\xi;\vartheta,\nu_n^\varepsilon(\vartheta)}\leq\mathcal V(\tau_n,\xi)+\varepsilon=V(\tau_n,\xi)+\varepsilon,\qquad P\text{-a.s.}
	\end{equation*}
	We fix any $\gamma^0\in\Gamma_{\tau,\tau_n}$ and define
	\begin{equation*}
		\mu_n^\varepsilon(\theta):=\gamma^0\oplus\nu_n^\varepsilon(\theta|_{[\tau_n,T]}),\qquad\theta\in\Theta_{\tau,T}.
	\end{equation*}
	Then $\mu_n^\varepsilon\in\mathcal M_{\tau,T}$. Applying (\ref{app:short-time-payoff}), we obtain, for every $\theta\in\Theta_{\tau,T}$ and $\vartheta:=\theta|_{[\tau_n,T]}$,
	\begin{equation*}
		\begin{aligned}
			Y_\tau^{\tau,\xi;\theta,\mu_n^\varepsilon(\theta)}&\leq E\left[Y_{\tau_n}^{\tau_n,\xi;\vartheta,\nu_n^\varepsilon(\vartheta)}\mid\mathcal F_\tau\right]+Ch_n^{\frac14}\\
			&\leq E[V(\tau_n,\xi)\mid\mathcal F_\tau]+\varepsilon+Ch_n^{\frac14}.
		\end{aligned}
	\end{equation*}
	Thus, taking the essential supremum over $\theta$ yields
	\begin{equation*}
		\mathcal V(\tau,\xi)\leq E[V(\tau_n,\xi)\mid\mathcal F_\tau]+\varepsilon+Ch_n^{\frac14}.
	\end{equation*}
	For the lower bound, we fix any $\mu\in\mathcal M_{\tau,T}$ and $\theta^0\in\Theta_{\tau,\tau_n}$. The induced continuation strategy
	\begin{equation*}
		\mu_n(\vartheta):=\left.\mu(\theta^0\oplus\vartheta)\right|_{[\tau_n,T]},\qquad\vartheta\in\Theta_{\tau_n,T},
	\end{equation*}
	belongs to $\mathcal M_{\tau_n,T}$ by nonanticipativity. Observe that
	\begin{equation*}
		V(\tau_n,\xi)=\mathcal V(\tau_n,\xi)\leq\esssup_{\vartheta\in\Theta_{\tau_n,T}}Y_{\tau_n}^{\tau_n,\xi;\vartheta,\mu_n(\vartheta)},\qquad P\text{-a.s.}
	\end{equation*}
	Hence, for our given $\varepsilon>0$, standard arguments allow us to construct $\vartheta_n^\varepsilon\in\Theta_{\tau_n,T}$ such that
	\begin{equation*}
		Y_{\tau_n}^{\tau_n,\xi;\vartheta_n^\varepsilon,\mu_n(\vartheta_n^\varepsilon)}\geq\mathcal V(\tau_n,\xi)-\varepsilon=V(\tau_n,\xi)-\varepsilon,\qquad P\text{-a.s.}
	\end{equation*}
	Then, putting $\theta_n^\varepsilon:=\theta^0\oplus\vartheta_n^\varepsilon$, and applying (\ref{app:short-time-payoff}) we obtain
	\begin{equation*}
		\begin{aligned}
			\esssup_{\theta\in\Theta_{\tau,T}}Y_\tau^{\tau,\xi;\theta,\mu(\theta)}&\geq Y_\tau^{\tau,\xi;\theta_n^\varepsilon,\mu(\theta_n^\varepsilon)}
			\geq E\left[Y_{\tau_n}^{\tau_n,\xi;\vartheta_n^\varepsilon,\mu_n(\vartheta_n^\varepsilon)}\mid\mathcal F_\tau\right]-Ch_n^{\frac14}\\
			&\geq E[V(\tau_n,\xi)\mid\mathcal F_\tau]-\varepsilon-Ch_n^{\frac14},\qquad P\text{-a.s.}
		\end{aligned}
	\end{equation*}
	Consequently, taking the essential infimum over $\mu$ gives the lower bound. Letting $\varepsilon\downarrow0$ in both bounds proves (\ref{app:value-approximation}).
	
	By Proposition \ref{bound} and Theorem \ref{continuity}, $V$ is bounded and continuous. Moreover, $V(\tau,\xi)$ is $\mathcal F_\tau$-measurable. Hence, (\ref{app:value-approximation}) and dominated convergence give
	\begin{equation*}
		\begin{aligned}
			E\left[|\mathcal V(\tau,\xi)-V(\tau,\xi)|\right]&\leq Ch_n^{\frac14}+E\left[|V(\tau_n,\xi)-V(\tau,\xi)|\right]
			\rightarrow0,\quad n\rightarrow\infty.
		\end{aligned}
	\end{equation*}
	This proves the lemma.
\end{proof}

Now we are ready to give the following proof.
\begin{proof}[Proof of Theorem \ref{DPP}]
	We prove the assertion for $V$, that, for $U$ is analogous. We use the restriction and concatenation argument of \cite{Li08sicon}. Fix $\tau,\hat\tau\in\mathcal T^0$ with $\tau\leq\hat\tau$ and $\xi\in L^2(\mathcal F_\tau;\overline D)$. Then uniqueness for our GBSDE gives
	\begin{equation*}
		Y_\tau^{\tau,\xi;\theta,\gamma}=G_{\tau,\hat\tau}^{\tau,\xi;\theta,\gamma}[Y_{\hat\tau}^{\tau,\xi;\theta,\gamma}],\qquad Y_{\hat\tau}^{\tau,\xi;\theta,\gamma}=Y_{\hat\tau}^{\hat\tau,X_{\hat\tau}^{\tau,\xi;\theta,\gamma};\theta,\gamma}.
	\end{equation*}
	For $\eta,\eta'\in L^\infty(\mathcal F_{\hat\tau})$, the two GBSDEs defining $G_{\tau,\hat\tau}^{\tau,\xi;\theta,\gamma}[\eta]$ and $G_{\tau,\hat\tau}^{\tau,\xi;\theta,\gamma}[\eta']$ have the same increasing process $K$ defined by the RSDE (\ref{eq RSDE}). Thus, the standard linearization argument (see the proof of \cite[Theorem~1.4]{Pardoux and Zhang}), (\ref{*4}), and the conditional Cauchy--Schwarz inequality yield
	\begin{equation}\label{app:semigroup-stability}
		\left|G_{\tau,\hat\tau}^{\tau,\xi;\theta,\gamma}[\eta]-G_{\tau,\hat\tau}^{\tau,\xi;\theta,\gamma}[\eta']\right|\leq C_*\left(E[|\eta-\eta'|^2\mid\mathcal F_\tau]\right)^{\frac12},
	\end{equation}
	uniformly in the controls. Put
	\begin{equation*}
		\begin{aligned}
			I(\theta_1,\gamma_1)&:=G_{\tau,\hat\tau}^{\tau,\xi;\theta_1,\gamma_1}[V(\hat\tau,X_{\hat\tau}^{\tau,\xi;\theta_1,\gamma_1})],\ \theta_1\in\Theta_{\tau,\hat\tau},\  \gamma_1\in\Gamma_{\tau,\hat\tau},\\
			I_{\mu_1}&:=\esssup_{\theta_1\in\Theta_{\tau,\hat\tau}}I(\theta_1,\mu_1(\theta_1)),\qquad V_{\hat\tau}(\tau,\xi):=\essinf_{\mu_1\in\mathcal M_{\tau,\hat\tau}}I_{\mu_1}.
		\end{aligned}
	\end{equation*}
	
	We first prove $V_{\hat\tau}(\tau,\xi)\leq V(\tau,\xi)$. Fix $\mu\in\mathcal M_{\tau,T}$. For any $\theta_2\in\Theta_{\hat\tau,T}$, define
	\begin{equation*}
		\mu_1(\theta_1):=\left.\mu(\theta_1\oplus\theta_2)\right|_{[\tau,\hat\tau]},\qquad\theta_1\in\Theta_{\tau,\hat\tau}.
	\end{equation*}
	By nonanticipativity, $\mu_1\in\mathcal M_{\tau,\hat\tau}$ and its definition does not depend on the choice of $\theta_2$. Choose $(\theta_1^i)\subset\Theta_{\tau,\hat\tau}$ with $I_{\mu_1}=\sup_iI(\theta_1^i,\mu_1(\theta_1^i))$, $P$-a.s. For $\varepsilon>0$, put
	\begin{equation*}
		A_i:=\{I_{\mu_1}\leq I(\theta_1^i,\mu_1(\theta_1^i))+\varepsilon\}\setminus\bigcup_{j<i}A_j,\qquad\theta_1^\varepsilon:=\sum_i\mathbf1_{A_i}\theta_1^i.
	\end{equation*}
	Then, as $(A_i)$ is an $\mathcal F_\tau$-measurable partition, $\theta_1^\varepsilon\in\Theta_{\tau,\hat\tau}$, and we have
	\begin{equation}\label{app:first-selection}
		V_{\hat\tau}(\tau,\xi)\leq I_{\mu_1}\leq G_{\tau,\hat\tau}^{\tau,\xi;\theta_1^\varepsilon,\mu_1(\theta_1^\varepsilon)}[V(\hat\tau,X_{\hat\tau}^{\tau,\xi;\theta_1^\varepsilon,\mu_1(\theta_1^\varepsilon)})]+\varepsilon.
	\end{equation}
	For this $\theta_1^\varepsilon$, we define the continuation strategy
	\begin{equation*}
		\mu_2(\theta_2):=\left.\mu(\theta_1^\varepsilon\oplus\theta_2)\right|_{[\hat\tau,T]},\qquad\theta_2\in\Theta_{\hat\tau,T}.
	\end{equation*}
	Again, $\mu_2\in\mathcal M_{\hat\tau,T}$. Thanks to Lemma \ref{app:stopping-value}
	\begin{equation*}
		V(\hat\tau,X_{\hat\tau}^{\tau,\xi;\theta_1^\varepsilon,\mu_1(\theta_1^\varepsilon)})\leq\esssup_{\theta_2\in\Theta_{\hat\tau,T}}Y_{\hat\tau}^{\hat\tau,X_{\hat\tau}^{\tau,\xi;\theta_1^\varepsilon,\mu_1(\theta_1^\varepsilon)};\theta_2,\mu_2(\theta_2)}.
	\end{equation*}
	A similar construction on a countable $\mathcal F_{\hat\tau}$-measurable partition gives $\theta_2^\varepsilon\in\Theta_{\hat\tau,T}$ such that
	\begin{equation*}
		V(\hat\tau,X_{\hat\tau}^{\tau,\xi;\theta_1^\varepsilon,\mu_1(\theta_1^\varepsilon)})\leq Y_{\hat\tau}^{\hat\tau,X_{\hat\tau}^{\tau,\xi;\theta_1^\varepsilon,\mu_1(\theta_1^\varepsilon)};\theta_2^\varepsilon,\mu_2(\theta_2^\varepsilon)}+\varepsilon=Y_{\hat\tau}^{\tau,\xi;\theta^\varepsilon,\mu(\theta^\varepsilon)}+\varepsilon,
	\end{equation*}
	where $\theta^\varepsilon:=\theta_1^\varepsilon\oplus\theta_2^\varepsilon$. Here we have used that $\mu(\theta^\varepsilon)=\mu_1(\theta_1^\varepsilon)\oplus\mu_2(\theta_2^\varepsilon)$. By the comparison theorem, (\ref{app:semigroup-stability}), (\ref{app:first-selection}), and the flow property,
	\begin{equation*}
		\begin{aligned}
			V_{\hat\tau}(\tau,\xi)&\leq G_{\tau,\hat\tau}^{\tau,\xi;\theta^\varepsilon,\mu(\theta^\varepsilon)}[Y_{\hat\tau}^{\tau,\xi;\theta^\varepsilon,\mu(\theta^\varepsilon)}]+(1+C_*)\varepsilon\\
			&=Y_\tau^{\tau,\xi;\theta^\varepsilon,\mu(\theta^\varepsilon)}+(1+C_*)\varepsilon\\
			&\leq\esssup_{\theta\in\Theta_{\tau,T}}Y_\tau^{\tau,\xi;\theta,\mu(\theta)}+(1+C_*)\varepsilon.
		\end{aligned}
	\end{equation*}
	Hence, taking the essential infimum over $\mu$ and letting $\varepsilon\downarrow0$ we obtain the first inequality.
	
	We next prove $V(\tau,\xi)\leq V_{\hat\tau}(\tau,\xi)$. Analogously to above, by combining strategies on the sets of a suitable $\mathcal F_\tau$-measurable partition, we can choose $\mu_1^\varepsilon\in\mathcal M_{\tau,\hat\tau}$ such that, for every $\theta_1\in\Theta_{\tau,\hat\tau}$,
	\begin{equation}\label{app:second-selection}
		V_{\hat\tau}(\tau,\xi)\geq G_{\tau,\hat\tau}^{\tau,\xi;\theta_1,\mu_1^\varepsilon(\theta_1)}[V(\hat\tau,X_{\hat\tau}^{\tau,\xi;\theta_1,\mu_1^\varepsilon(\theta_1)})]-\varepsilon.
	\end{equation}
	We now take a Borel partition $(D_i)_{i\geq1}$ of $\overline D$, with $x_i\in D_i$ and $\operatorname{diam}(D_i)\leq\delta$, $i\geq 1$. For each $i$, Lemma \ref{app:stopping-value} and the same construction on a suitable $\mathcal F_{\hat\tau}$-measurable partition give a strategy $\nu_i^\varepsilon\in\mathcal M_{\hat\tau,T}$ satisfying
	\begin{equation*}
		\esssup_{\theta_2\in\Theta_{\hat\tau,T}}Y_{\hat\tau}^{\hat\tau,x_i;\theta_2,\nu_i^\varepsilon(\theta_2)}\leq V(\hat\tau,x_i)+\varepsilon.
	\end{equation*}
	For $\theta=\theta_1\oplus\theta_2\in\Theta_{\tau,T}$, we put $\zeta^{\theta_1}:=X_{\hat\tau}^{\tau,\xi;\theta_1,\mu_1^\varepsilon(\theta_1)}$ and define
	\begin{equation*}
		\mu^\varepsilon(\theta):=\mu_1^\varepsilon(\theta_1)\oplus\sum_i\mathbf1_{\{\zeta^{\theta_1}\in D_i\}}\nu_i^\varepsilon(\theta_2).
	\end{equation*}
	This strategy belongs to $\mathcal M_{\tau,T}$.
	We choose $\delta$ so that $2C\delta^{1/2}\leq\varepsilon$ for $C$ the constant appearing below. Standard spatial estimates and the flow property give, for every $\theta=\theta_1\oplus\theta_2$,
	\begin{equation*}
		\begin{aligned}
			Y_{\hat\tau}^{\tau,\xi;\theta,\mu^\varepsilon(\theta)}&\leq\sum_i\mathbf1_{\{\zeta^{\theta_1}\in D_i\}}\bigl(V(\hat\tau,x_i)+\varepsilon+C\delta^{\frac12}\bigr)\\
			&\leq V(\hat\tau,\zeta^{\theta_1})+2\varepsilon.
		\end{aligned}
	\end{equation*}
	As further up, the comparison theorem, (\ref{app:semigroup-stability}), and (\ref{app:second-selection}) yield
	\begin{equation*}
		V_{\hat\tau}(\tau,\xi)\geq Y_\tau^{\tau,\xi;\theta,\mu^\varepsilon(\theta)}-(1+2C_*)\varepsilon.
	\end{equation*}
	Taking the essential supremum over $\theta$ gives $V_{\hat\tau}(\tau,\xi)\geq V(\tau,\xi)-(1+2C_*)\varepsilon$. Finally, letting $\varepsilon\downarrow0$ completes the proof for $V$.
\end{proof}

\end{document}